\documentclass[11pt]{article}
\usepackage{amsmath, amsthm, amssymb, amsfonts, bm}
\usepackage{graphicx}
\usepackage{mathabx}
\usepackage{subfig}
\usepackage{enumerate}
\usepackage{xcolor}
\usepackage{cite}
\usepackage{hyperref}
\usepackage{array, boldline, makecell, booktabs}
\usepackage{multirow}
\usepackage{caption}
\usepackage[bottom]{footmisc}

\usepackage{tikz}
\usetikzlibrary{positioning}

\makeatletter
\newtheorem{theorem}{Theorem}
\newtheorem{lemma}[theorem]{Lemma}

\newtheorem{definition}[theorem]{Definition}

\allowdisplaybreaks \numberwithin{equation}{section}
\title{A deterministic-statistical approach to reconstruct moving sources using sparse partial data}
\author{Yanfang Liu\thanks{Michigan Technological University, Houghton, MI, USA. {\it yanfangl@mtu.edu} },\ 
Yukun Guo\thanks{Harbin Institute of Technology, Harbin, P. R. China, {\it ykguo@hit.edu.cn}}\quad
and Jiguang Sun\thanks{Michigan Technological University, Houghton, MI, USA. {\it jiguangs@mtu.edu} }
}
\date{}
\begin{document}
\maketitle

 \captionsetup[figure]{name={Fig },labelsep=period}

\begin{abstract}
    We consider the reconstruction of moving sources using partial measured data. A two-step deterministic-statistical approach is proposed. In the first step, an approximate direct sampling method is developed to obtain the locations of the sources at different times. Such information is coded in the priors, which is critical for the success of the Bayesian method in the second step. The well-posedness of the posterior measure is analyzed in the sense of the Hellinger distance. Both steps are based on the same physical model and use the same set of measured data. The combined approach inherits the merits of the deterministic method and Bayesian inversion as demonstrated by the numerical examples. 
    \\
    
\end{abstract}

\section{Introduction}

The detection and identification of moving targets using waves have many important applications such as radar, underwater acoustics, and through the wall imaging \cite{Baum1999, Gilbert2004, Cheney2009, LiEtal2012}. 
 In this paper, we consider the reconstruction of moving acoustic point sources, which can be mathematically formulated as a time domain inverse source problem. Several methods were proposed recently from the inverse problem community for the time-domain inverse problems of wave equations, e.g., the algebraic method, the time-reversal techniques, the sampling-type methods and the method of fundamental solutions \cite{NIO12, Ohe20,GMC13, GF15, CGMS20}. See also \cite{HKLZ19} for a uniqueness result for the inverse moving source problem. In practice, the measured data are usually sparse, partial and compromised by noises. Classical methods such as the linear sampling method and the direct sampling method might not provide satisfactory results. In general, the performance deteriorates significantly as the data become fewer \cite{LiuSun2019JCP,Guo2016AA}.

In recent years, the Bayesian method became popular for inverse problems \cite{JariE.Somersalo2006, Stuart2010}. For statistical inference, the parameters in the physical model are viewed as random variables. 
Using the Bayes' formula, a posterior distribution of the unknown parameter can be formally derived. The main task is then to explore the posterior density function numerically.
Statistical estimates such as the MAP (maximum a posteriori
estimate) or CM (conditional mean) are used to characterize the unknown parameter.
For some recent applications of the Bayesian method for inverse problems, we refer readers to \cite{LiuEtal2019, Kaipio2019, Yang2020IP} and the references therein.

Motivated by the detection of the moving object behind the wall using ultrawideband (UWB) radar \cite{LiEtal2012, Wang2018}, the current paper considers a time-domain inverse problem to determine the moving path of a radiating source with partial data. We proposed a two-step deterministic-statistical approach. In the first step, we develop an approximate direct sampling method (ADSM) to obtain the rough locations of the sources at different times. Since only partial data are available, the locations are usually not accurate compared to the case of full data. Nonetheless, the ADSM does provide useful information of the unknown sources. In the second step, a Bayesian method is used to obtain more detailed information. The priors, which are built on the information obtained in the first step, play a key role for the success of the Bayesian inversion. The idea of utilizing a combined deterministic-statistical approach was first employed to treat the inverse acoustic scattering problem to reconstruct an obstacle with limited-aperture data in \cite{LDS20}. Later, the approach was used to determine both the locations and strengths of static multipolar sources from partial radiated acoustic fields \cite{LLSX20}. For recent developments on path reconstructions of moving point sources, we refer the readers to \cite{GMC13, GHHLL16, CGMS20, NIO12, Wang2018, WGLL17, WGLL19}. 
%  and the inverse elastic scattering problems with partial data \cite{LiSunXu2020}

Inverse problems with partial measurements arise from many applications in science and engineering.
Compared with the case of full measurements, it is more challenging in general. The reconstruction deteriorates when the data are less or the observation aperture becomes smaller. In this paper, for partial data, we propose a deterministic-statistical approach to reconstruct the moving paths of acoustic sources. The contributions are as follows: 1) an approximate direct sampling method is developed to obtain rough locations of moving acoustic sources; 2) an MCMC (Markov chain Monte Carlo) algorithm is employed to improve/construct the moving paths of the sources using the information obtained by the ADSM; and 3) the well-posedness of the posterior measure is analyzed under the Hellinger distance. Both steps, ADSM and MCMC, are based on the same physical model and use the same set of measured data. The combined approach inherits the merits of the deterministic method and Bayesian inversion as demonstrated by the numerical examples.

The rest of the paper is organized as follows. In Section \ref{sec:model_problem}, the direct and inverse problems are introduced. An approximate direct sampling method is developed in Section \ref{sec:direct_sampling}. The Bayesian inversion scheme is presented in Section \ref{sec:bayesian_inversion}. Section \ref{sec:numerical_results} contains several validating numerical examples. Finally, we draw some conclusions in Section \ref{sec:conclusion}.

\section{Direct and Inverse Source Problems} \label{sec:model_problem}

We begin with the mathematical model governing the wave source radiation. Denote by $\Omega \subset \mathbb{R}^3$ a simply-connected bounded domain. Consider the wave propagation due to a point excitation in a homogeneous and isotropic media. The speed of the wave in the background media is denoted by the constant $c$. Let $d_{\Omega}:=\sup_{x,y\in \Omega} \|x-y\|$ and $T >2d_{\Omega}/c$ be the terminal time for measuring the radiated data. 

For a single point source, the radiated wave field $u$ satisfies the following wave equation
\begin{subequations}\label{wave equation}
\begin{equation}\label{eq:wave_equation}
c^{-2}\partial_{t t}u(x,t)- \Delta u(x,t)=\lambda (t)\delta(x-z(t)) \quad \text{in}\; \mathbb{R}^3 \times (0,T],
\end{equation}
with initial conditions
\begin{equation}\label{initial condition}
u(x,0)= \partial_{t} u(x,0)= 0 \quad \text{in}\; \mathbb{R}^3, 
\end{equation}
\end{subequations}
where $\delta$ denotes the Dirac delta distribution, $\lambda:\mathbb{R}\to \mathbb{R}$ is the temporal pulse (magnitude of the impulsive load) and $\lambda(t)=0$ for $t<0$, $z: \mathbb{R}_{+}\rightarrow \Omega$ is the trajectory of the moving point source. We refer to \cite{Ohe20} for a regularity result on the solution of problem \eqref{wave equation} in some spatial-temporal Sobolev spaces.
 
Throughout the paper, we assume that $z \in \mathcal{C}^1[0, T]$. Furthermore, we assume that the point source moves slowly compared with the speed of the wave $c$, that is, its velocity $v:=\mathrm{d}z/\mathrm{d}t$ satisfies 
\begin{equation}\label{vtllc}
\|v(t) \|\ll c, \quad \forall t\in[0, T].
\end{equation}
This is the case for many radar applications such as through-wall imaging radars where $c$ is the speed of electromagnetic waves and $v(t)$ is the speed of the visually obscured target, e.g., a human being inside the building \cite{LiEtal2012, Wang2018}.
 
The direct problem can be described as follows: given $z(t)$ and $\lambda(t)$, find the radiating field $u$ satisfying \eqref{wave equation}. It is well known that the solution to this problem is given explicitly by the Li\'{e}nard–Wiechert retarded potential (see, e.g., \cite{Jackson2007}):
\begin{equation}\label{exact solution 1}
    u_z(x,t) = \dfrac{\lambda(\tau)}{4 \pi \|x-z(\tau)\|\left(1-\frac{v(\tau) \cdot (x-z(\tau))}{c\|x-z(\tau)\|}\right)},\quad t\in(0,T],
\end{equation}
where the retarded time $\tau$ is the solution of the equation
$$ 
t = \tau+ c^{-1}{\|x-z(\tau)\|}.
$$
When $z(t)\equiv z_0$ is a static point source, the solution to \eqref{wave equation} is simply
\begin{equation}\label{exact solution 2}
    u_{z_0}(x,t) = \dfrac{\lambda(t- c^{-1}\|x-z_0\|)}{4\pi \|x-z_0\|}.
\end{equation}
For a fixed small enough $T$, since $\|v(t)\| \ll c, t \in [0, T]$, it holds that 
\[
|t-\tau| \ll 1 \quad \text{and} \quad \left|\frac{v(\tau) \cdot (x-z(\tau))}{c\|x-z(\tau)\|} \right| \ll 1.
\]
Let $U_T:=\Gamma \times [0, T]$ and define $\|u(x,t)\|_{U_T}^2 := \int_0^T \int_\Gamma |u(x,t)|^2 {\rm d}x {\rm d}t$. We have that
\begin{equation}\label{Duzuz0}
\|u_z(x,t)-u_{z_0}(x,t)\|^2_{U_T} := \int_0^T \int_\Gamma |u_z(x,t)-u_{z_0}(x,t)|^2 {\rm d}x {\rm d}t \ll 1.   
\end{equation}
% The inverse problem is motivated by the detection of the moving object behind the wall using ultrawideband (UWB) radar \cite{Wang2018}. The radar sends many short wavelets in a period of time $[0,T]$ and the received radar signals are utilized to identify the path of the target. 

In this paper, we consider the time domain inverse source problem (TISP) of recovering the trajectory $z(t)$ of a moving point source for \eqref{wave equation} from the measurement data $u(x,t)|_{\Gamma \times (0,T]}$, where $\Gamma\subset\mathbb{R}^3\backslash\overline{\Omega}$ is a measurement surface (see Fig~\ref{fig:illustration} for a geometric illustration of the problem). In particular, we are interested in the case of partial data, e.g., when $\Gamma$ is a fraction of a sphere. 
The temporal pulse $\lambda(t)\in \mathcal{C}^1[0, T]$ is assumed to be a periodic function with period $p\in\mathbb{R}_{+}$. 
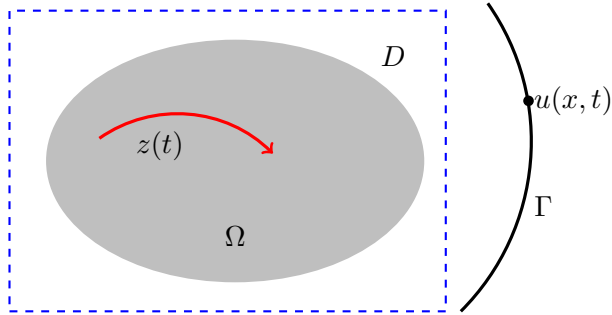
\begin{figure}[h!]
	\centering
	\begin{tikzpicture}[thick]
	\pgfmathsetseed{3}
	\filldraw [lightgray] (0, 0) ellipse [x radius=2.5cm, y radius=1.6cm] ;
	\draw node at (0, -1) {$\Omega$};
	\draw node at (2.1, 1.4) {$D$};
	\draw[<-] (0.5, 0.1) arc(45:125:1.8cm) [very thick, red];
	\draw node at (-1, 0.2) {$z(t)$};
	\draw (3, -2) arc(-45:35:3.2cm) [very thick, black]; 
	\draw node at (4.5,0.8) {$u(x,t)$};
	\draw node at (4.1,-0.6) {$\Gamma$};
    \fill [black] (3.9,0.8) circle (2pt);
    \draw [blue, dashed] (-3, -2) rectangle (2.8, 2); % sampling domain D
	\end{tikzpicture}
	\caption{An illustration of the time-dependent inverse source problem.} 
\label{fig:illustration}
\end{figure}

\section{Direct Sampling Method} \label{sec:direct_sampling}

In this section we shall develop an approximate direct sampling method (ADSM) for the inverse problem {TISP} to recover the path $z(t)$ from the measurement data $u(x,t)|_{\Gamma \times (0,T]}$. The sampling-type methods share the basic idea of identifying the unknown target by constructing some indicator functions over the probing/sampling region \cite{CK19}. Once the indicators are evaluated, the geometrical information (e.g., location and shape) of the target could be recovered using a point-wise criterion to determine whether the sampling point lies inside or outside the target. 

As one of the sampling-type methods, the direct sampling method became popular recently due to several compelling features. The direct sampling method is fast and easy to implement. Moreover, they are robust to noise-contaminated measurements in general. However, the direct sampling methods usually require a large amount of measurements and are inherently qualitative. Nonetheless, the direct sampling methods can provide useful information even for partial measurement data. 

%Exemplary sampling-type approaches for typical time-harmonic inverse scattering problems involve the linear sampling method, the factorization method, the direct sampling/orthogonal sampling methods and the extended sampling method \cite{LS18}, to name a few. In particular, we refer to \cite{ZGLL19} for a recent study on the direct sampling method for recovering static multipolar point sources. Recently, the sampling-type methods have been applied to time domain inverse scattering problems as well in various scenarios, for instance, see \cite{GMC13, GHHLL16, CMG17} and the references therein. Some efforts have also been made to reconstruct the trajectory of a moving point source using the direct sampling approaches \cite{WGLL17, WGLL19, CGMS20}. 

% On the other hand, the Bayesian approaches are data-driven and hence have the capability of compensating the lack of measured data and produce rather accurate reconstructions in the statistical sense. Therefore, it is reasonable to incorporate the statistical techniques into classical deterministic methods to alleviate the amount of required measurements. Motivated by these discussions, in this work we aim to combine the existing direct sampling scheme with a statistical model, in an attempt to reconstruct the motion trajectory more efficiently and effectively, especially in the extremely challenging case of limited aperture data. 

Let $D$ be a bounded domain in $\mathbb{R}^3$ such that $\Omega \subset D$. The measurement sensors are located on a surface $\Gamma \subset \mathbb{R}^3\setminus \Bar{D}$ (see Fig.~\ref{fig:illustration}). Define 
\begin{equation}\label{phixty}
  \phi(x,t;y) = \frac{\lambda(t-c^{-1}\|x-y\|)}{4 \pi \|x-y\|},\quad (x,t;y) \in \Gamma \times (0,T] \times D.
\end{equation}
Let $\{T_1,\cdots,T_J\}$ be a partition of $(0,T]$ such that $T_{j}=( (j-1)p, jp], j=1,2,\cdots, J$. Since the temporal function  $\lambda(t)$ is a periodic function with period $p$ and $\|v\|\ll c$, the displacement of the point source is small on each $T_j$. Thus $z(t)$ can be approximated by a stationary point $y \in \mathbb{R}^3$ on $T_j$.

Given the measurement data  $u(x,t)$ with $x \in \Gamma$ and $t \in T_j$, $j=1,2, \cdots, J$, we propose the following indicator function
\begin{equation}\label{IndicatorC}
    I(y; T_j):= \int _{T_j} \dfrac{|\langle u(x,t),\phi(x,t;y)\rangle_{L^2(\Gamma)} |}{ \|u(x,t)\|_{L^2(\Gamma)} \|\phi(x,t;y)\|_{L^2(\Gamma)} }~\mathrm{d} t, \quad y \in D,
\end{equation}
where $\langle\cdot,\cdot\rangle_{L^2(\Gamma)}$ denotes the usual $L^2$ inner product. Note that, in \eqref{IndicatorC}, instead of the exact radiating field for a moving point source $y \in D$, we use an approximated field $\phi$ for a stationary point source in \eqref{phixty}. This is why we call the following algorithm the approximated direct sampling method (ADSM). It is clear that the use of $\phi$ simplifies the computation of the indicators.

The numerical results in Section \ref{sec:numerical_results} show that this indicator function approaches its maximum when the sampling point $y$ tends to the exact instantaneous location of the point source at each time steps. 
 
Now we are ready to present the approximate direct sampling method for the TISP.
\begin{itemize}
    \item[] {{ ADSM:} Approximate Direct Sampling Method}
    \item[] {\bf Step 1:} Collect the data $u(x,t)$ at the measurement points $\{x_l\}\in\Gamma$, a sequence of discrete time $\{t^n_j\}\in (0, T)$ and for each $j$, $\{t^n_j\}\in T_j$;
    \item[] {\bf Step 2:} Generate the sampling points $\mathcal{T}_h$ for $D$;
    \item[] {\bf Step 3:} Compute $I(y;T_j)$ for the sub-interval $T_j$ at sampling point $y \in \mathcal{T}_h$;
    \item[] {\bf Step 4:} Locate the maximum of $I(y_j; T_j)$ for each $j$, take the corresponding $y_j$ as an approximation of $z(t)$ on $T_{j}$. The sequence of locations $\{y_j\}_{j=1}^J$ is the reconstructed moving path.
\end{itemize}

%\begin{table}[h!]
%\centering
%\begin{tabular}{cp{.8\textwidth}}
%\toprule
%\multicolumn{2}{l}{{\bf ADSM:}\quad Approximate Direct Sampling Method}\\
%\midrule
%{\bf Step 1:} & Collect the data $u(x,t)$ at the measurement points $\{x_l\}$ on $\Gamma$ and a sequence of discrete time $\{t^n_j\}\in (0, T)$ and for each $j$, $\{t^n_j\}\in T_j$;\\
%{\bf Step 2:} & Generate the sampling mesh $\mathcal{T}_h$ for $D$; \\
%{\bf Step 3:} & Compute $I(y;T_j)$ for the sub-interval $T_j$ and sampling point $y \in D$; \\
%{\bf Step 4:} & Locate the maximum of $I(y_j; T_j)$ for each $j$, take the corresponding $y_j$ as an approximation of $z(t)$ where $t\in T_{j}$. The sequence $\{y_j\}_{j=1}^J$ is the reconstructed moving path.\\
%\bottomrule
%\end{tabular}
%\end{table}

The performance of the {ADSM} depends on how much data $u(x,t)$ are available. 
If the measured data are complete or nearly complete, for example, $\Gamma$ is a sphere with $D$ inside and $u(x,t)$ is available for $x \in \Gamma$ and $t \in (0, T]$, the algorithm can produce a good reconstruction of the trajectory of the moving point source $z(t)$. However, in practice, the measured data $u(x,t)$ are usually partial and the reconstructions can be unsatisfactory. Nonetheless, the ADSM provides useful information of the moving source. Such information can be coded in the priors for the Bayesian inversion to obtain improved reconstructions.

\section{Bayesian Inversion}\label{sec:bayesian_inversion}

In this section, we employ the Bayesian inversion to refine the results obtained by the ADSM. The {TISP} is reformulated as a statistical inference problem for the source location. The approximate location obtained by the ADSM is coded in the priors for Bayesian inversion.  The well-posedness of the posterior density function is analyzed using the Bayesian approximation error approach \cite{JariE.Somersalo2006} and the framework in \cite{Stuart2010}. An MCMC algorithm is then used to explore the posterior density function of the source location.

\subsection{Bayesian Framework for Inverse Problems}
The inverse problem can be viewed as seeking information about the unknown $q$ given the measurement $m$ under the model function $\mathcal{F}$, which might be inaccurate and contain noise $\xi$, i.e., given $m$, reconstruct $q$ such that
\begin{equation} \label{statistical model}
   m=\mathcal{F}(q,\xi).
\end{equation}

For statistical inverse problems, the parameters in \eqref{statistical model} are treated as random variables. Denote the probability density function and the probability measure of a random variable $\gamma$ by $\pi_{\gamma}$ and $\mu_{\gamma}$, respectively. Note that the information about $q$ obtained by the {ADSM} will be coded into the prior density $\pi(q)\footnote{If the arguments coincide with the probability density function, we drop the subscripts. For example, we write $\pi_{q}(q)=\pi(q)$, $\pi_{m|q}(m|q)=\pi(m|q)$, but retain the subscript in $\pi_{e}(m-\mathcal{F}(q)-\zeta)$.}$. The probability of $m$ given $q$, which is called the likelihood function, is denoted by $\pi(m|q)$.

Bayesian inversion seeks the posterior distribution $\pi(q|m)$. Using the Bayes' theorem, the joint distribution of the associated variables can be decomposed as
\begin{equation*}
 \pi(q,m,\xi)=\pi(m,\xi|q)\pi(q) = \pi(q,\xi|m) \pi(m)
\end{equation*}
and the posterior distribution of $q$ is then given by
\begin{equation*}
    \pi(q|m)= \int \pi(q,\xi|m) \mathrm{d} \xi.
\end{equation*}
Moreover, $\pi(q|m) \propto \pi(m|q)\pi(q)$, where $\propto$ means ``proportional to''. 
 % There are various methods, e.g., Markov chain Monte Carlo (MCMC), to explore the posterior probability density distribution $\pi(q|m)$. 
Point estimates of $\pi(q|m)$ are often used as the solution of the Bayesian method. Two popular statistical point estimates are the maximum a posterior estimate (MAP) 
\[
q_{\text{MAP}}=\arg \max\limits_q \, \pi(q|m)
\]
and the conditional mean (CM) 
\[
q_{\text{CM}}=\mathbb{E}(\pi(q|m)).
\]

In \eqref{statistical model}, $m=\{u(x,t)\} \in Y=\mathbb{R}^{N_{x}\times N_{p}}$ is the noisy measurement data collected on $\Gamma \times T_{j}$, $j=1,2,\cdots,J$, where $N_{x}$ and $N_{p}$ are the numbers of observation positions and times. The unknown parameter $q\in X= \mathbb{R}^3$, i.e., the position of the point source on the time period $T_{j}$, is assumed to be static due to \eqref{vtllc}. The space $Y$ is equipped with the Frobenius norm $\|A\|_{Y}=\left( \sum_{i=1}^{N_{p}} \sum_{j=1}^{N_{x}} |a_{ij}|^2 \right)^{\frac{1}{2}}$ for $ A\in Y$.  In this paper, the Markov chain Monte Carlo (MCMC) method is employed to explore the posterior distribution $\pi(q|m)$ using the Bayesian approximation error approach and the CM is adopted as the point estimate for $q$.

\subsection{Bayesian Approximation Error Approach}
The main idea of the Bayesian approximation error (BAE) approach is to replace the accurate forward model by a less accurate but computationally feasible one \cite{JariE.Somersalo2006, Kolehmainen2011, Kaipio2019}. Let $\bar{q}=z(t)$ and $\bar{\mathcal{F}}$ be the accurate forward model given by \eqref{exact solution 1}. We consider the following statistical model for the TISP
\begin{equation}\label{statistical model 1}
    m=\bar{\mathcal{F}}(\bar{q})+e,
\end{equation}
where $e$ is the additive error. Furthermore, we assume the noise $e$ follows a normal distribution with mean $e_{*}$ and variance $\Sigma_{ee}$, i.e., $\pi_{e}(e)=\mathcal{N}(e_{*},\Sigma_{ee})$. Here we take $e_{*}=0$.

Using the statistical model \eqref{statistical model 1} would require the knowledge of the exact path of the point source for $t\in(0,T]$, which is inaccessible. Fortunately, the exact path is not necessary for the BAE. One proceeds as follows. Let $P$ be a projection operator and $q=P\bar{q}$. Instead of the accurate forward model $\bar{F}(\bar{q})$, one uses the following approximate forward operator
\begin{equation}\label{ApproxF}
   \mathcal{F}(q) = \frac{\lambda(t-c^{-1}\|x-q\|)}{4\pi\| x-q\|},
\end{equation}
which is computationally cheaper and more  approachable. Then the statistical model can be written as
\begin{equation*}
m=\bar{\mathcal{F}}(\bar{q})+e = \mathcal{F}(q)+\zeta+e, 
\end{equation*}
where $\zeta=\bar{\mathcal{F}}(\bar{q})-\mathcal{F}(q)$ is the approximation error. Due to the fact that $\|v(t)\| \ll c$ and \eqref{Duzuz0}, it is sufficient to study the well-posedness for the approximate forward operator $\mathcal{F}(q)$.

Assume that $e$ is separately independent of $q$ and $\zeta$. According to \cite{Kolehmainen2011, Kaipio2019}, the posterior distribution satisfies
\begin{equation*}
    \pi(q|m)~\propto~ \pi(m|q)\pi(q)={\int_{Y} \pi_{e}(m-\mathcal{F}(q)-\zeta)\pi(\zeta|q)~\mathrm{d}\zeta}~\pi(q).
\end{equation*}
In the approximation error approach, $\pi(\zeta|q) $ is approximated with a normal distribution. Assume that the normal approximation of the joint distribution $\pi(\zeta,q)$ is given by
\begin{equation*}
    \pi(\zeta,q)~\propto~ \exp \left\{ -\frac{1}{2}
    \begin{pmatrix}
         \zeta-\zeta_{*} \\
          q-q_{*}
    \end{pmatrix}^\top 
    \begin{pmatrix}
        \Sigma_{\zeta\zeta} & \Sigma_{\zeta q} \\
        \Sigma_{q\zeta}  & \Sigma_{qq}
    \end{pmatrix}^{-1}
    \begin{pmatrix}
         \zeta-\zeta_{*} \\
          q-q_{*}
    \end{pmatrix}   
    \right\},
\end{equation*}
Then we can write $\zeta|q \sim \mathcal{N}(\zeta_{*|q},\Sigma_{\zeta|q})$, where
\begin{equation*}
    \begin{split}
        \zeta_{*|q}&= \zeta_{*} +\Sigma_{\zeta q}\Sigma^{-1}_{q q} (q-q_{*}),\\
        \Sigma_{\zeta|q}&= \Sigma_{\zeta\zeta} -\Sigma_{\zeta q}\Sigma^{-1}_{q q}\Sigma_{q \zeta } .\\
    \end{split}
\end{equation*}
Define the normal random variable $w$ so that $w=e+\zeta|q$. It holds that
\begin{equation*}
    w|q \sim \mathcal{N}(w_{*|q},\Sigma_{w|q}),
\end{equation*}
where
\begin{equation*}
    \begin{split}
        w_{*|q}&= e_{*}+\zeta_{*} +\Sigma_{\zeta q}\Sigma^{-1}_{q q} (q-q_{*}),\\
        \Sigma_{w|q}&=\Sigma_{ee}+ \Sigma_{\zeta\zeta} -\Sigma_{\zeta q}\Sigma^{-1}_{q q}\Sigma_{q \zeta } .\\
    \end{split}
\end{equation*}

The approximate posterior distribution can be written as
 \begin{equation*}
     \pi(q|m)~\propto~ \exp \left(-\frac{1}{2} \|( m-\mathcal{F}(q)- w_{*|q})^{\top} \Sigma_{w|q}^{-1}( m-\mathcal{F}(q)- w_{*|q})\|^2_{Y} \right) \pi(q) ,
 \end{equation*}
Assume that $\mu_{m}$ is absolutely continuous with respect to $\mu_{q}$, i.e., $\mu_{m}\ll \mu_{q}$. Using Bayes' formula, we have that
\begin{equation*}
    \frac{\mathrm{d} \mu_{m}}{\mathrm{d}\mu_{q}}(q)=\frac{1}{L(m)}\exp (-G(q;m)),
\end{equation*}
where $L(m):=\displaystyle{ \int_{X} \exp \left(-\frac{1}{2} \|( m-\mathcal{F}(q)- w_{*|q})^{\top} \Sigma_{w|q}^{-1}( m-\mathcal{F}(q)- w_{*|q})\|^2_{Y} \right)  } \mathrm{d} \mu_{q}(q)$.
In this paper, the prior is chosen to be a normal distribution, i.e., $q \sim \mathcal{N}(q_{*},\Sigma_{qq})$. We note that in practice other prior distributions can be used as well. 

%\textcolor{red}{Neglecting the correlation between the unknown and the modelled additive error and setting $\Gamma_{\zeta q}=0$}, one has the enhanced error model \cite{JariE.Somersalo2006}. 

We now study the properties of the operator $\mathcal{F}$ following \cite{Stuart2010}.
\begin{lemma} \label{Lemma1}
For every $\varepsilon>0$, there exists $C:=C(\varepsilon)\in \mathbb{R}$ such that, for all $q\in X$,
\begin{equation*}\label{F_property 1}
\|\mathcal{F}(q)\|_{Y}\leqslant \exp (\varepsilon \|q\|^2_X+C).
\end{equation*}
\end{lemma}
\begin{proof}
Since $\lambda(t)\in \mathcal{C}^{1}[0,T]$, we have $|\lambda(t)|\leqslant C_{1}$. Define $d= \inf\{\|y_1-y_2\|: y_1\in \Gamma, y_2\in D\}$. By \eqref{ApproxF}, it holds that
 \begin{equation*}
 \begin{split}
    |\mathcal{F}(q)| &  \leqslant \frac{|\lambda(t-c^{-1}\| x-q\|)|}{4\pi d}\leqslant \frac{C_1}{4\pi d} . 
 \end{split}
 \end{equation*}
 Therefore,
 \begin{equation*}
   \|\mathcal{F}(q)\|_{Y}\leqslant \exp (\varepsilon \|q\|^2_X+C)
 \end{equation*}
 and the proof is complete.
\end{proof}
\begin{lemma}
For every $r>0$ there exists a $C:=C(r)>0$ such that, for all $q_{1},q_{2}\in X$ with $\max\{\|q_{1}\|_{X},\|q_{2}\|_{X}\}<r$,
\begin{equation*}\label{F_property 2}
\|\mathcal{F}(q_1)-\mathcal{F}(q_2)\|_{Y}\leqslant
C \|q_1-q_2\|_X.
\end{equation*}
\end{lemma}
\begin{proof}
Define $\tilde{d}= \sup\{\|y_1-y_2\|: y_1\in \Gamma, y_2\in D\} $. Since $\lambda(t)\in \mathcal{C}^1 [0,T]$, we have \[
|\lambda(t_{1})-\lambda(t_{2})|\leqslant C_{2}|t_1-t_2|, ~\forall t_{1},t_{2} \in (0,T),
\] for some constant $C_2$. Consequently,
\begin{equation*}
    \begin{split}
    |\mathcal{F}(q_1)-\mathcal{F}(q_2)| = & \left|\frac{\lambda(t-c^{-1}\|x- q_1\|)}{4\pi \| x-q_1\|} - \frac{\lambda(t-c^{-1}\|x-q_2\|)}{4\pi \|x-q_2\|}  \right| \\
    = & \left|\frac{\lambda(t-c^{-1}\|x-q_1\|)\|x-q_2\|-\lambda(t-c^{-1}\|x-q_2\|)\| x-q_{1}\|}{4\pi \| x-q_1\|\|x-q_2\|}\right| \\
    \leq & \left |\frac{\lambda(t-c^{-1}\| x-q_1\|)\|x- q_2\|-\lambda(t-c^{-1}\| x-q_1\|)\| x-q_1\|}{4\pi d^2} \right| \\
    & + \left|\frac{\lambda(t-c^{-1}\| x-q_{1}\|)\| x-q_{1}
    \|-\lambda(t-c^{-1}\| x-q_{2}\|)\| x-q_1\|}{4\pi d^2} \right| \\
    \leq & C_{1}\left|\frac{\|x- q_2\|-\| x-q_1\|}{4\pi d^2} \right|+  C_{2} \left|\frac{\|x- q_{2}\|-\| x-q_{1}\|}{4\pi d^2}  \right|\frac{1}{c}\tilde{d} \\ 
    \leqslant & \frac{\|q_1-q_2\|_{X}}{4\pi d^2} \left( C_{1}+ \frac{1}{c}C_{2}\tilde{d} \right).  \\
    \end{split}
\end{equation*}
It yields that
\begin{equation*}
\|\mathcal{F}(q_1)-\mathcal{F}(q_2)\|_{Y}\leqslant
C \|q_1-q_2\|_X.
\end{equation*}
\end{proof}

\begin{definition}
The Hellinger distance between two probability measures $\mu_1$ and $\mu_2$ with common reference measure $\nu$ is defined as
\begin{equation*}
    d_{\rm Hell}(\mu_1, \mu_2)=\left( \int \left(\sqrt{\mathrm{d} \mu_{1} / \mathrm{d} \nu}-\sqrt{\mathrm{d} \mu_{2} / \mathrm{d} \nu }\right)^2 ~ \mathrm{d} \nu \right)^{1/2}.
\end{equation*}
\end{definition}

\begin{theorem}\label{Theorem4}[Page 530 of \cite{Stuart2010} \textbf{The Fernique theorem}]
If $\mu$ is a Gaussian measure with mean $0$ on a Banach space $X$ so that $\mu(X)=1$, then there exists $\alpha>0$ such that 
\begin{equation*}
    \int_{X} \exp\left( \alpha \|x\|^2_{X}\right) \mu({ \mathrm{d} x}) < \infty.
\end{equation*}
\end{theorem}

The following theorem provides the well-posedness of the Bayesian approximation error approach.
\begin{theorem}
Assume that $\mu_{q}$ is a Gaussian measure satisfying $\mu_q(X)=1$ and $\mu_{m} \ll \mu_{q}$. For $m_1$ and $m_2$ with $\max\{\|m_1\|_{Y}, \|m_2\|_{Y}\}\leq r$, there exists $M=M(r)>0$ such that
\begin{equation*}
   d_{\rm Hell}(\mu_{m_1}, \mu_{m_2}) \leq M \|m_1-m_2\|_{Y}.
\end{equation*}
\end{theorem}

\begin{proof}
Given 
\begin{equation*}
    L(m)=\int_{X} \exp \left(-\frac{1}{2} \|( m-\mathcal{F}(q)- w_{*|q})^{\top} \Sigma_{w|q}^{-1}( m-\mathcal{F}(q)- w_{*|q})\|^2_{Y} \right)  \mathrm{d} \mu_{q}(q),
\end{equation*}
it is obvious that 
\begin{equation}\label{Hproof1}
    0\leq L(m)\leq 1.
\end{equation}
According to Lemma~\ref{Lemma1}, we have that 
\begin{equation}\label{Hproof2}
\begin{split}
   L(m) & \geq \int_{\|q\|_{X}\leq 1} \exp \left(- \left|\Sigma_{w|q}^{-1}\right|~\| m- w_{*|q}\|^2_{Y}-\left|\Sigma_{w|q}^{-1}\right|~\| \mathcal{F}(q)\|^2_{Y} \right)  \mathrm{d} \mu_{q}(q) \\
   & \geq \int_{\|q\|_{X}\leq 1} \exp (-M) \mathrm{d} \mu_{q}(q) \\
   & = \exp(-M)\mu_{q}\{\|q\|_{X}\leq 1 \} \\
   & > 0,
\end{split}
\end{equation}
since $\mu_{q}$ is a Gaussian measure.

Using the Fernique Theorem (Theorem~\ref{Theorem4}) and Lemma~\ref{Lemma1}, for $\mu_{q}$, it holds that
\begin{equation} \label{Hproof3}
   \begin{split}
   | L(m_1)-L(m_2)| & \leq \int_{X} \left|\exp (-G(q;m_{1}))- \exp (-G(q;m_{2})) \right|  \mathrm{d} \mu_{q}(q). \\
      & \leq \int_{X} \left| G(q;m_{1})-G(q;m_{2}) \right|  \mathrm{d}\mu_{q}(q) \\
      & \leq \int_{X}  \frac{\left|\Sigma_{w|q}^{-1}\right|}{2}\left| \| m_1-\mathcal{F}(q)- w_{*|q}\|_{Y}^2- \| m_2-\mathcal{F}(q)- w_{*|q}\|_{Y}^2 \right|  \mathrm{d} \mu_{q}(q)\\
      & \leq  \int_{X}  \frac{\left|\Sigma_{w|q}^{-1}\right|}{2}\left(\left| \| m_1\|_{Y}^2-\| m_2\|_{Y}^2\right| +  \|\mathcal{F}(q)+ w_{*|q}\|_{Y} ~ \| m_1- m_2\|_{Y}  \right) \mathrm{d} \mu_{q}(q)\\
      & \leq  \int_{X}  \frac{\left|\Sigma_{w|q}^{-1}\right|}{2}\left( \| m_1\|_{Y}+\| m_2\|_{Y} + \|\mathcal{F}(q)+ w_{*|q}\|_{Y} \right) \mathrm{d} \mu_{q}(q) \| m_1- m_2\| _{Y}\\
      &\leq M \|m_1-m_2\|_{Y}.
    \end{split}
\end{equation}
From the definition of the Hellinger distance, we obtain that 
\begin{equation} \label{Hproof4}
  \begin{split}
     d_{\rm Hell}^2(\mu_{m_1}, \mu_{m_2}) = & \frac{1}{2} \int_{X} \left\{ \left(\frac{\exp (-G(q;m_{1}))}{L(m_{1})} \right)^{1/2} -\left(\frac{\exp (-G(q;m_{2}))}{L(m_{2})} \right)^{1/2}\right\}^2 \mathrm{d} \mu_{q}(q)  \\
    = & \frac{1}{2} \int_{X} \left\{ \left(\frac{\exp (-G(q;m_{1}))}{L(m_{1})} \right)^{1/2} -\left(\frac{\exp (-G(q;m_{2}))}{L(m_{1})} \right)^{1/2} \right. \\
    & \left. {} +\left(\frac{\exp (-G(q;m_{2}))}{L(m_{1})} \right)^{1/2} -\left(\frac{\exp (-G(q;m_{2}))}{L(m_2)} \right)^{1/2} \right\}^2 \mathrm{d} \mu_{q}(q) \\
    \leq & L(m_{1})^{-1} \int_{X} \left\{ {\exp \left(-\frac{1}{2}G(q;m_{1}) \right)} -{\exp \left(-\frac{1}{2}G(q;m_{2}) \right)} \right\}^2 \mathrm{d} \mu_{q}(q) \\
    & + \left| L(m_{1})^{-1/2}- L(m_{2})^{-1/2} \right|^2 \int_{X}  {\exp (-G(q;m_2))}  \mathrm{d} \mu_{q}(q).
  \end{split}
\end{equation}
With the Fernique theorem and Lemma~\ref{Lemma1}, it holds that 
\begin{equation}\label{Hproof5}
    \begin{split}
       & \int_{X} \left\{ {\exp \Big(-\frac{1}{2}G(q;m_1) \Big)} -{\exp \Big(-\frac{1}{2}G(q;m_{2}) \Big)} \right\}^2 \mathrm{d} \mu_{q}(q) \\
       & \leqslant \int_{X} \Big|\frac{1}{2}G(q;m_1) -\frac{1}{2}G(q;m_2) \Big|^2 \mathrm{d} \mu_{q}(q) \\ 
       & \leqslant \frac{ \left|\Sigma_{w|q}^{-2}\right|}{16} \int_{X}  \Big|\|m_{1}-\mathcal{F}(q)- w_{*|q}\|_{Y}^2 -\|m_{2}-\mathcal{F}(q)- w_{*|q}\|_{Y}^2\Big|^2 \mathrm{d} \mu_{q}(q) \\
       &\leqslant M \|m_{1}-m_{2}\|_{Y}^2.
    \end{split}
\end{equation}
Using the bounds on $L(m_1)$ and $L(m_2)$, we have that
\begin{equation} \label{Hproof6}
    \left | L(m_1)^{-1/2}- L(m_2)^{-1/2} \right |^2 \leqslant M \max \Big(L(m_1)^{-3}, L(m_2)^{-3} \Big) | L(m_{1})- L(m_{2}) |^2.
\end{equation}
Combining \eqref{Hproof1}-\eqref{Hproof6}, we  that
\begin{equation*}
   d_{\rm Hell}(\mu_{m_1}, \mu_{m_2}) \leqslant M\|m_1-m_2\|.
\end{equation*}
\end{proof}

We are now ready to present the MCMC method to reconstruct the trajectory $z(t)$ by incorporating in the priors the information obtained by the {ADSM}.

\begin{itemize}
    \item[] {\bf ADSM-MCMC:}
    \item[1.] Given $p$, $J$, and measured data $u(x_l,t_{j}^{n})$ on $U_T$.
    \item[2.] For $j=1,\ldots, J$, do
    \begin{itemize}
        \item[a.] Use the {ADSM} to find $y_j$ for $T_j$ and set  $q_{j}^{(1)}=y_j$.
        \item[b.] For $k=1,\cdots,K-1$, do
        \begin{itemize}
            \item[I.]
        Generate \[\Tilde{q}=\beta y_{j}+ (1-\beta)q_{j-1} + \sigma \mathcal{N}(0,1) \text{ if } j >1 \,(\beta = 1 \text{ if } j=1);\]
        \item[II.]  Compute $\alpha(\tilde{q},q_{j}^{(k)}) = \min \left(1, \dfrac{\pi(\tilde{q})}{\pi\left(q_{j}^{(k)}\right)} \right)$;
        \item[III.] Draw $\tilde{\alpha} \sim \mathcal{U}(0,1)$. If $\alpha>\tilde{\alpha}$, then $q_{j}^{(k+1)}=\tilde{q}$, otherwise $q_{j}^{(k+1)}=q_{j}^{(k)}$.
        \end{itemize}
        \item[c.] Set $q_{j}=\frac{1}{K} \displaystyle{ \sum_{k=1}^K q_{j}^{(k)}}$.
    \end{itemize}
    \item[3.] Output the moving path $\{q_j\}_{j=1}^J$. 
\end{itemize}

%\begin{table}[h]
%\centering
%\begin{tabular}{cp{.8\textwidth}}
%\toprule
%\multicolumn{2}{l}{{\bf Algorithm ADSM-MCMC:}\quad Sampling-Bayesian scheme for imaging the moving source}\\
%\midrule
%{\bf Step 1:} & Properly choose period $p$ and collect the data $u(x,t)$ at the measurement points $\{x_l\}$ on $\Gamma$ and a sequence of discrete time $\{t^n_j\}\in (0,T)$ and for each $j$, $\{t^n_j\}\in T_j$. Use {\bf ADSM} to find $y_1$ and $y_2$; \\
%{\bf Step 2:} & Choose the initial value $q_0=\frac{y_1+y_2}{2}$ and $j=1$; \\
%{\bf Step 3:} & Let $q_{j}^{(1)}=q_{j-1}$ and use the {\bf ADSM} to find $y_j$. For $k=1,\cdots,K-1$, generate \[
%\Tilde{q}=\lambda q_{j-1}+ (1-\lambda) y_j + \sigma \mathcal{N}(0,1), \lambda \in [0,1];
%\] \\
%{\bf Step 4:} & Compute $\alpha(\tilde{q},q_{j}^{(k)}) = \min \left(1, \dfrac{\pi(\tilde{q})}{\pi(q_{j}^{(k)})} \right)$; \\
%{\bf Step 5:} & Draw $\tilde{\alpha} \sim \mathcal{U}(0,1)$. If $\alpha>\tilde{\alpha}$, then $q_{j}^{(k+1)}=\tilde{q}$, otherwise $q_{j}^{(k+1)}=q_{j}^{(k)}$; \\
%{\bf Step 6:} & Set $q_j= \displaystyle{ \sum_{k=1}^k q_{j}^{(k)}}$. Stop when $j=J$, otherwise, $j\leftarrow j+1 $ and repeat from Step 3.\\
%\bottomrule
%\end{tabular}
%\end{table}

We remark that how to use of the locations obtained by the ADSM is problem-dependent. For example, if the target moves rather slow, one can set $\beta = 0$. If the target moves fast, one can choose $\beta=1$.

\section{Numerical Examples} \label{sec:numerical_results}

In this section, we will present several numerical examples to demonstrate the performance of the proposed method. On each $T_j$, we choose $\lambda(t)$ to be the Ricker wavelet, which has been widely used in geophysics (see, e.g., \cite{Wang2018}),
\begin{equation}\label{Ricker}
    \lambda(t) = \left(1-2\pi^2f_{0}^2(t-p/2)^2\right) \exp\left(-\pi^2f_{0}^2(t-p/2)^2\right),\quad t\in [0,p],
\end{equation}
where $f_0\in \mathbb{R}_+$ is the central frequency.
The synthetic data $u(x,t)$ is generated using \eqref{exact solution 2} with some noises:
\[
u^{\varepsilon}(x,t)=u(x,t)(1+\varepsilon r),
\]
where $\varepsilon$ is the noise level and $r$ is a random number from the uniform distribution $\mathcal{U}[-1, 1]$.

For all examples, we set $c=330$, the terminal time $T=4$ and $p=0.1$. The number of equidistant time steps in one period is $N_{p}$. Since $(0,T]= \bigcup_{j=1}^{J} T_{j}$, where $J=T/p$, we have the time increment $\Delta t=p/N_{p}$ and $N_{T}=J N_{p}$. The time discretization is $t_{j}^{n}=\left((j-1)N_{p} + n\right)\Delta t$, $n=1,2,\cdots, N_{p}$, $j=1,\cdots,J$.  Using the spherical coordinates $(R,\theta,\eta)$
\[
 x:= R(\sin\eta \cos\theta, \sin\eta \sin\theta, \cos\eta),
\] 
the measurement aperture $\Gamma$ is a patch on the sphere with radius $R=7$, $\theta \in (0, 2\pi]$, $\eta \in (0, \pi]$. Assume that three measurement data sets are given by
$$\{u(x, t): (x,t)=(\theta,\eta,t)\in S_{i},  i=1,2,3\},$$ where (see Fig~\ref{fig: lambda and observation})
\begin{align*}
    S_1 & =\Big\{\frac{\pi}{16}l,~l=1,\cdots,32 \Big\}\times\Big\{\frac{\pi}{5}s,~s=1,2,3,4 \Big\} \times \{ t_{j}^{n}\},~N_x=128, ~N_{p}=12,\\
    S_2 & =\Big\{\pi+\frac{\pi}{8}l,~l=0,1,\cdots,8 \Big\}\times\Big\{\frac{\pi}{4},\frac{\pi}{2} \Big\} \times \{ t_{j}^{n}\} ,~N_{x}=18, ~N_{p}=10,\\
    S_3 & =\Big\{\pi, \frac{5\pi}{4}, \frac{3\pi}{2} \Big\}\times\Big\{\frac{\pi}{4},\frac{\pi}{2}\Big\}\times \{ t_{j}^{n}\},~N_{x}=6, ~N_{p}=7.  
\end{align*}
Note that the spacial coverage of $S_1$ is the whole sphere corresponding to the case of full aperture data.
The spacial coverage is $1/4$ of the sphere for $S_2$ and $1/8$ for $S_3$. In fact, for $S_3$, there are just 6 measurement locations.

Let the central frequency $f_{0}=100$ for the temporal function $\lambda(t)$ in \eqref{Ricker} in one period (see  Fig~\ref{fig: lambda and observation}(a)). The measurement locations are showed in Fig~\ref{fig: lambda and observation}(b)-(d) for $S_{1}$, $S_{2}$ and $S_{3}$, respectively.
\begin{figure}[ht] 
\centering
\subfloat[]{\includegraphics[width = 2in]{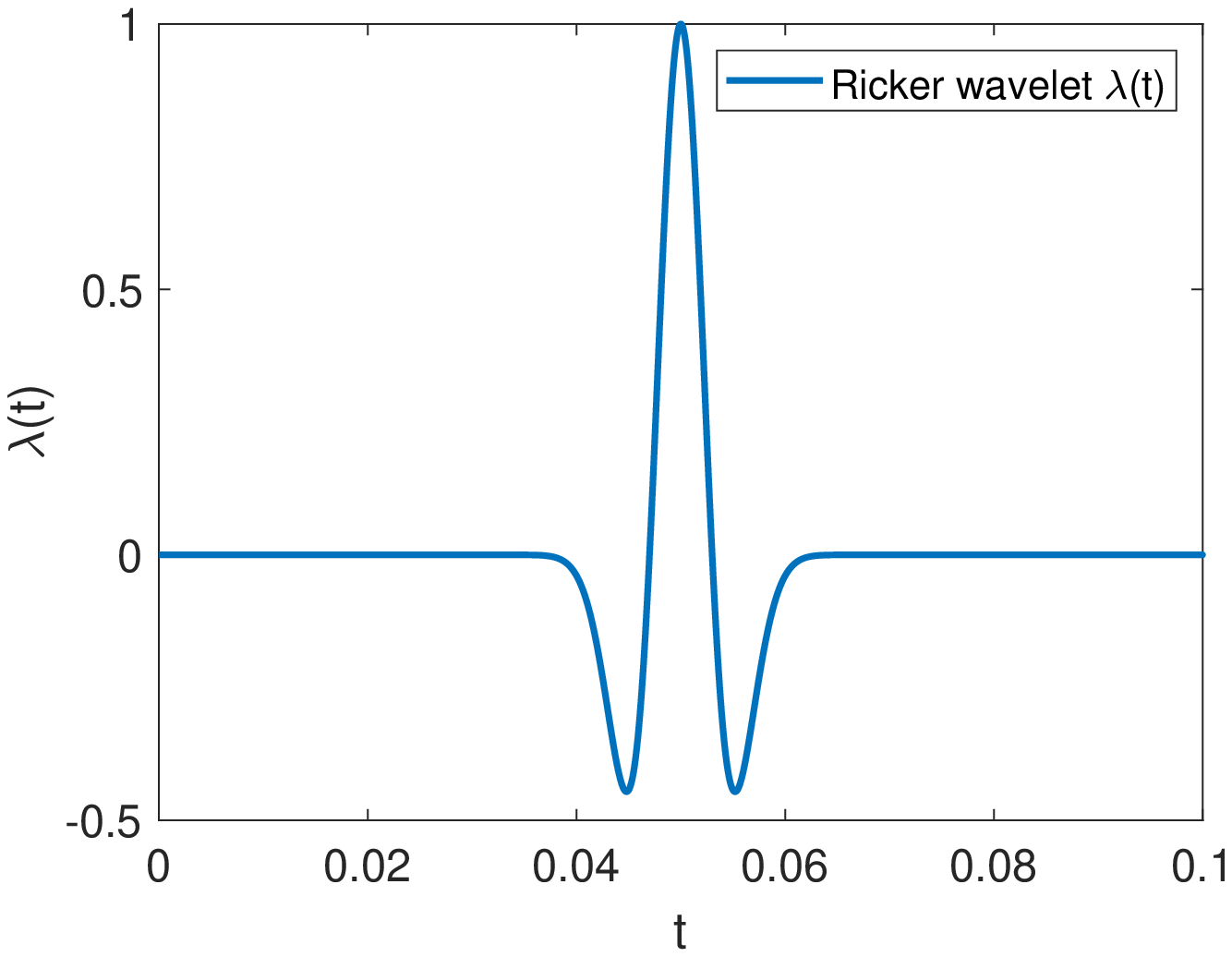}} \hspace{0.2in}
\subfloat[]{\includegraphics[width = 2.3in]{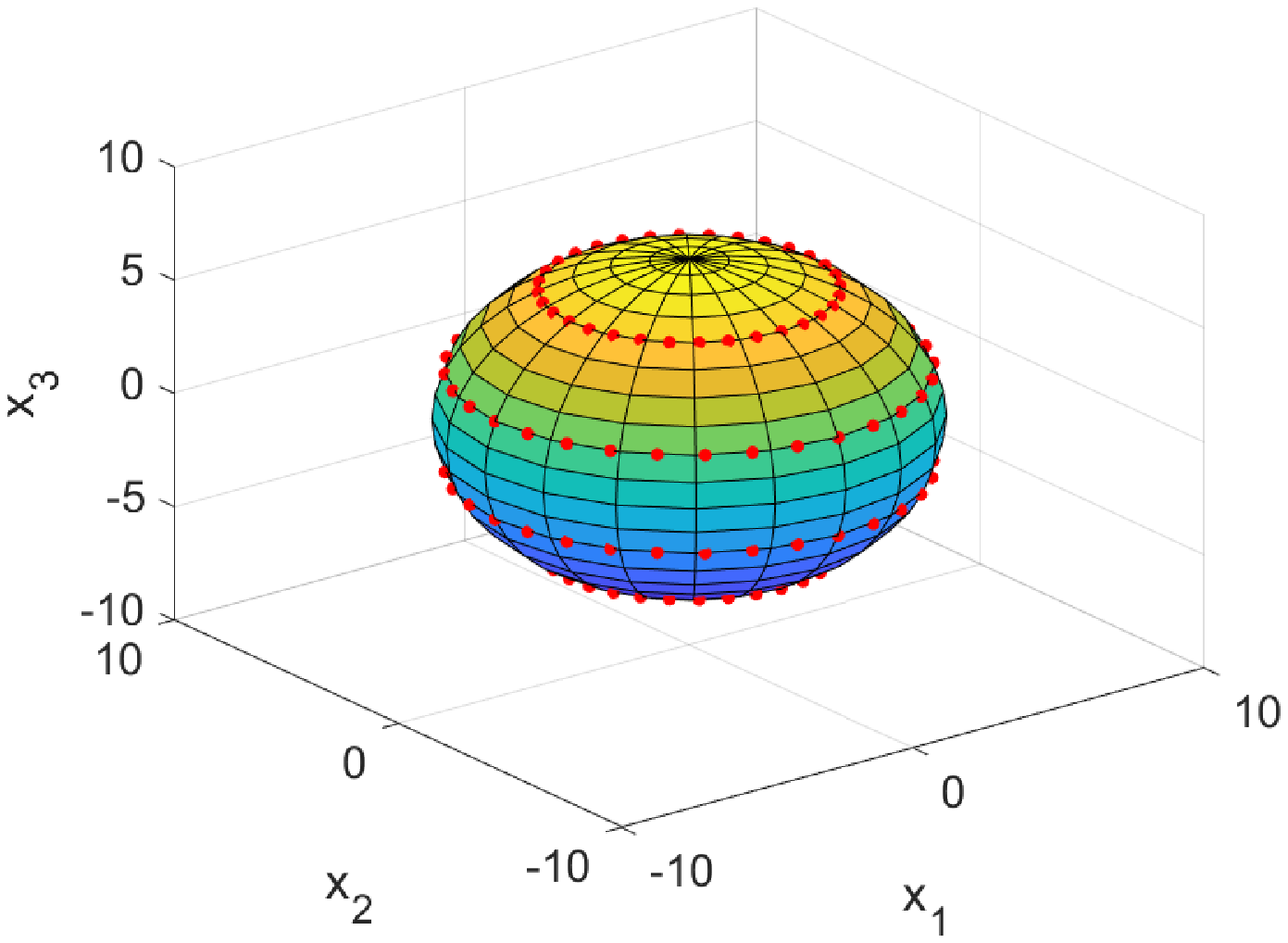}}\\
\subfloat[]{\includegraphics[width = 2.3in]{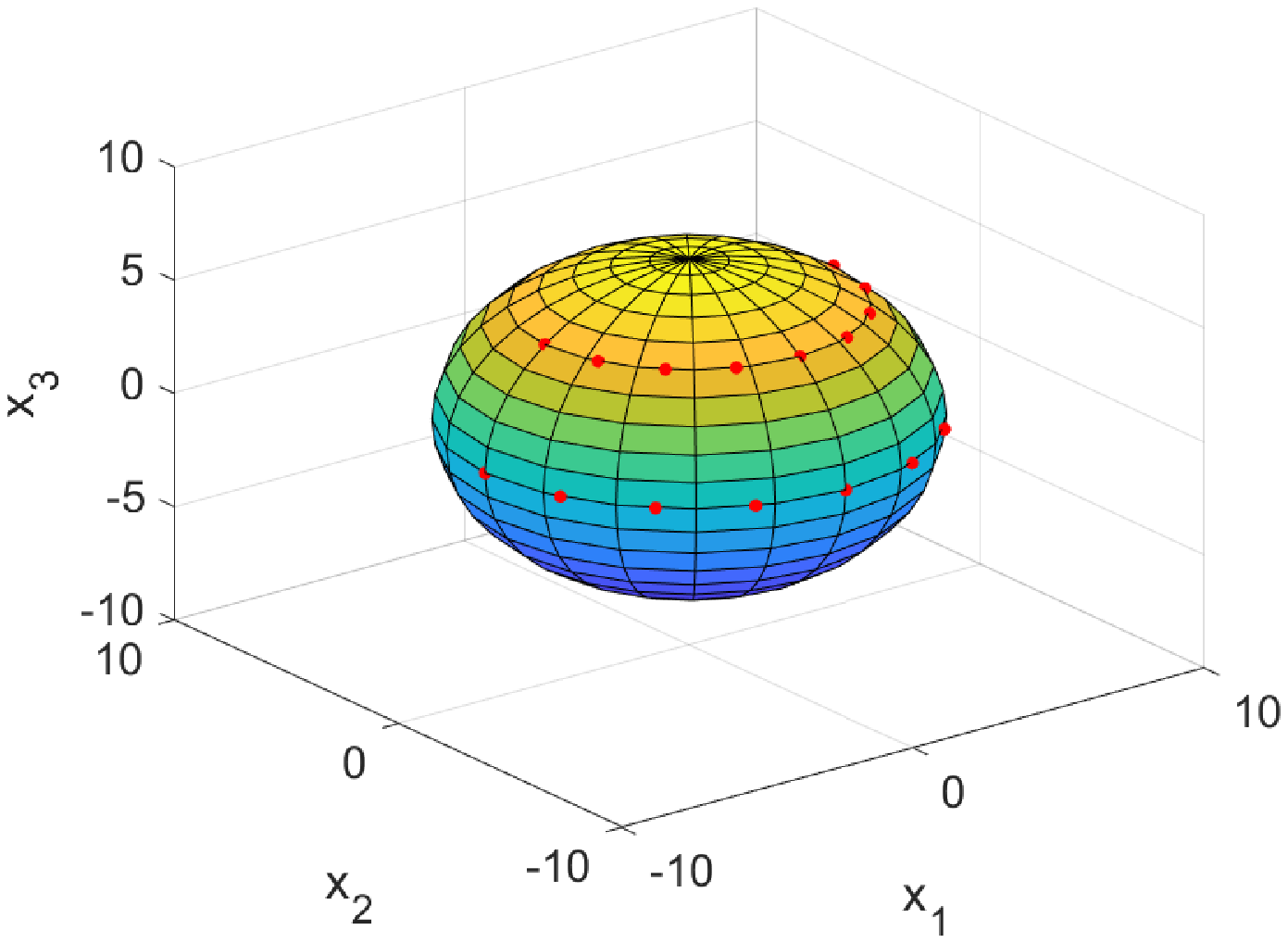}}
\subfloat[]{\includegraphics[width = 2.3in]{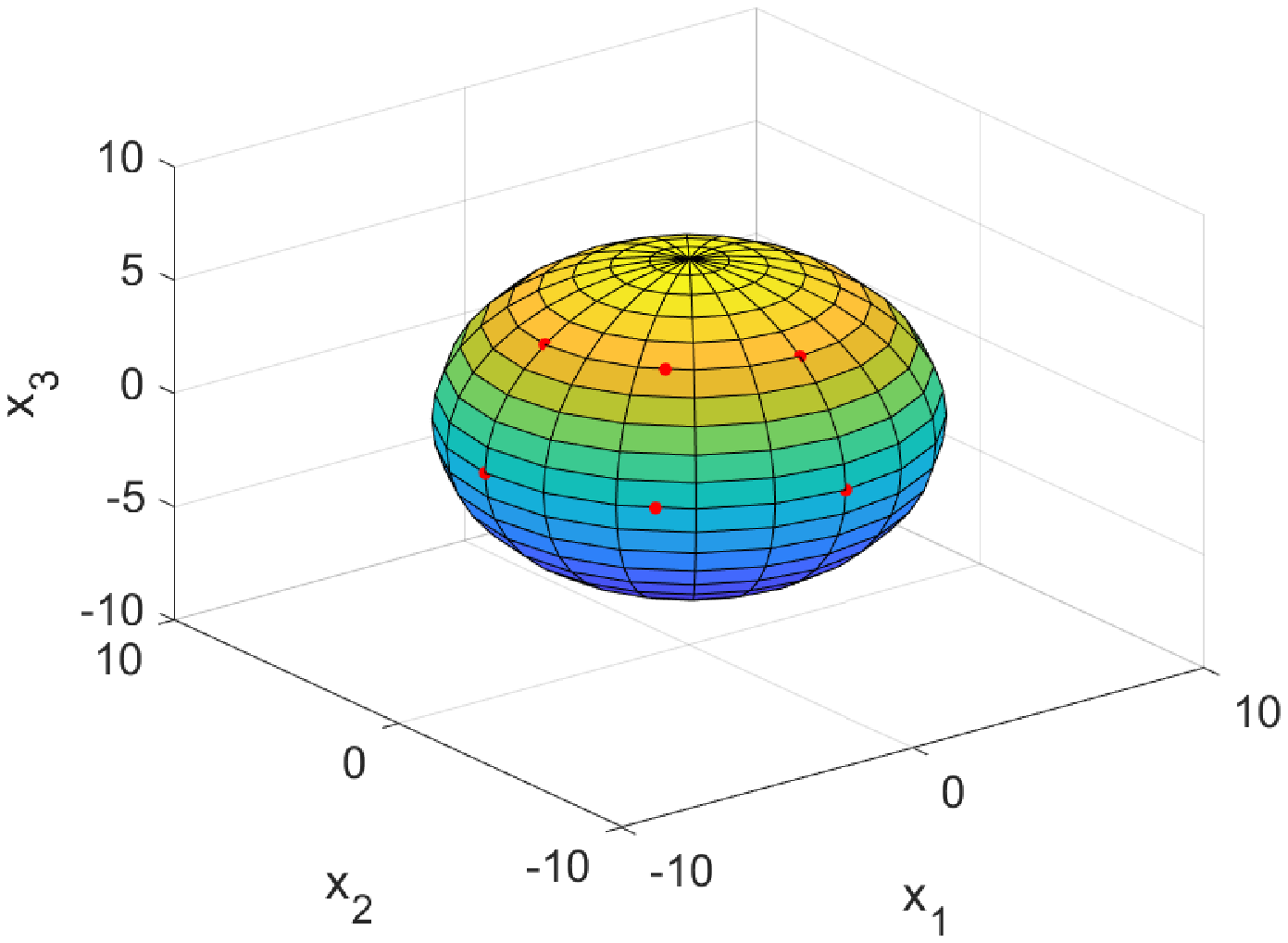}}
\caption{The pulse function and measurement locations. (a) the temporal function $\lambda(t)$; (b)-(d) measurement locations for $S_1, S_2$ and $S_3$.}
\label{fig: lambda and observation}
\end{figure}

We take $101\times 101 \times 101$ sampling points, denoted by $\mathcal{T}_h$, uniformly distributed in the sampling domain $D=[-5,5]^3$. For $j=1,2,\cdots,J$, the discrete version of the indicator function \eqref{IndicatorC} can be written as 
\begin{equation*}
    {\bm I} (y, T_j):= \frac{\sum\limits_{n=1}^{N_p}  \sum\limits_{l=1}^{N_x} |u(x_{l},t_j^n)\phi(x_{l},t_j^n;y) | }{\sum\limits_{n=1}^{N_p} \left(\sum\limits_{l=1}^{N_x} u^2(x_{l},t_j^n)  \right)^{1/2} \sum\limits_{n=1}^{N_p} \left(  \sum\limits_{l=1}^{N_x} \phi^2(x_{l},t_j^n;y)  \right)^{1/2} }, \quad ~y\in \mathcal{T}_h.
\end{equation*}
 The {ADSM} reconstructs the locations of the source as the global maximum points ${y_{j}}$ of the discrete indicator function $\bm{I}(y, T_{j})$ for $y \in \mathcal{T}_{h}$. The sequence $\{y_{j}\}_{j=1}^{J}$ is the reconstructed path $z(t)$ by the ADSM.
For partial data, the reconstructed path by the ADSM can be unsatisfactory. The result becomes worse for fewer measurement data as we shall see. 

Based on the information obtained by the ADSM, the Bayesian method is employed to improve the reconstructions.
We set $ w|q \sim \mathcal{N}(10^{-4} \mathbf{1}_{N_{x}\times 1},10^{-3}I_{N_{x}\times N_{x}})$, where $\mathbf{1}_{N_{x}\times 1}$ is a $N_{x}$-by-$1$ vector of ones and $I_{N_{x}\times N_{x}}$ is the $N_{x}$-by-$N_{x}$ identity matrix. For each $T_j$, $K=5000$ samples are generated in the MCMC algorithm. The conditional means (CM) are used as the locations of the target.  
% In all the numerical examples, since the target moves slow, we set $\beta=0$ when generating samples in the Bayesian inference.

\subsection{Example 1: Reconstruction of a C-shape path.} 
The exact moving path of a point source is given by (Fig.~\ref{fig: C} (d))
\[
z(t)=\left(1.5 + 3 \cos(4-t), 2+3\sin(2+t), 1.2-4\sin(t/2) \right), \quad t \in [0, T].
\]
We take $\Sigma_{qq}=0.2I_{3\times 3} $ as the covariance of the prior distribution. The prior mean $q_{1,*}$ for $j=1$ is $y_{1}$ which is obtained from the {ADSM} at $T_{1}$. At step $j=2,\cdots,J$, the prior mean $q_{j,*}$ is selected to be the {CM} of the Markov chain $\{q_{j-1}^{k}\}_{k=1}^{K}$ in the previous step, i.e. $q_{j,*}=q_{j-1}$.

The indicator functions $\bm{I}(y, T_1)$ using the measured data on $S_1, S_2, S_3$ are shown in Fig~\ref{fig: C} (a)-(c). The reconstructed path using the {ADSM} and the {ADSM-MCMC} are presented in Fig.~\ref{fig: C}(e)-(h) and (i)-(l), respectively. When the observation data is sufficient (the full aperture case $S_1$), the ADSM yields a satisfactory reconstruction as shown in Fig.~\ref{fig: C}(e). It can be seen that, when the measured data decrease, the reconstructions of both the {ADSM} and the {ADSM-MCMC} deteriorate (the second row (e)-(h) and the third row (i)-(l)). Comparing (f) and (j), (g) and (k), or (h) and (l), the {ADSM-MCMC} significantly outperforms the {ADSM}, in particular, when the measurement data are less. Moreover, the Bayesian inference is robust to noises as indicated by Fig.~\ref{fig: C}(k) and (l). The reconstructions are quite satisfactory for noise levels  $\varepsilon=1\%$ and $10\%$ with $(x,t)\in S_{3}$. 

\begin{figure}[ht]

\subfloat[]{\includegraphics[width = 1.55in]{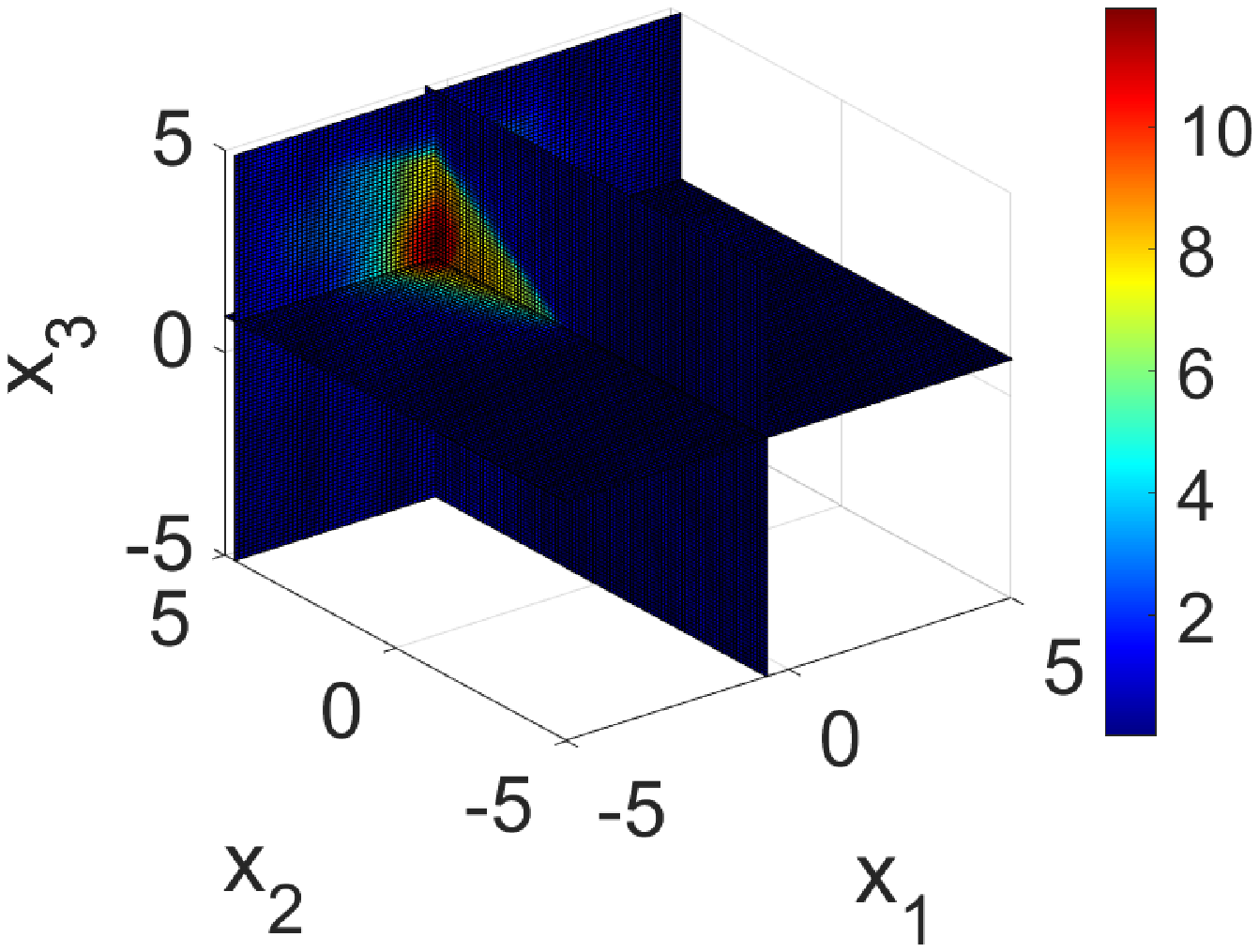}}
\subfloat[]{\includegraphics[width = 1.55in]{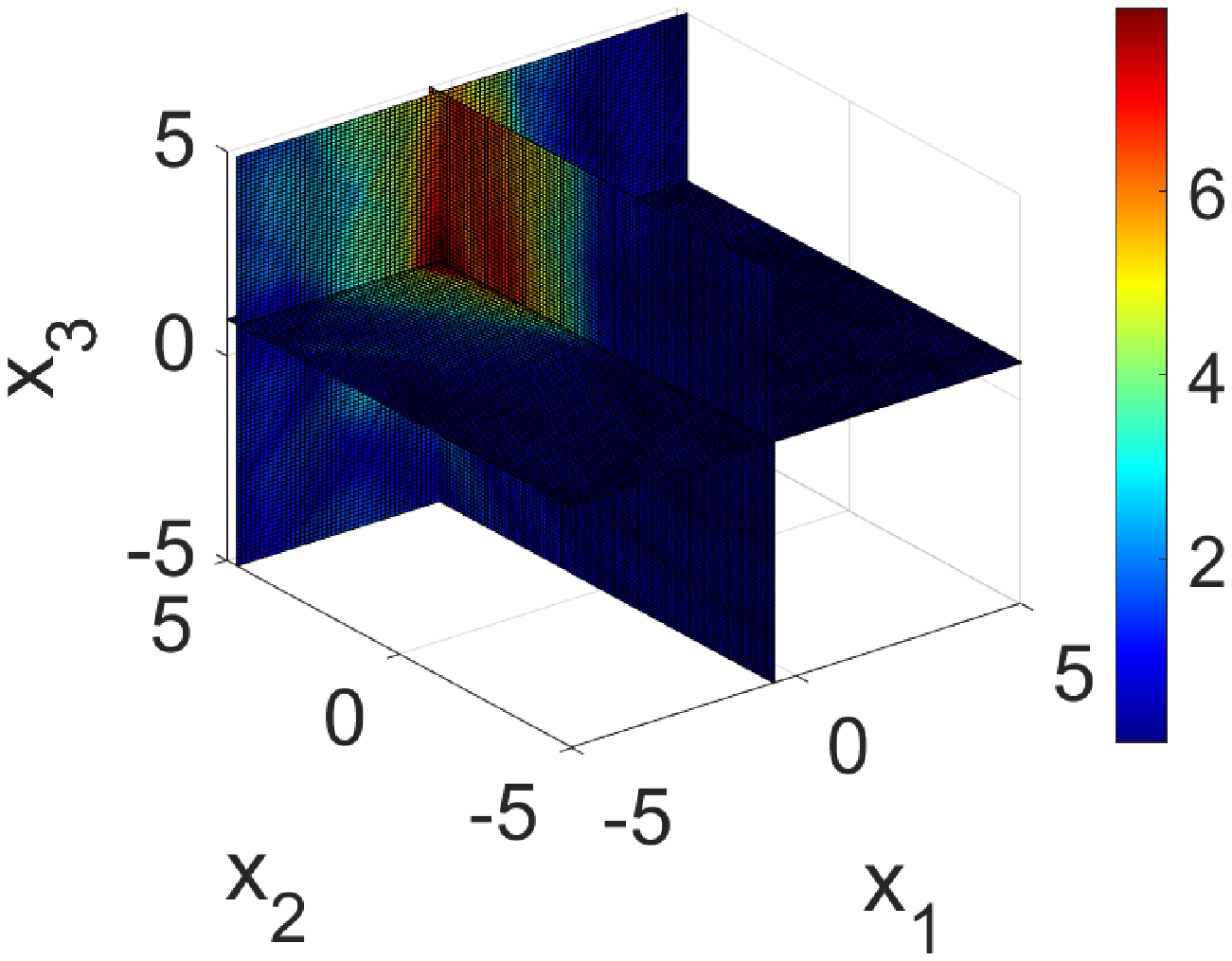}}
\subfloat[]{\includegraphics[width = 1.55in]{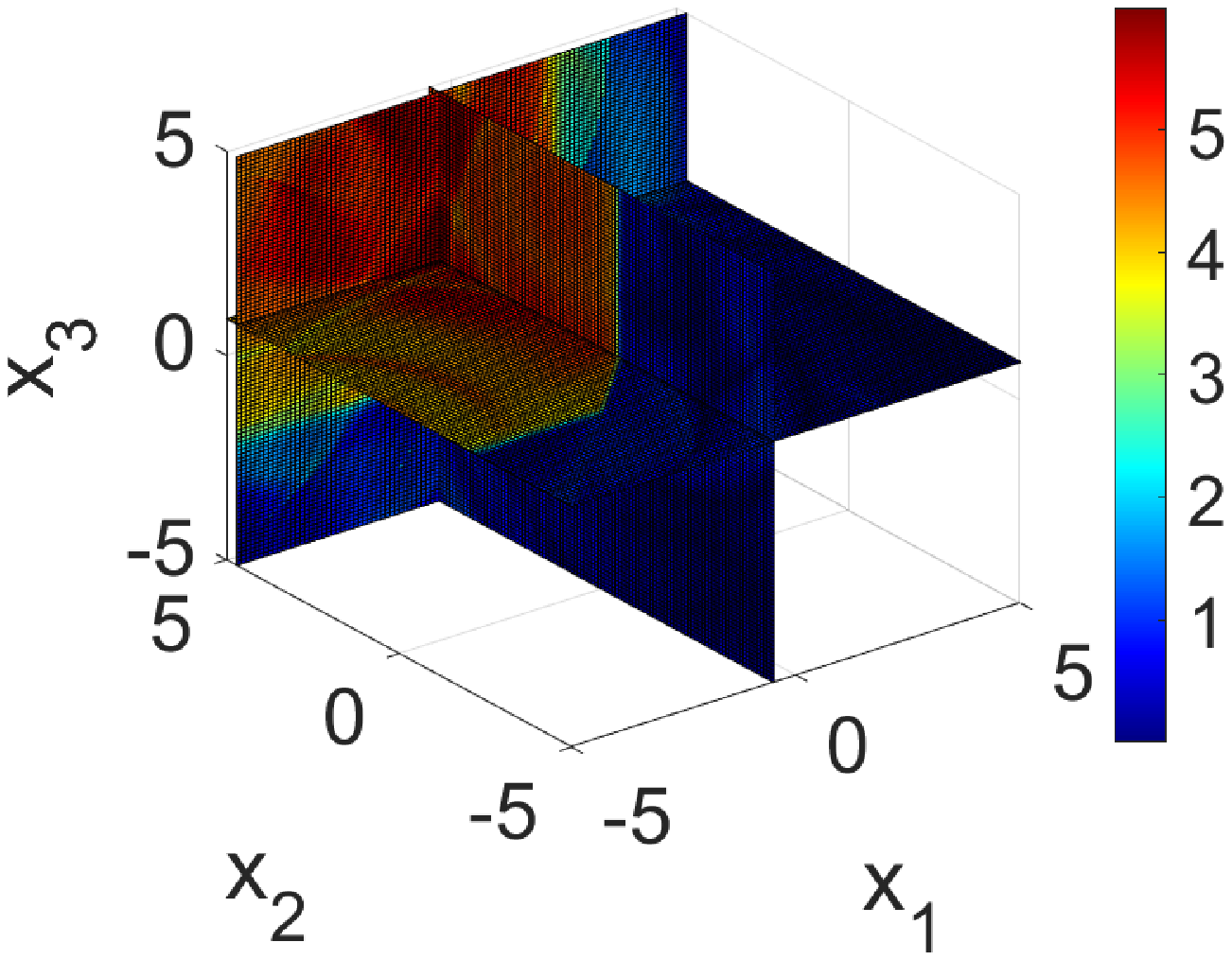}}
\subfloat[]{\includegraphics[width = 1.55in]{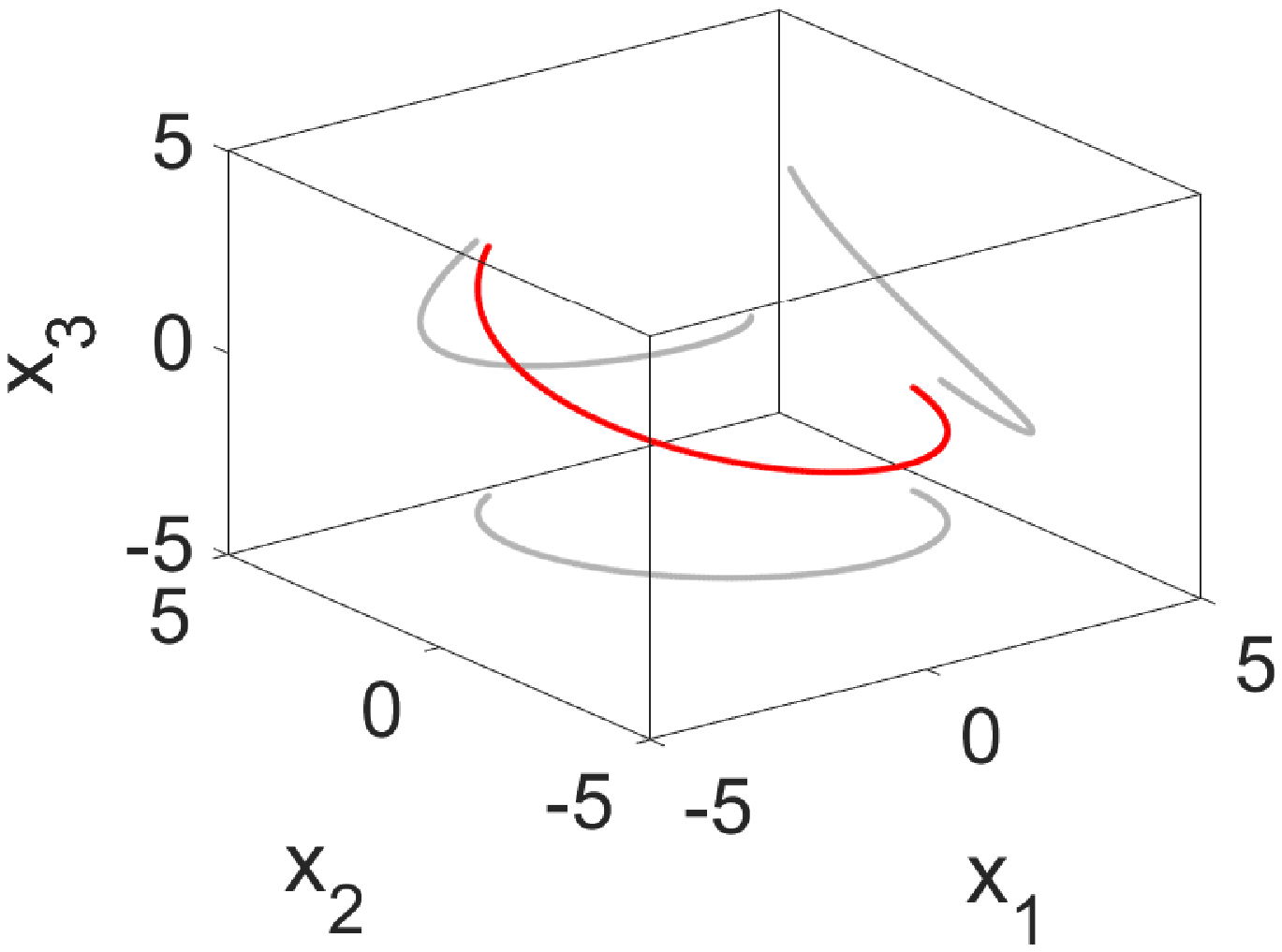}}
\\
\subfloat[]{\includegraphics[width = 1.55in]{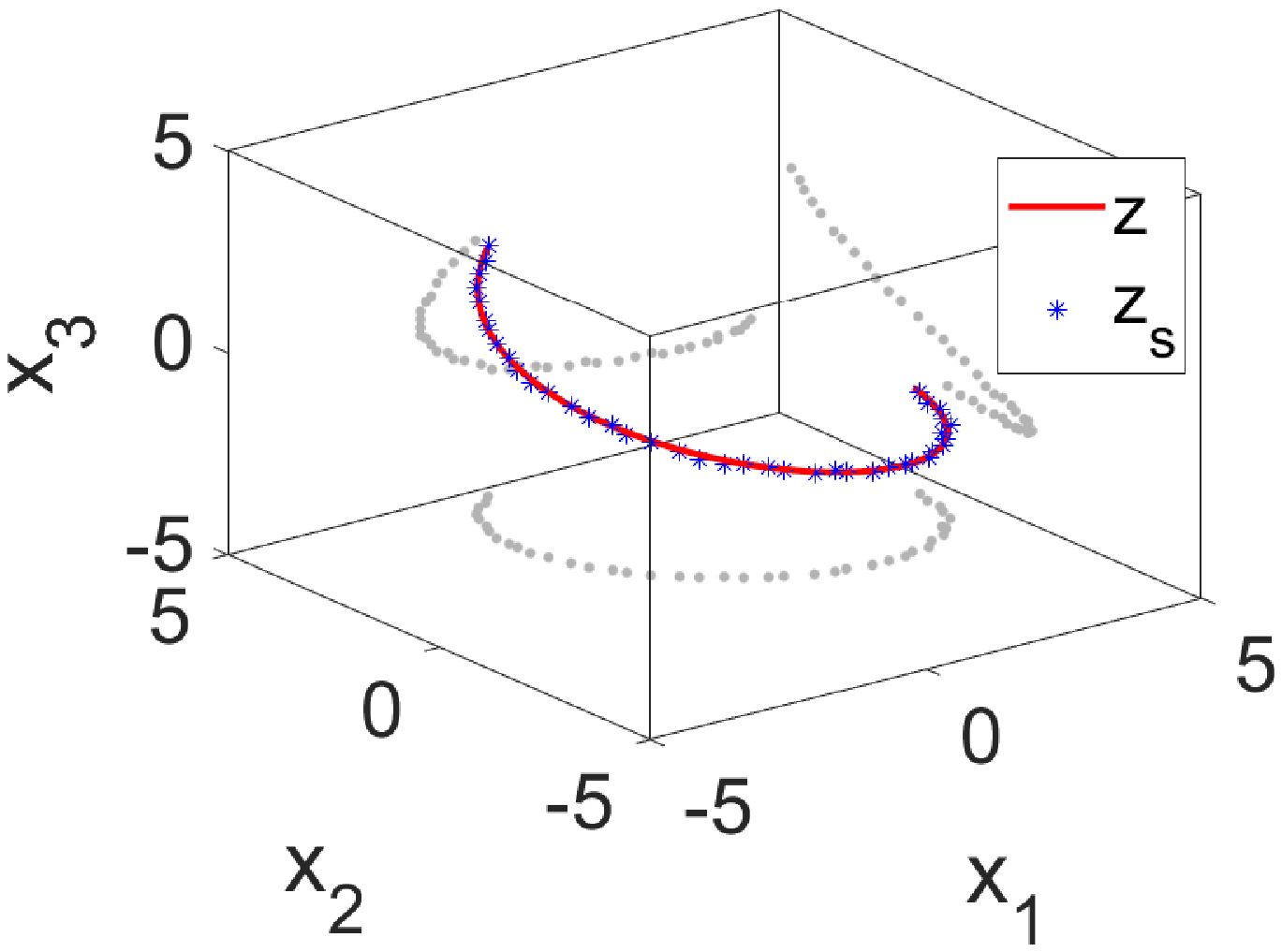}} 
\subfloat[]{\includegraphics[width = 1.55in]{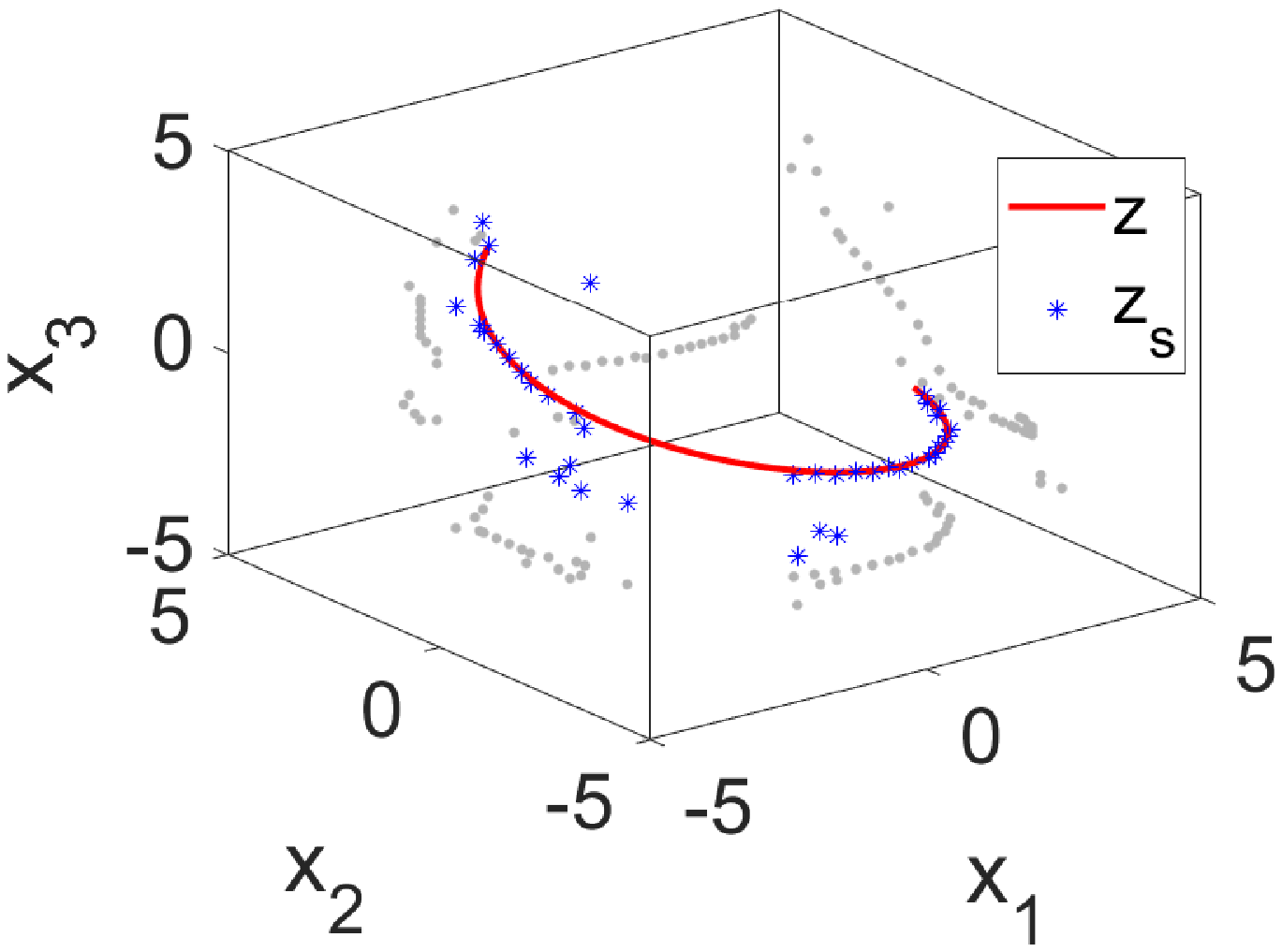}}
\subfloat[]{\includegraphics[width = 1.55in]{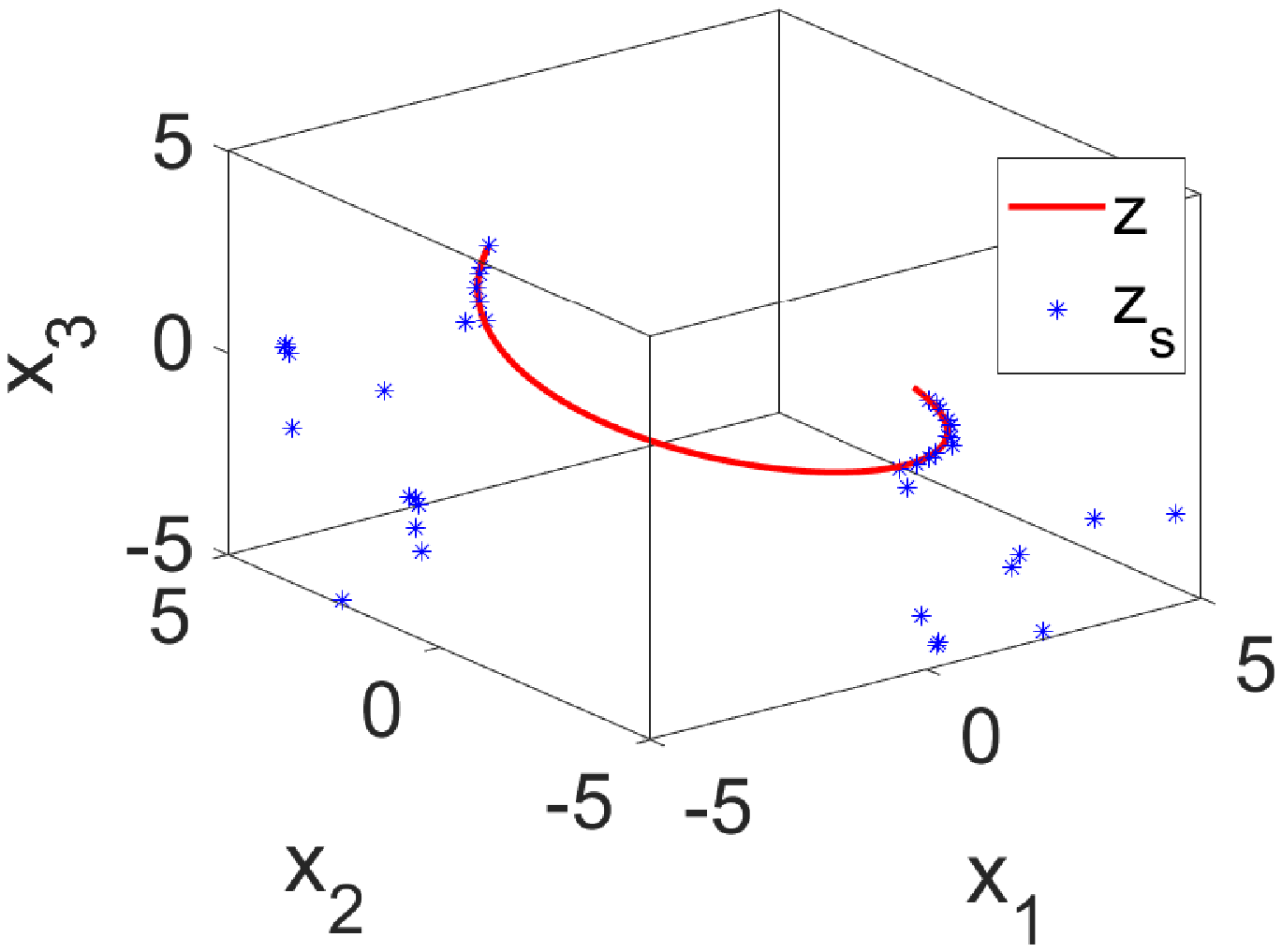}}
\subfloat[]{\includegraphics[width = 1.55in]{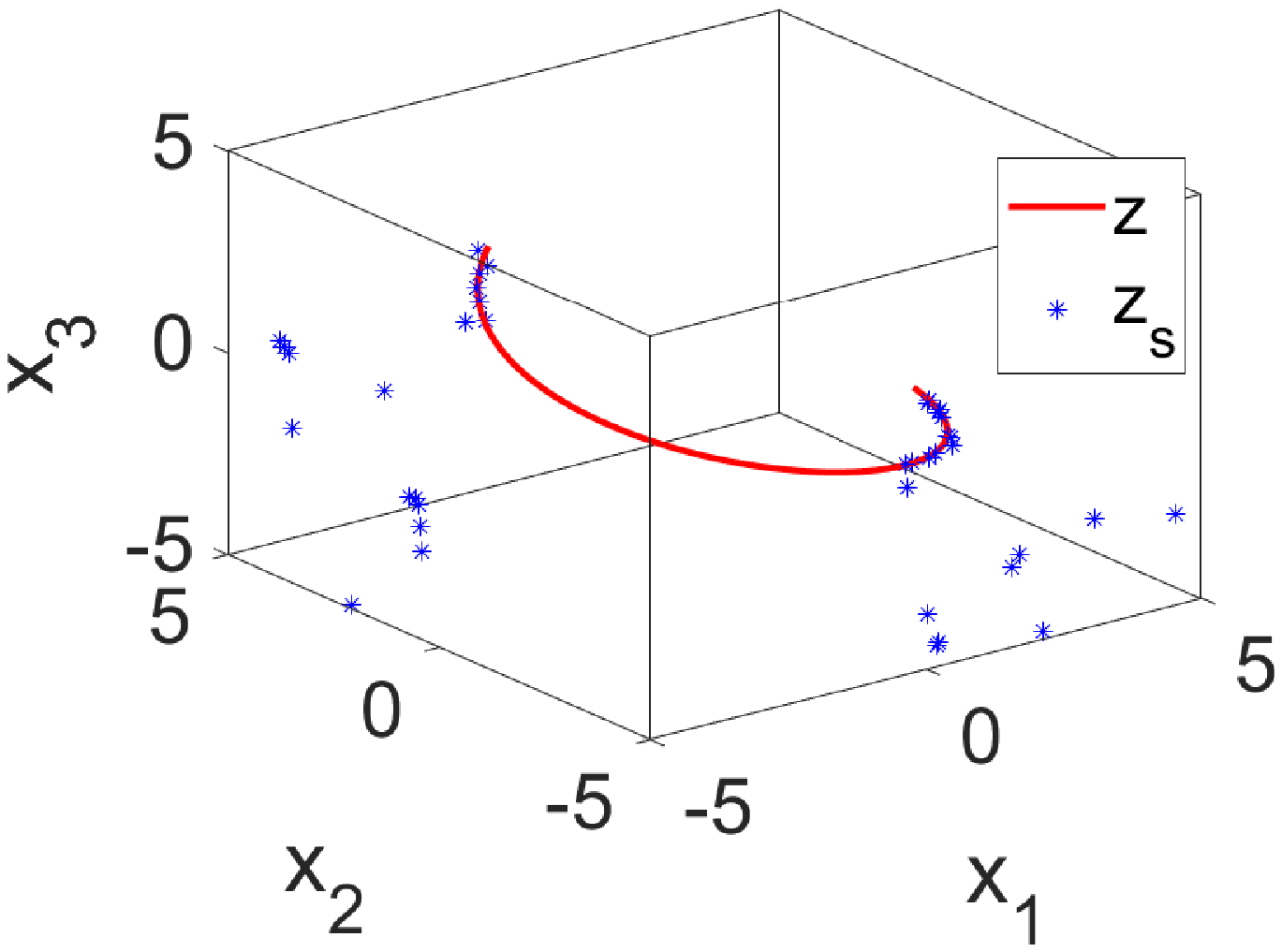}}
\\

\subfloat[]{\includegraphics[width = 1.55in]{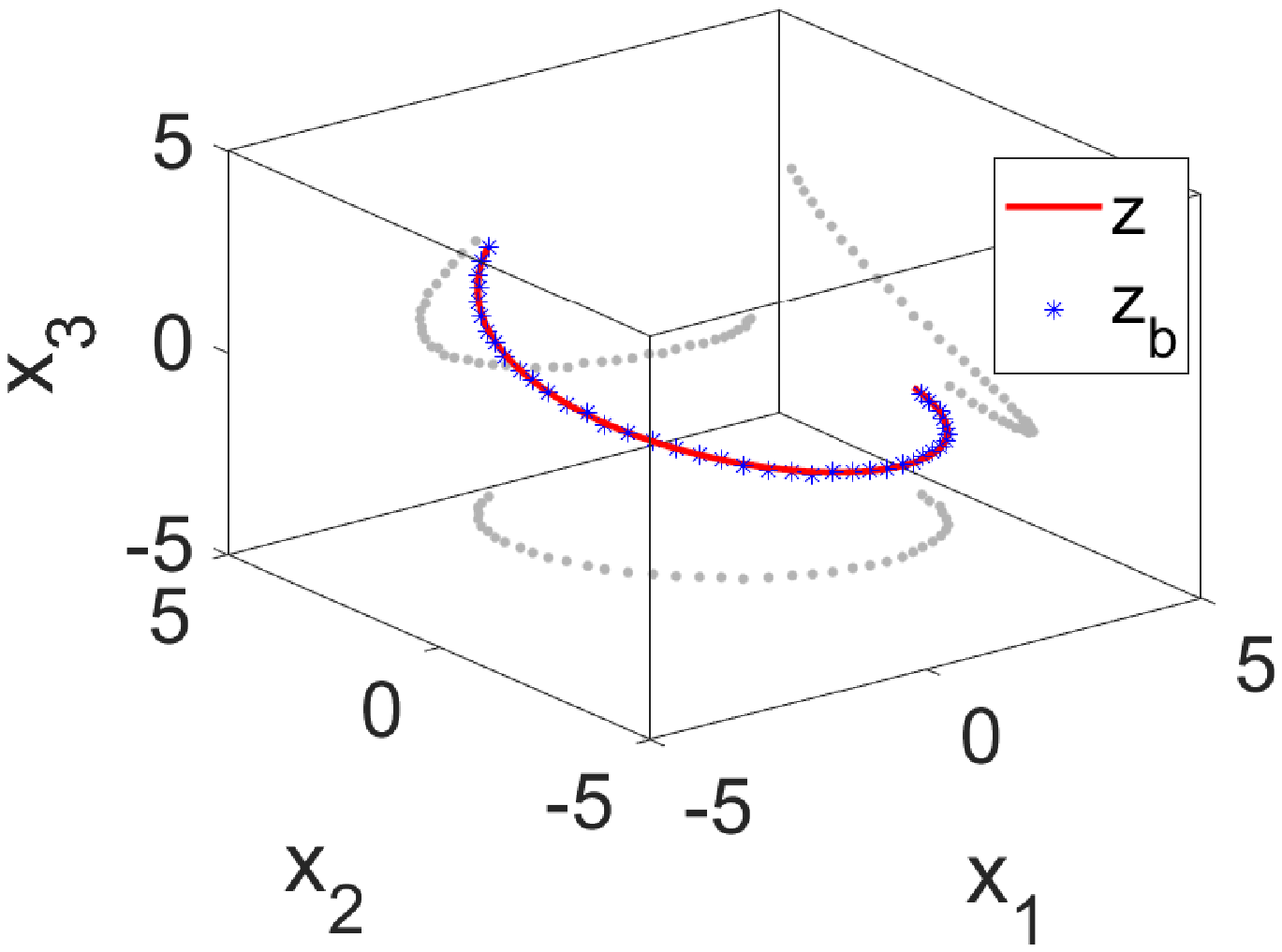}} 
\subfloat[]{\includegraphics[width = 1.55in]{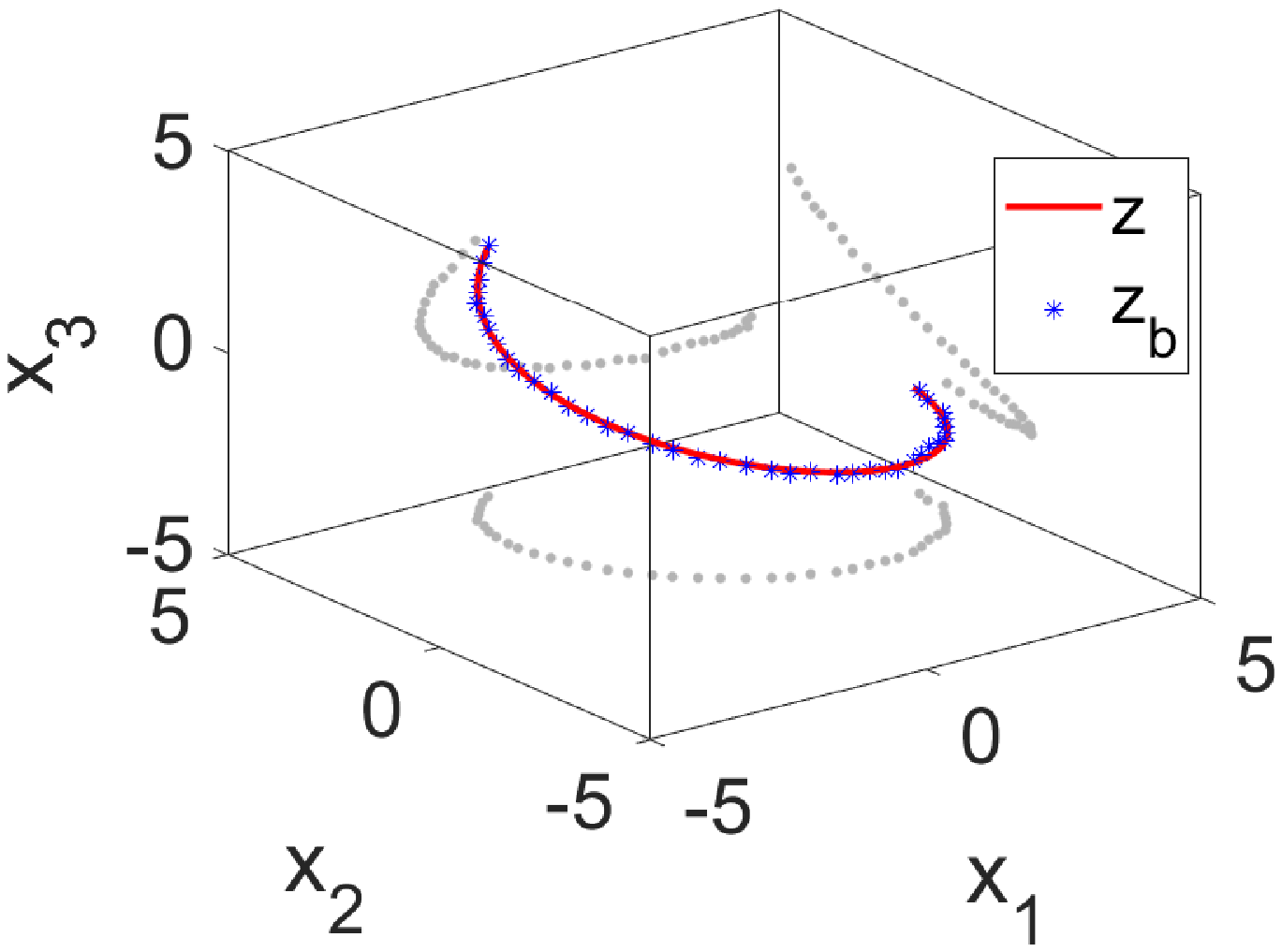}}
\subfloat[]{\includegraphics[width = 1.55in]{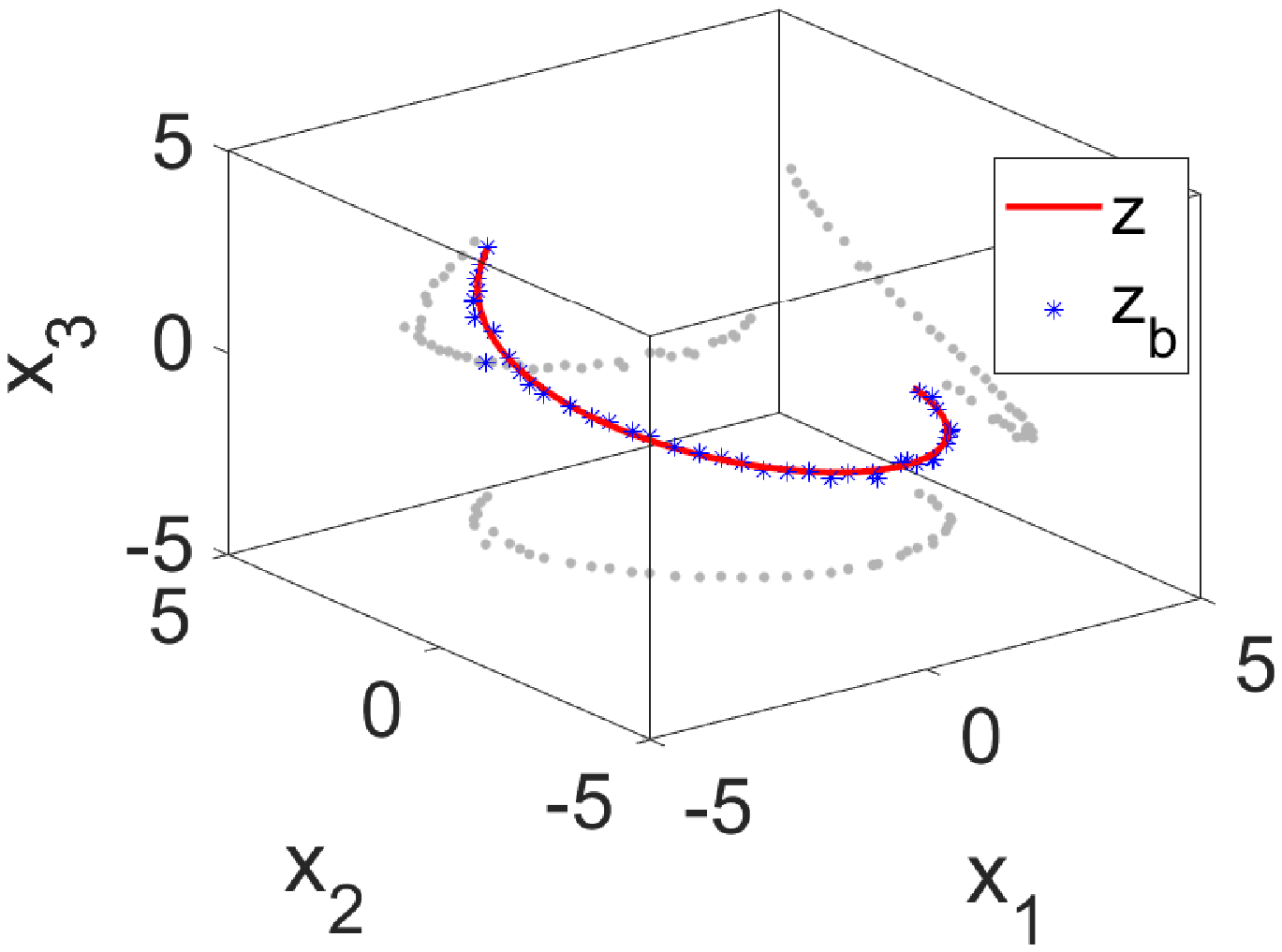}}
\subfloat[]{\includegraphics[width = 1.55in]{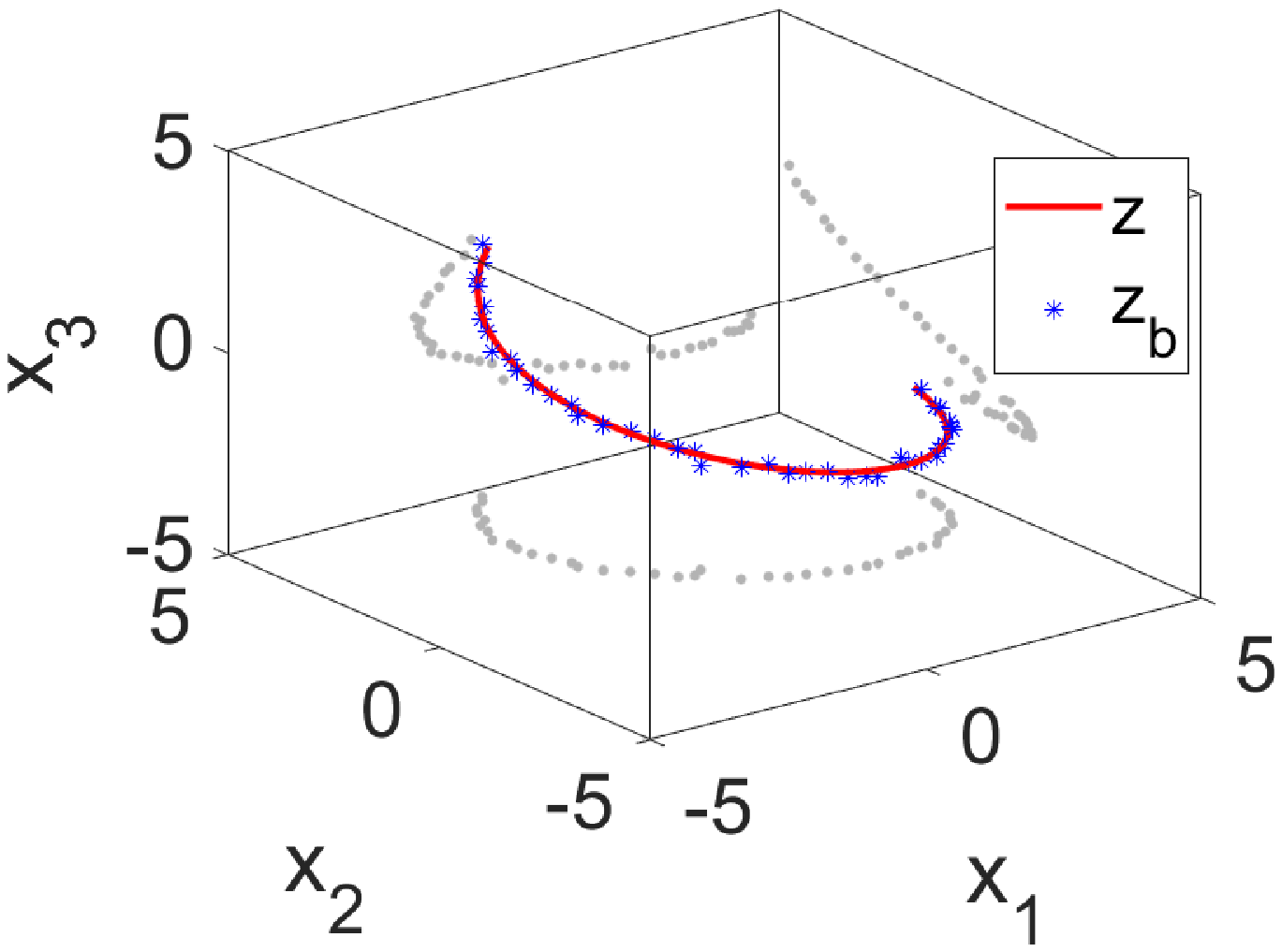}}

\caption{Reconstructions of a C-shape path. (a)-(c) are cross-section plots of the indicator functions $\bm{I}(y,T_1)$. (d) is the exact trajectory $z(t)$. (e)-(g) are the reconstructions by the {ADSM} using data on $S_1, S_2, S_3$ with $\varepsilon =1\%$. (h) is the reconstruction by the ADSM using data on $S_3$ with $\varepsilon =10\%$. (i)-(k) are the reconstructions by the {ADSM-MCMC} using data on $S_1, S_2, S_3$ with $\varepsilon =1\%$. (l) is the reconstruction by the ADSM-MCMC using data on $S_3$ with $\varepsilon =10\%$. $\mathsf{z}$, $\mathsf{z_{s}}$, and $\mathsf{z_{b}}$ represent the exact path, the reconstruction by the {ADSM}, and the reconstruction by the {ADSM-MCMC}, respectively.}
\label{fig: C}
\end{figure}

\subsection{Example 2: Reconstruction of a bow-shape path.} 
We consider the path given by (Fig.~\ref{fig: Bow} (d))
\[
z(t)=\left(3 - 1.6 t, 0.2+2.6\sin(1.25t), -0.3-2.1\sin(1.75t) \right), \quad t \in [0, T].
\]
The covariance $\Sigma_{qq}=0.2I_{3\times 3}$ is used. Fig.~\ref{fig: Bow} show the reconstructions. Again, Fig.~\ref{fig: Bow} (a)-(c) shows the plots of the indicator functions using the measured data for $S_1, S_2, S_3$, respectively.  The recovered moving paths by the {ADSM} are given in Fig.~\ref{fig: Bow} (e)-(h), which suggest that the { ADSM} can provide useful information, but far from accurate when the measured data are partial. In the second step, the Bayesian method refines the paths based on the information obtained by the {ADSM}. Similar to the previous example, as the measurement data become less, the performance of the {ADSM} is getting worse (Fig.~\ref{fig: Bow} (f)-(h)). The results of the {ADSM-MCMC} Fig.~\ref{fig: Bow} (j)-(l) are significantly better than those produced by the {ADSM}. 

%Fig~\ref{fig: Bow}(g)-(h) shows that the reconstruction by the {ADSM} in the second time period $T_{2}$ deviates drastically from the exact location of $z(t)$.
%Nevertheless, Bayesian approach is capable of reconstructing the overall correct path of $z(t)$ even with limited noisy measurements. \textcolor{blue}{Yanfang: It is not clear here. You need to provide more results here. What does it mean by "takes some time"?}
\begin{figure}[ht]

\subfloat[]{\includegraphics[width = 1.55in]{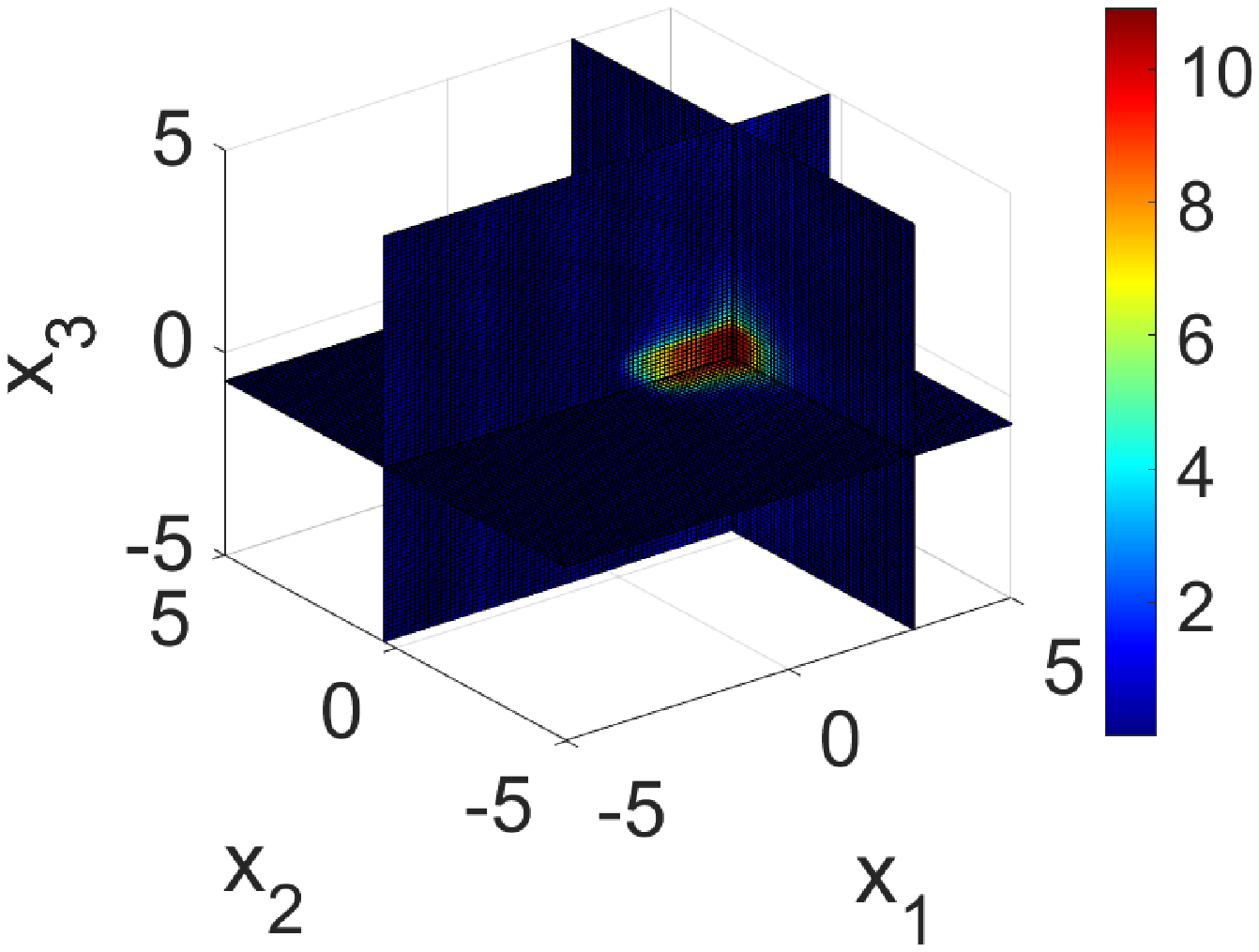}}
\subfloat[]{\includegraphics[width = 1.55in]{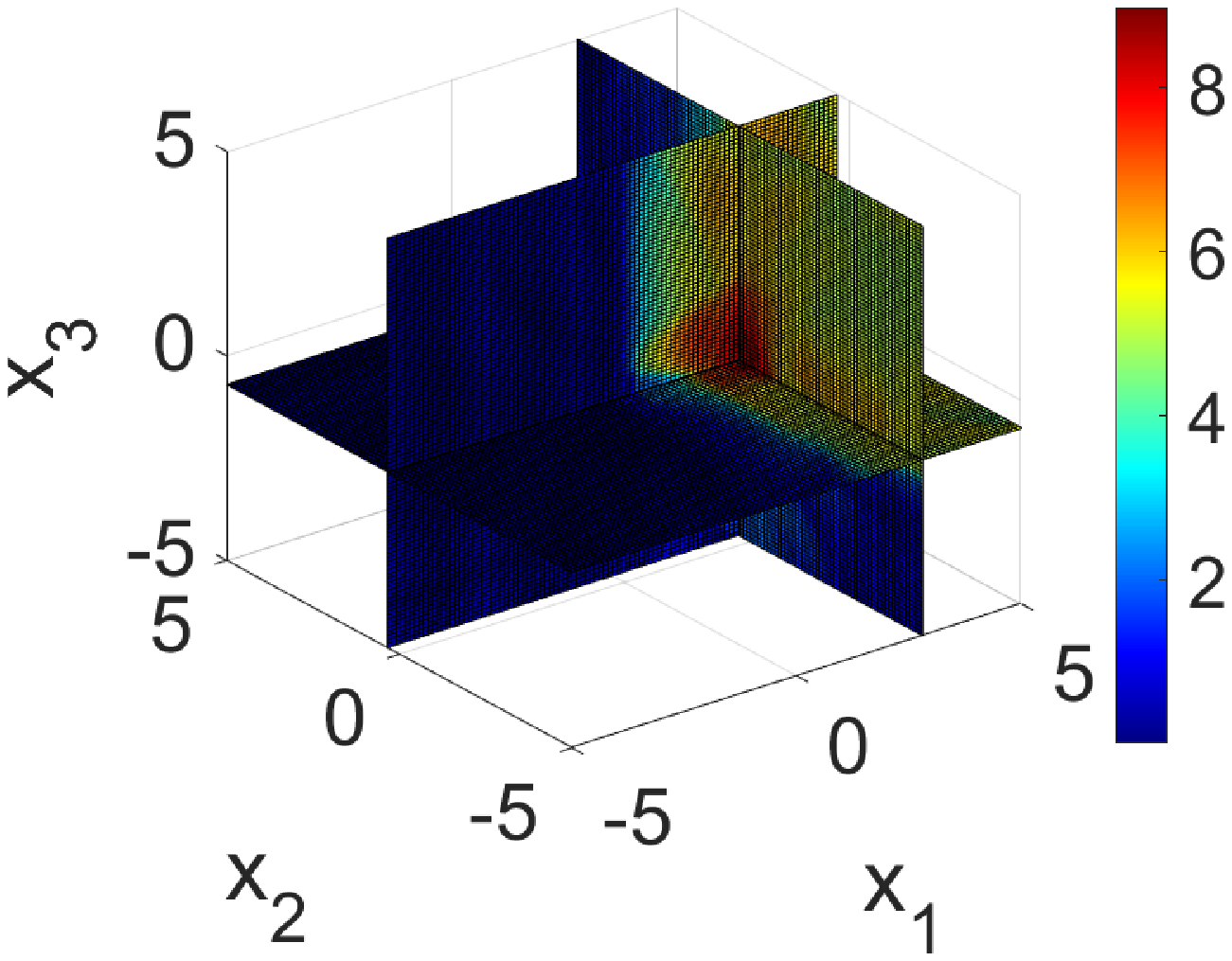}}
\subfloat[]{\includegraphics[width = 1.55in]{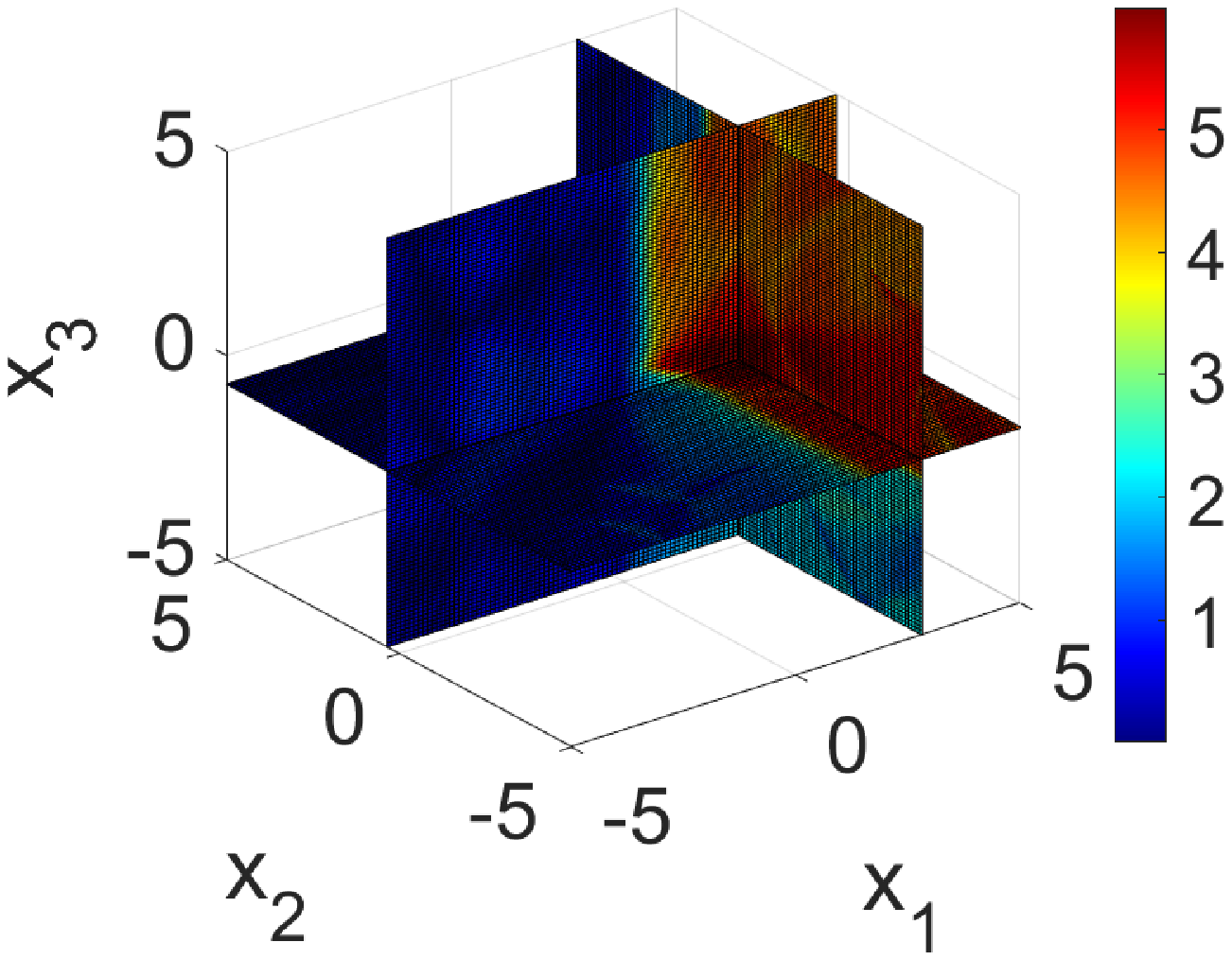}}
\subfloat[]{\includegraphics[width = 1.55in]{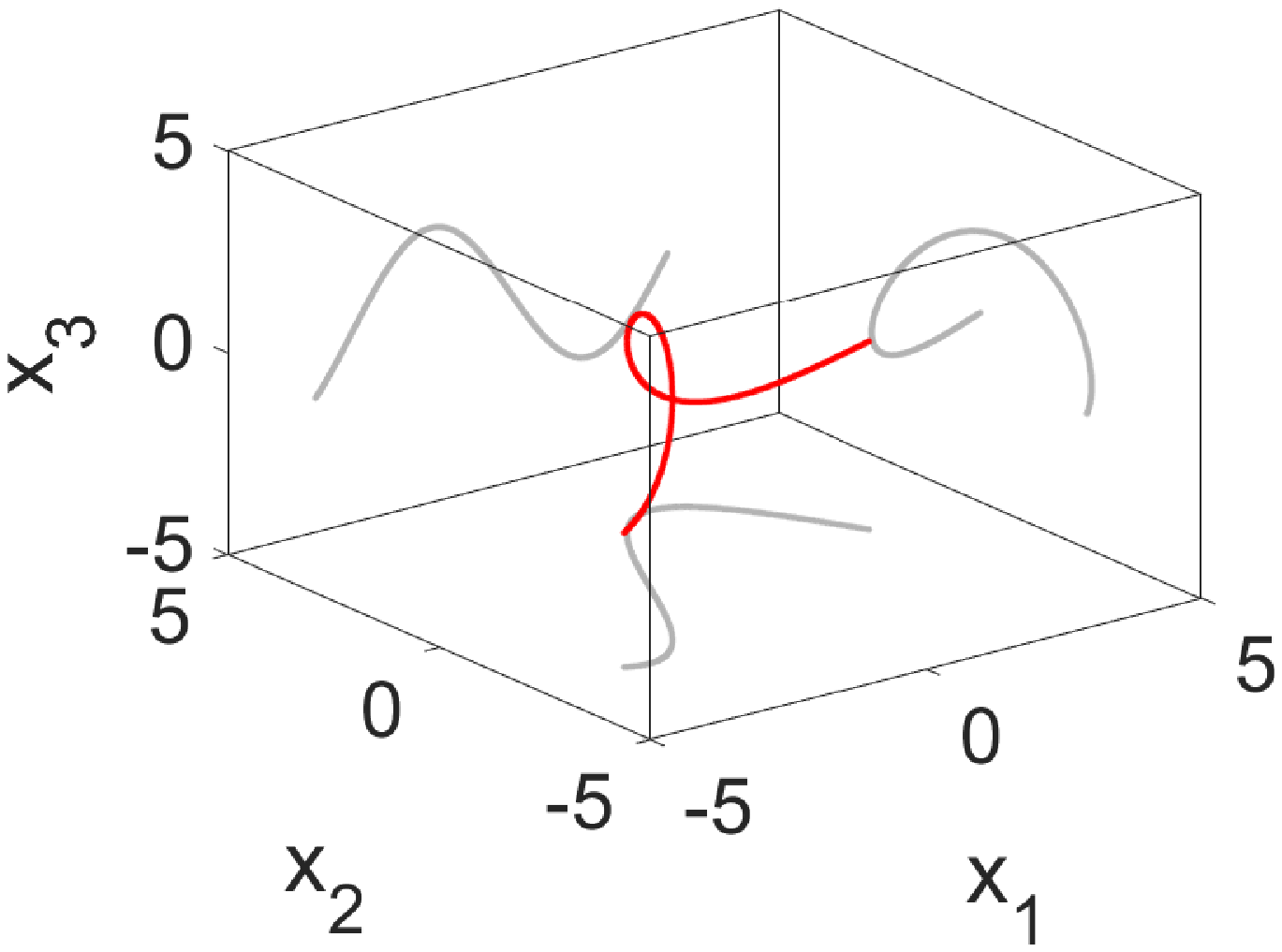}} 
\\
\subfloat[]{\includegraphics[width = 1.55in]{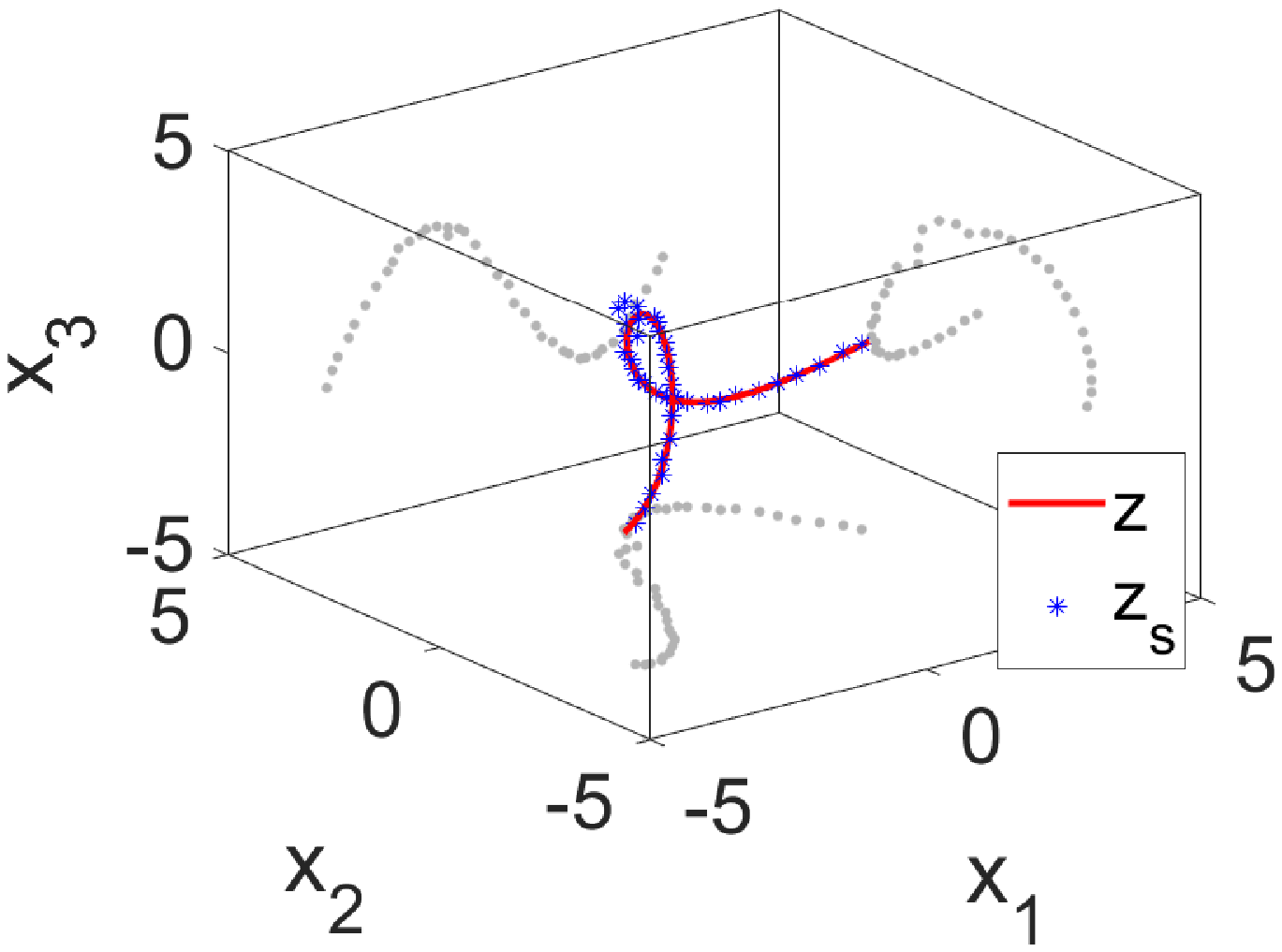}} 
\subfloat[]{\includegraphics[width = 1.55in]{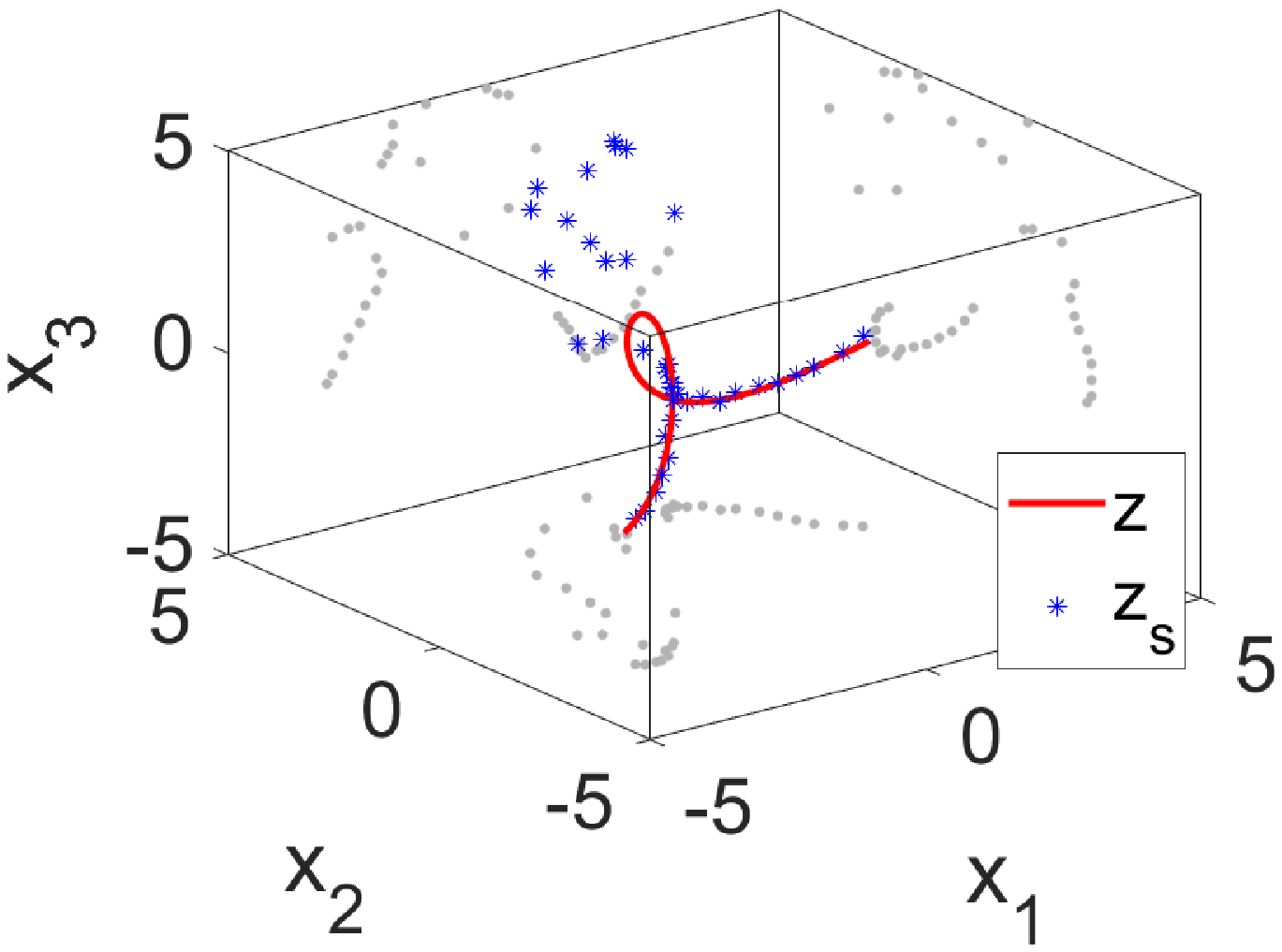}}
\subfloat[]{\includegraphics[width = 1.55in]{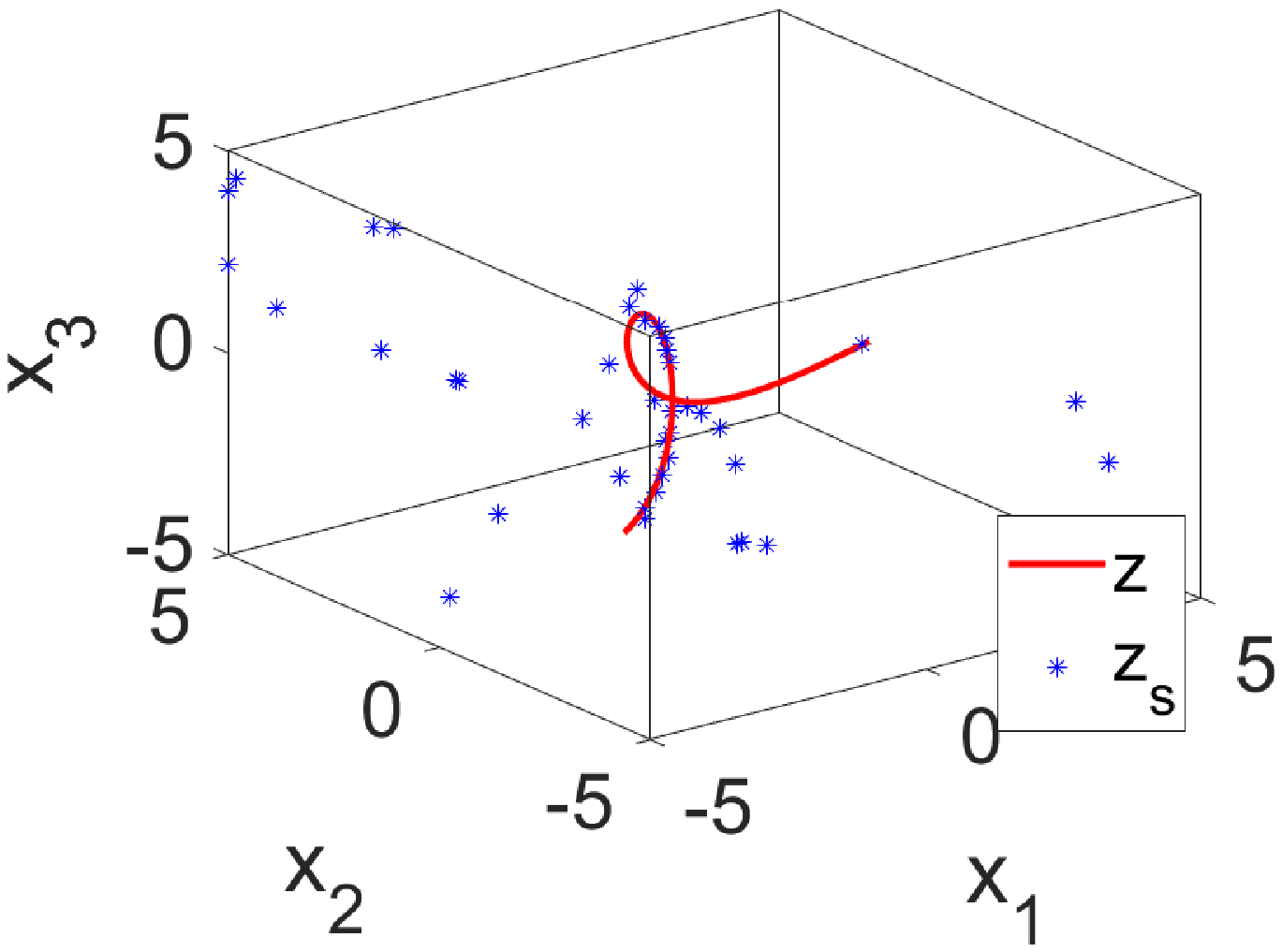}}
\subfloat[]{\includegraphics[width = 1.55in]{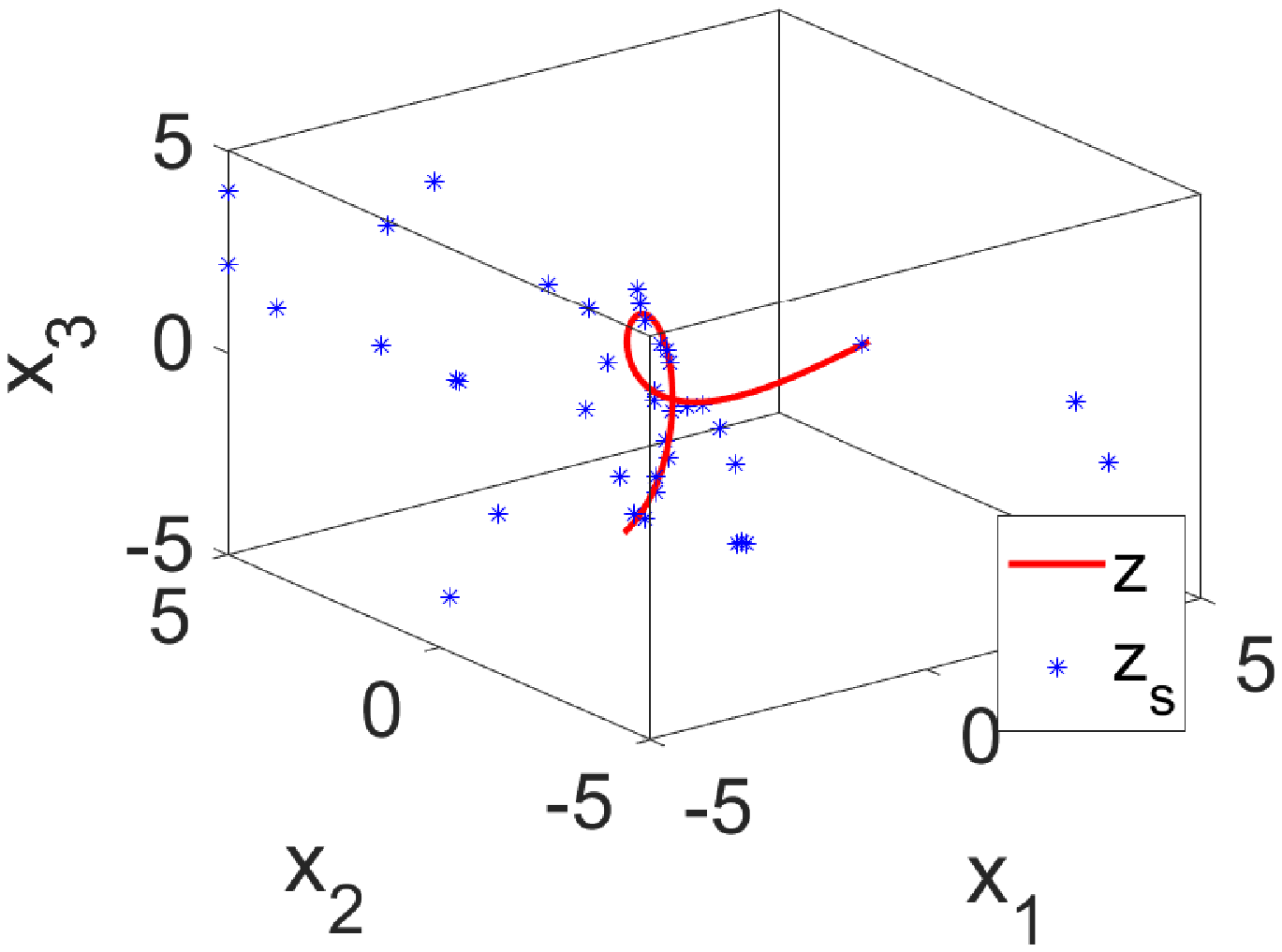}}
\\
\subfloat[]{\includegraphics[width = 1.55in]{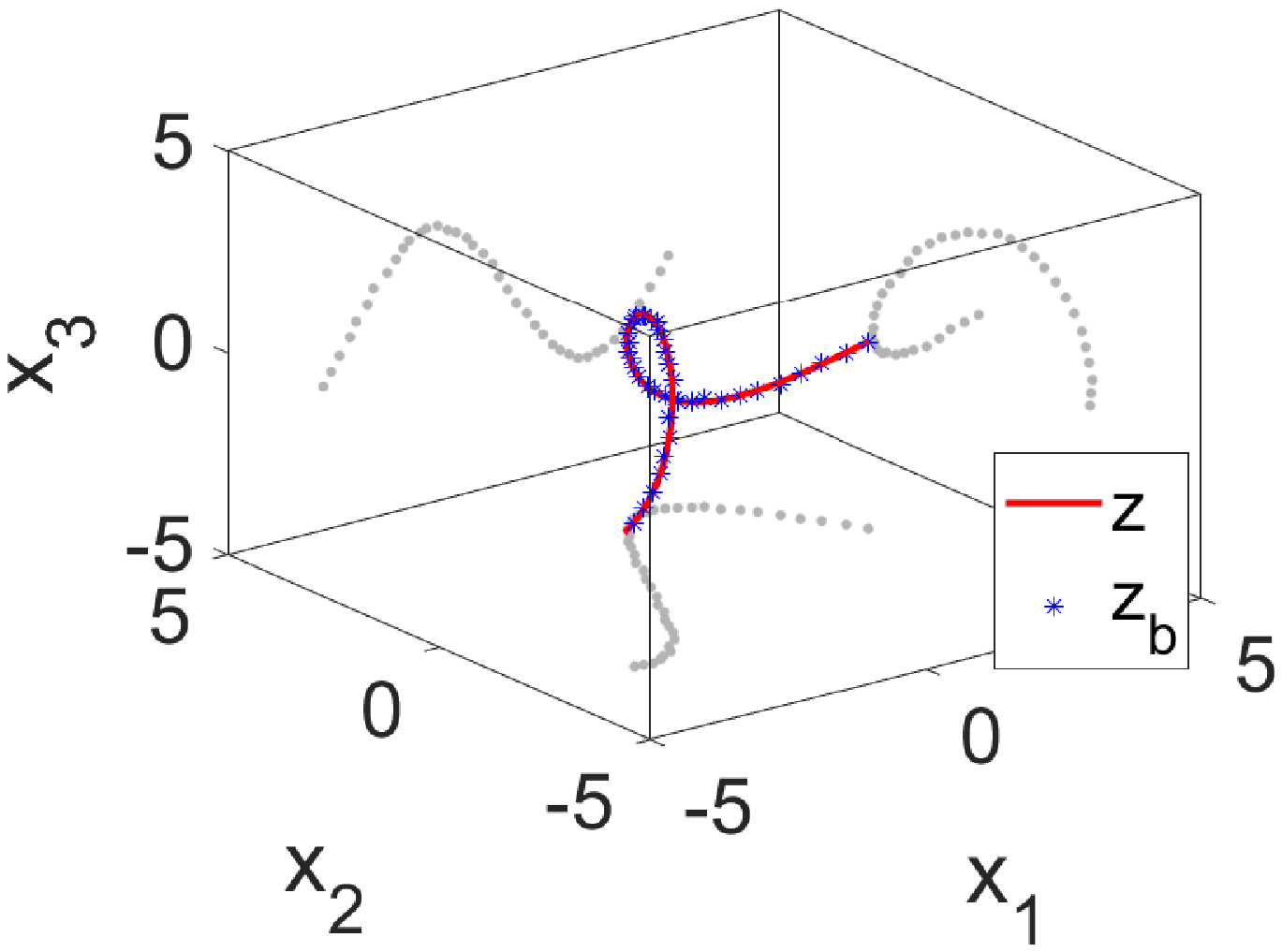}} 
\subfloat[]{\includegraphics[width = 1.55in]{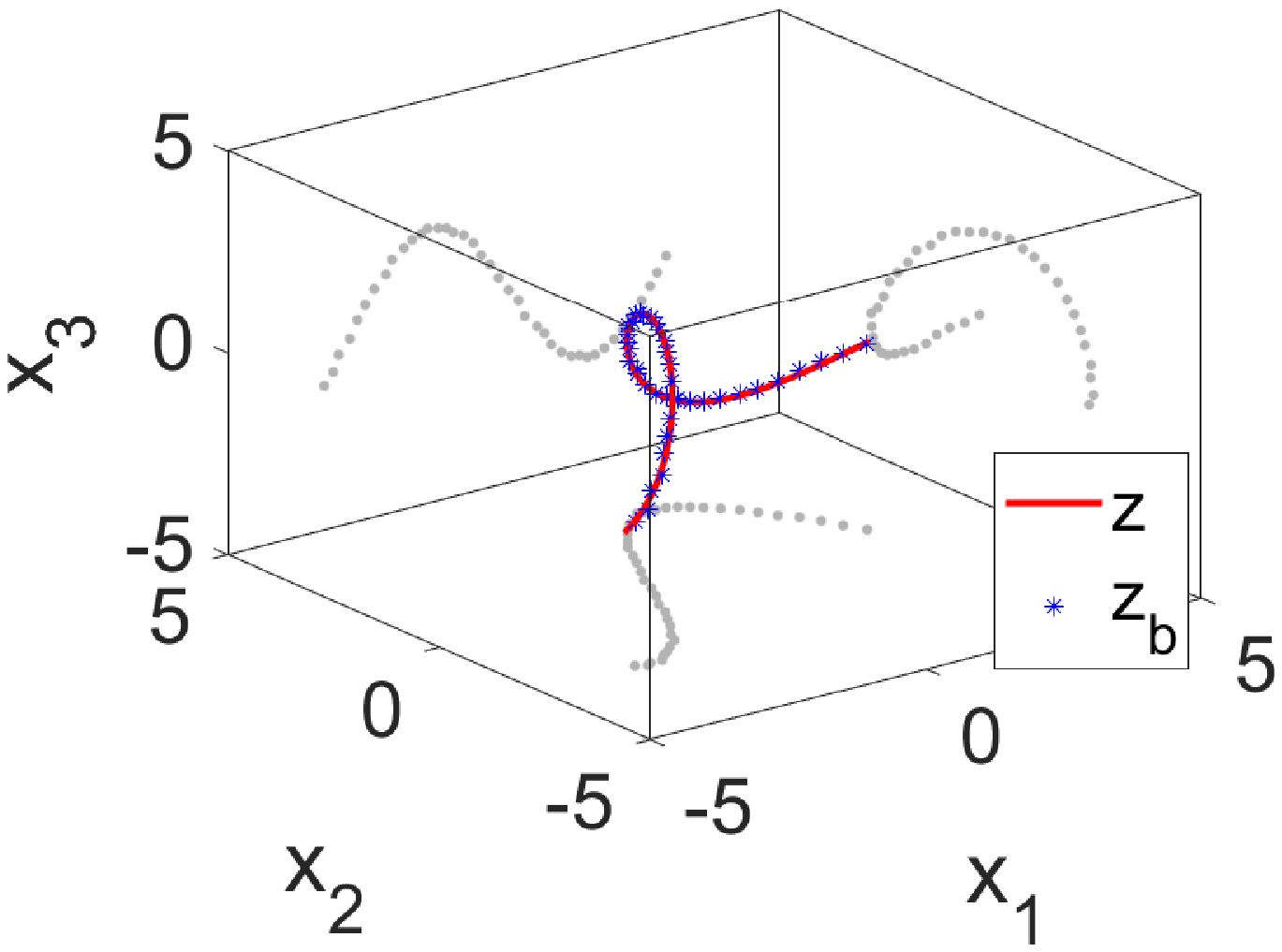}}
\subfloat[]{\includegraphics[width = 1.55in]{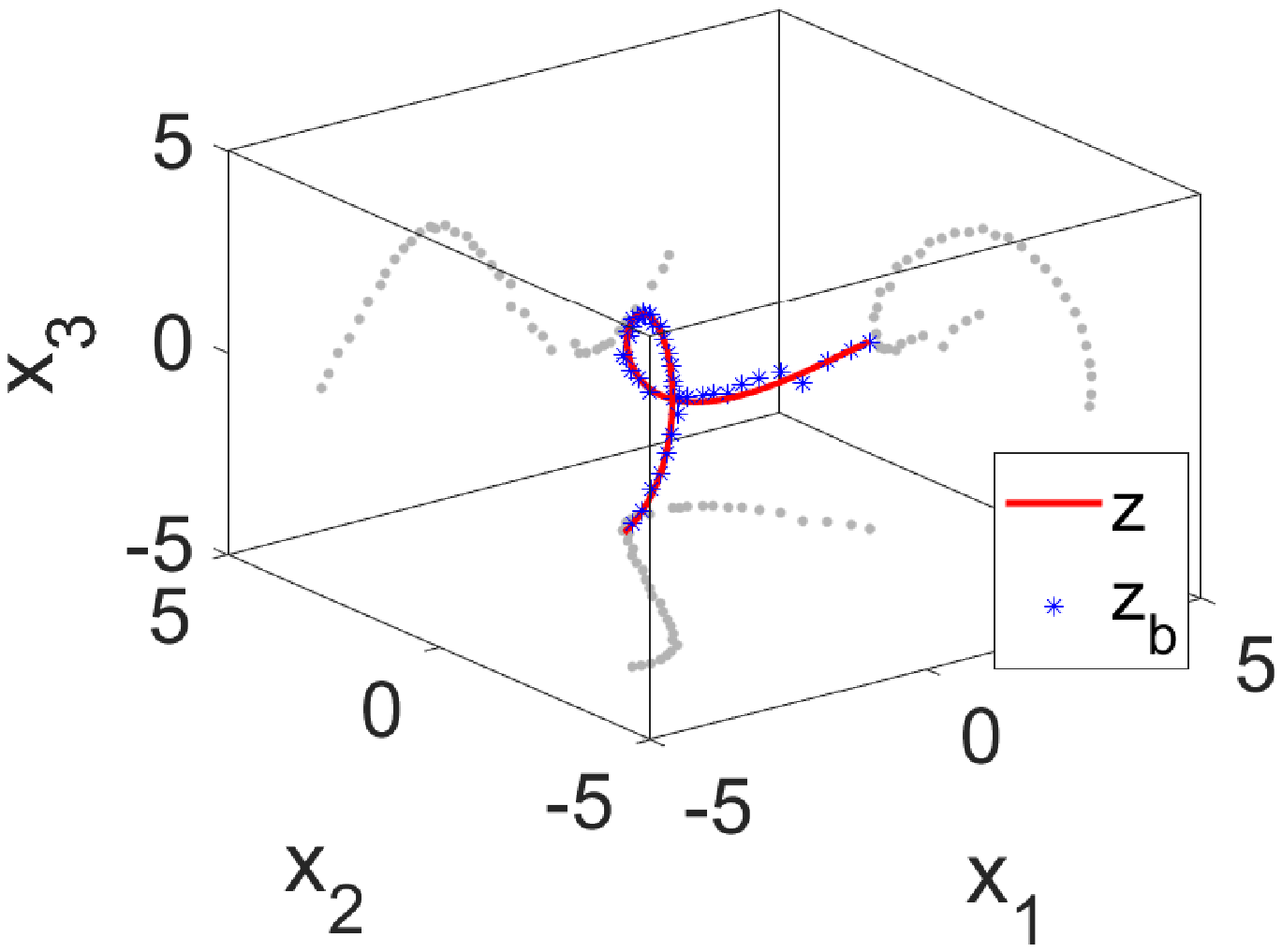}}
\subfloat[]{\includegraphics[width = 1.55in]{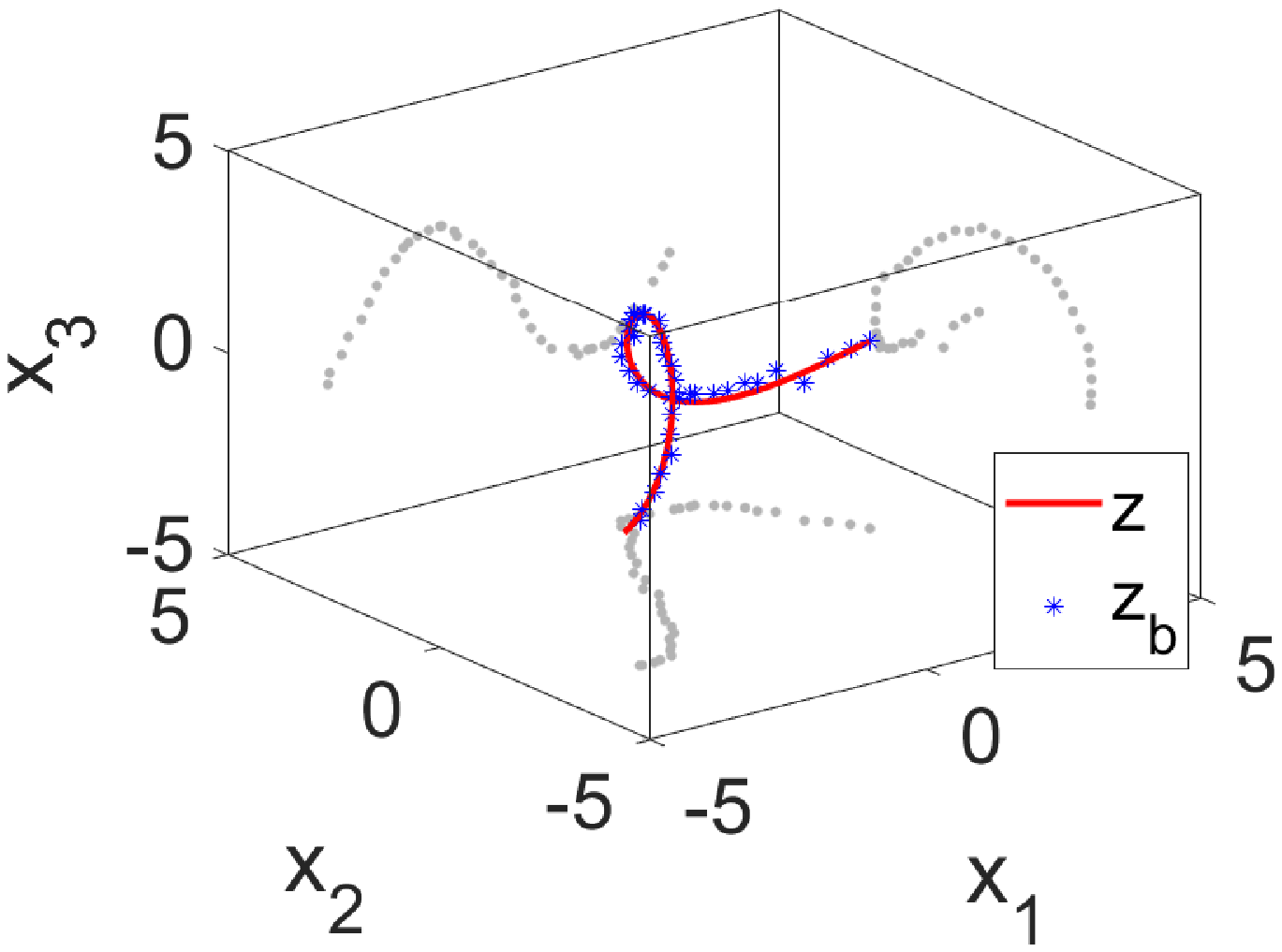}}
\caption{Reconstruction of a bow-shape path. (a)-(c) are the cross-section plots of the indicator functions $\bm{I}(y,T_1)$. (d) is the exact trajectory $z(t)$. (e)-(g) are the reconstructions by the {ADSM} using data on $S_1, S_2, S_3$ with $\varepsilon =1\%$. (h) is the reconstruction by the ADSM using data on $S_3$ with $\varepsilon =10\%$. (i)-(k) are the reconstructions by the {ADSM-MCMC} using data on $S_1, S_2, S_3$ with $\varepsilon =1\%$. (l) is the reconstruction by the ADSM-MCMC using data on $S_3$ with $\varepsilon =10\%$. $\mathsf{z}$,$\mathsf{z_{s}}$, and $\mathsf{z_{b}}$ represent the exact path, the reconstruction by the {ADSM}, and the reconstruction by the {ADSM-MCMC}, respectively.}
\label{fig: Bow}
\end{figure}

\subsection{Example 3: Simultaneous reconstruction of two paths} 

We consider the case of two moving point sources which are well-separated. Assume that the signal functions of these sources are the same and their exact moving paths are given by (Fig.~\ref{fig: Two_curves} (d))
\begin{align*}
    z_1(t)= & (2-2\cos(4-0.5t), 1+3\sin(2+t), 2),\\
    z_2(t)= & (-4, -3+1.3t, 1.5).
\end{align*}

Since \eqref{eq:wave_equation} is linear with respect to the sources, the proposed method works similar to the case of a single point source. To be precise, the source term of \eqref{eq:wave_equation} is the superposition $\lambda (t)\delta(x-z_1(t)) +\lambda (t)\delta(x-z_2(t))$. Correspondingly, the exact solution to \eqref{wave equation} is now given by 
\[
u(x, t; z(t)) = u_1(x, t; z_1(t)) +u_2(x, t; z_2(t)),
\]
where $u_1$ and $u_2$ are given by \eqref{exact solution 1} with $z(t)$ replaced by $z_1(t)$ and $z_2(t)$, respectively.

The covariance is chosen to be $\Sigma_{qq}=0.4I_{3\times 3}$. The {ADSM} reconstructs two local maximums in Fig.~\ref{fig: Two_curves} (a)-(c), indicating the presence of two point sources. Similar to the previous two examples, when the measured data are complete or almost complete, the {ADSM} can provide satisfactory reconstructions. The performance is getting worse as the data become less (Fig.~\ref{fig: Two_curves}(f)-(h)). The {ADSM-MCMC} improves the reconstructions significantly (Fig.~\ref{fig: Two_curves} (j)-(l)). 

\begin{figure}[ht]
\subfloat[]{\includegraphics[width = 1.55in]{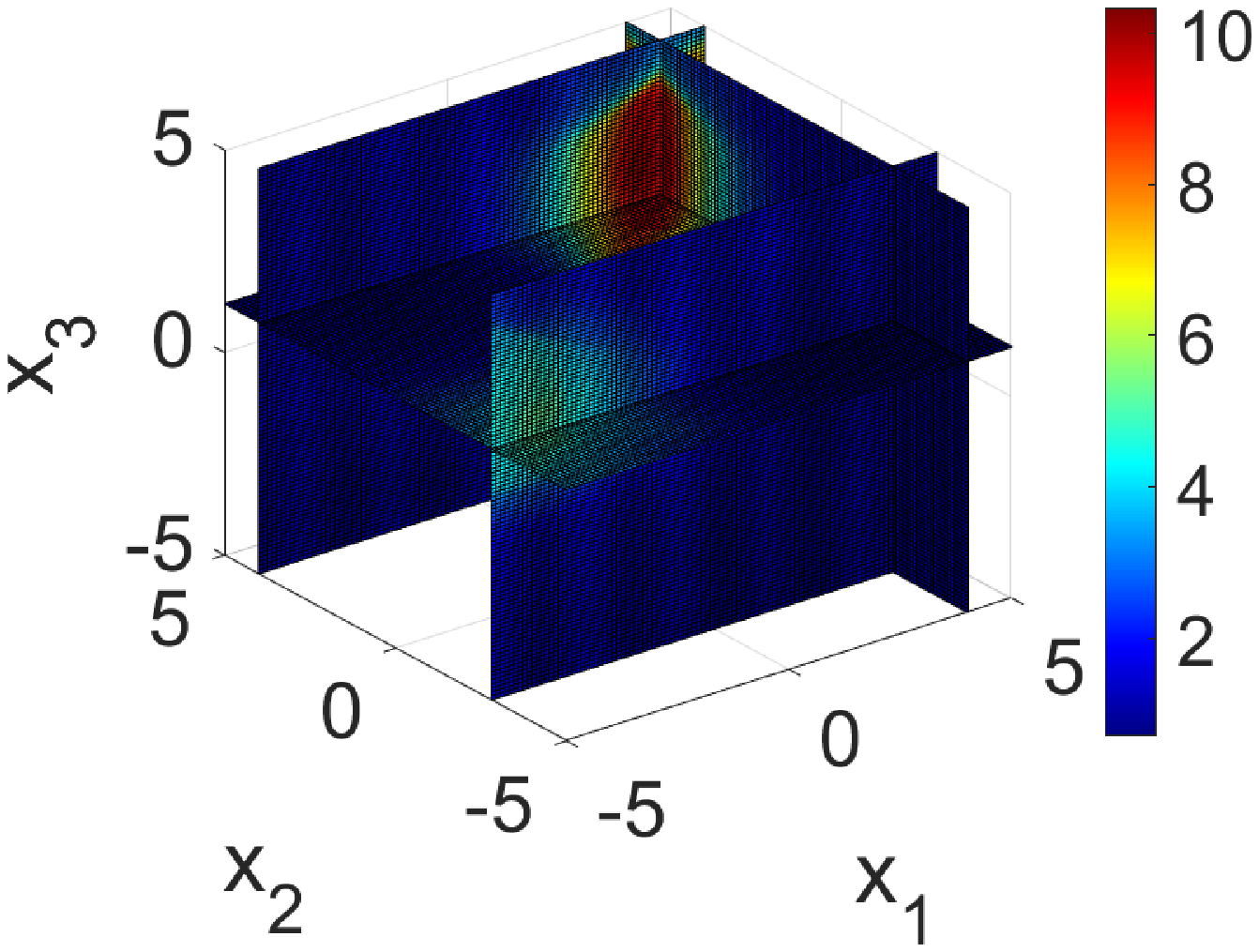}}
\subfloat[]{\includegraphics[width = 1.55in]{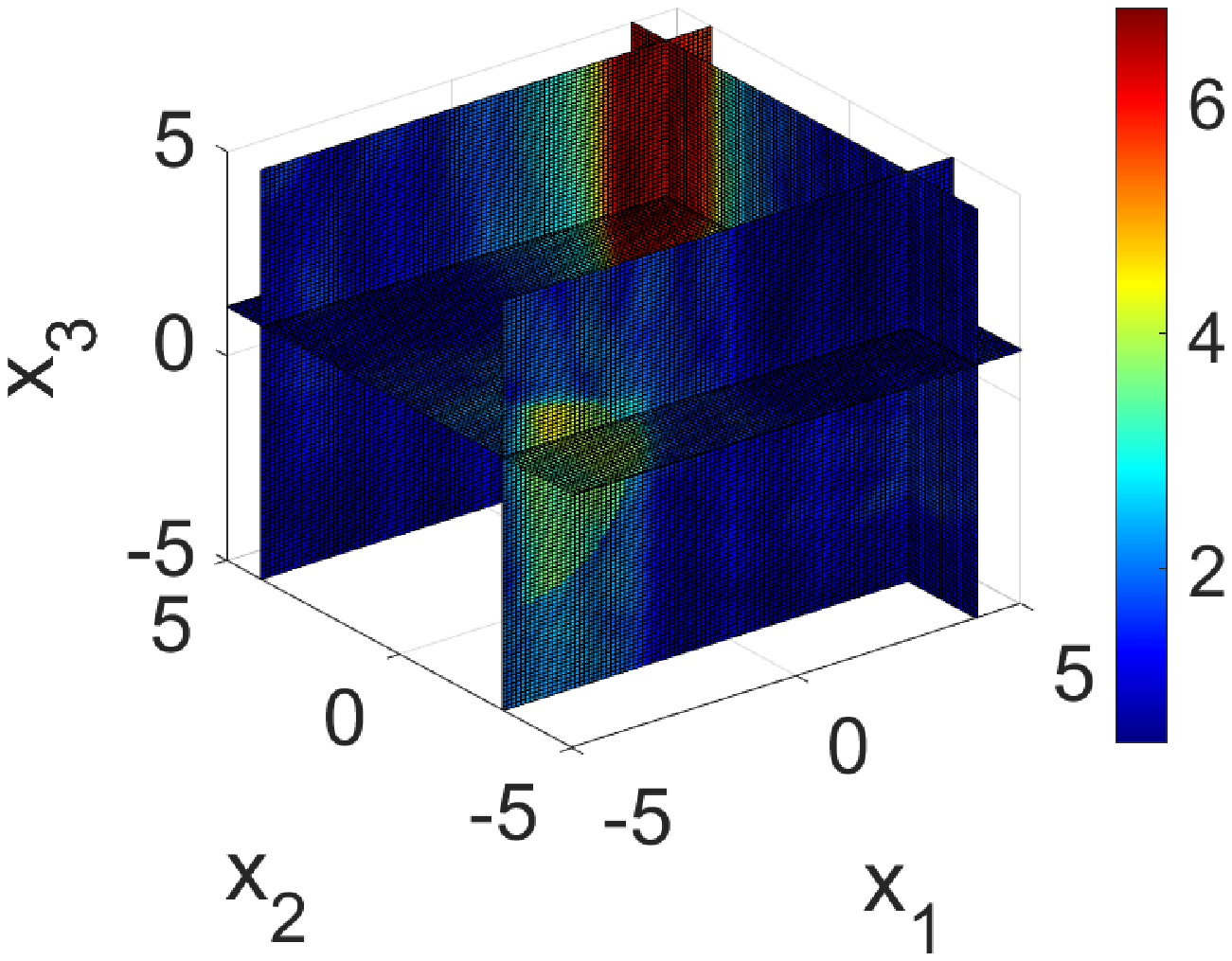}}
\subfloat[]{\includegraphics[width = 1.55in]{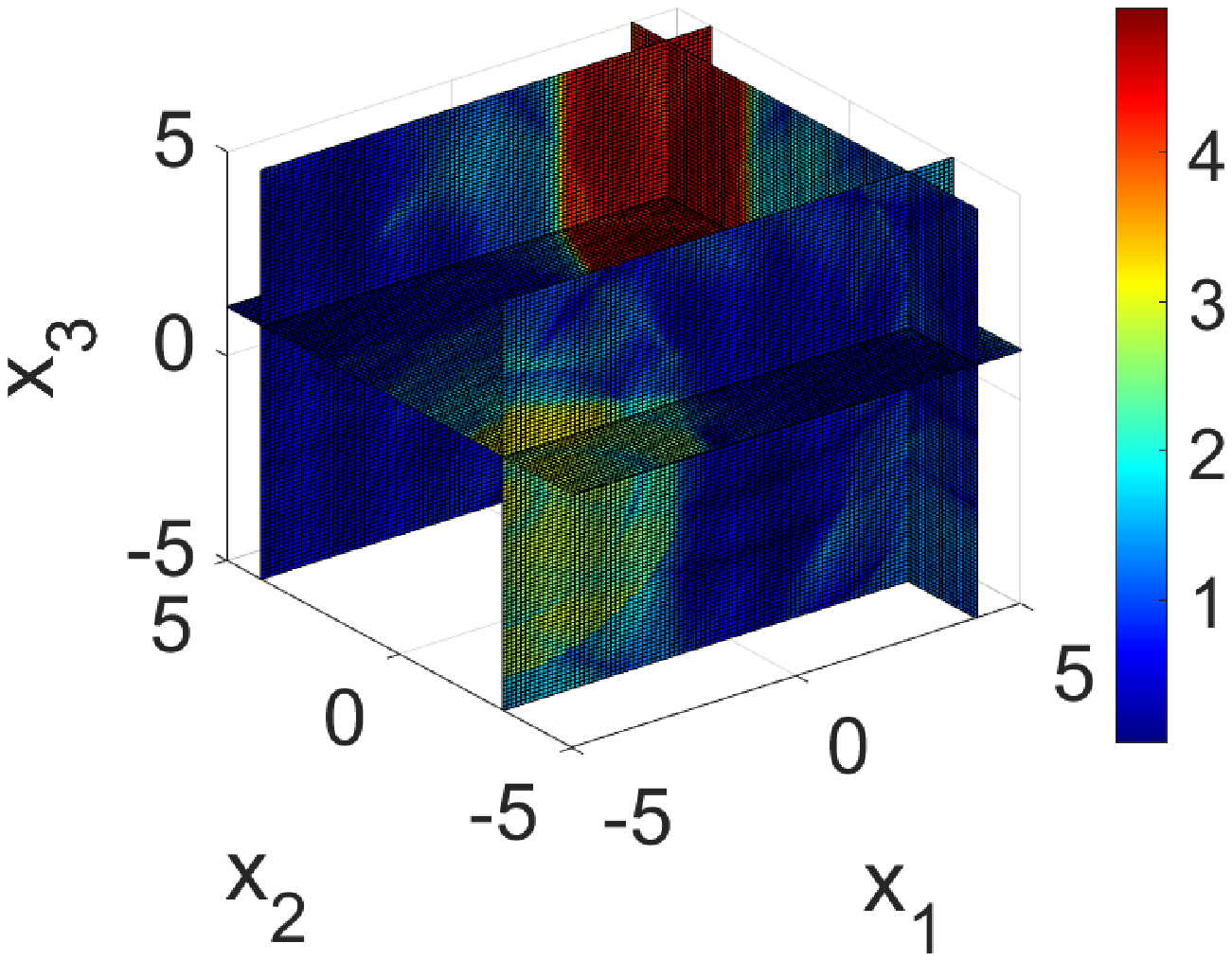}}
\subfloat[]{\includegraphics[width = 1.55in]{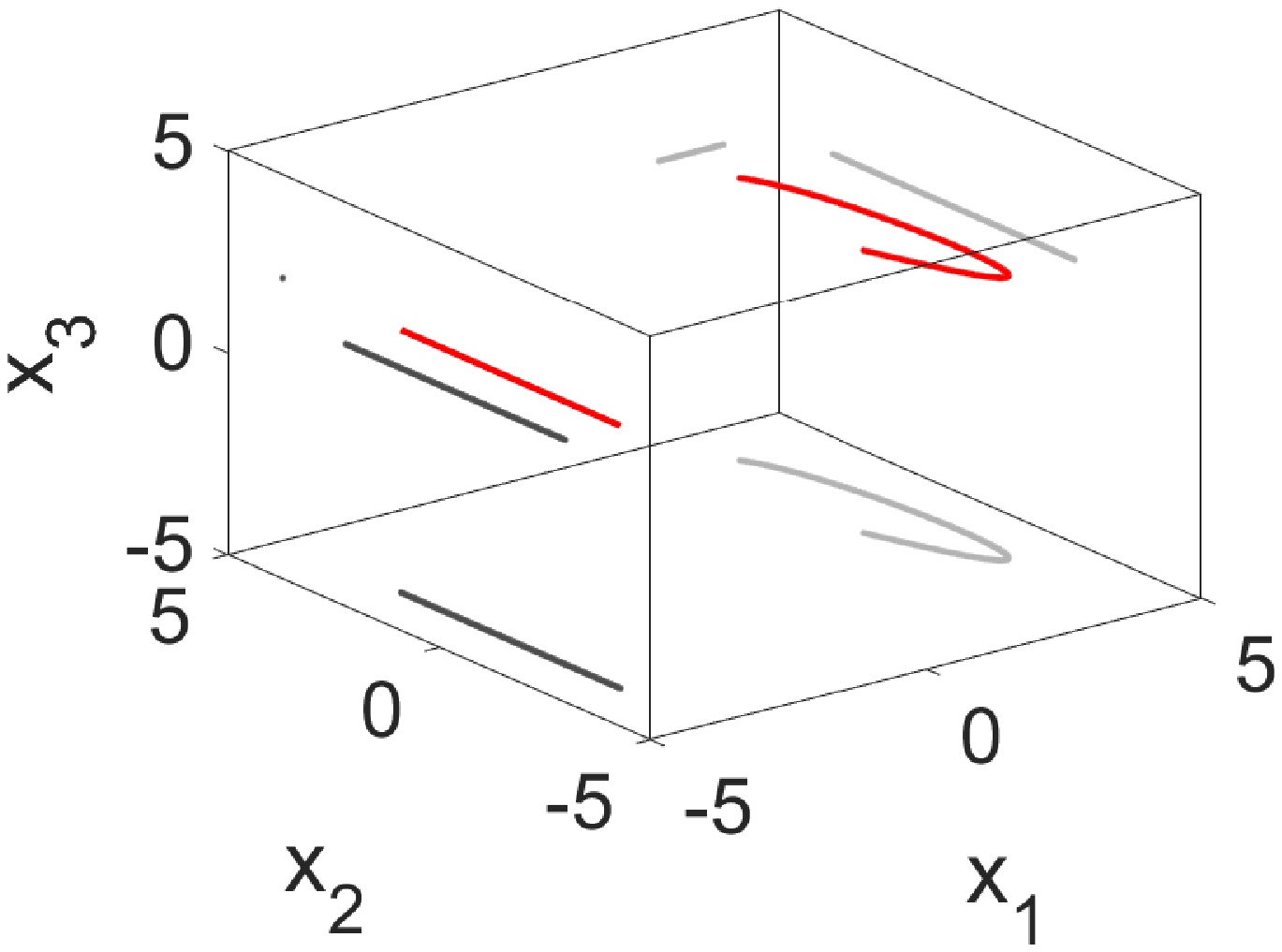}} \\
\subfloat[]{\includegraphics[width = 1.55in]{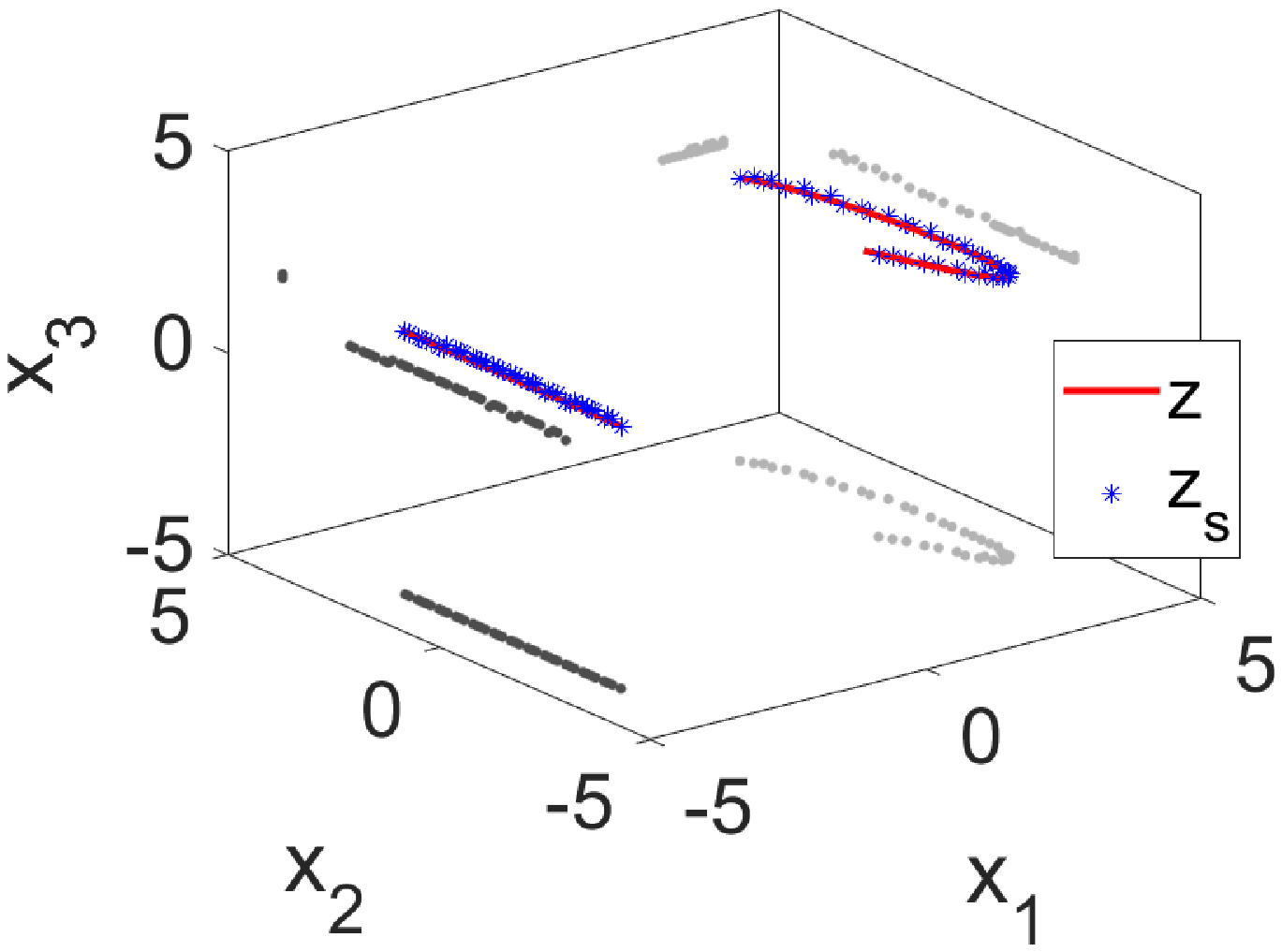}}
\subfloat[]{\includegraphics[width = 1.55in]{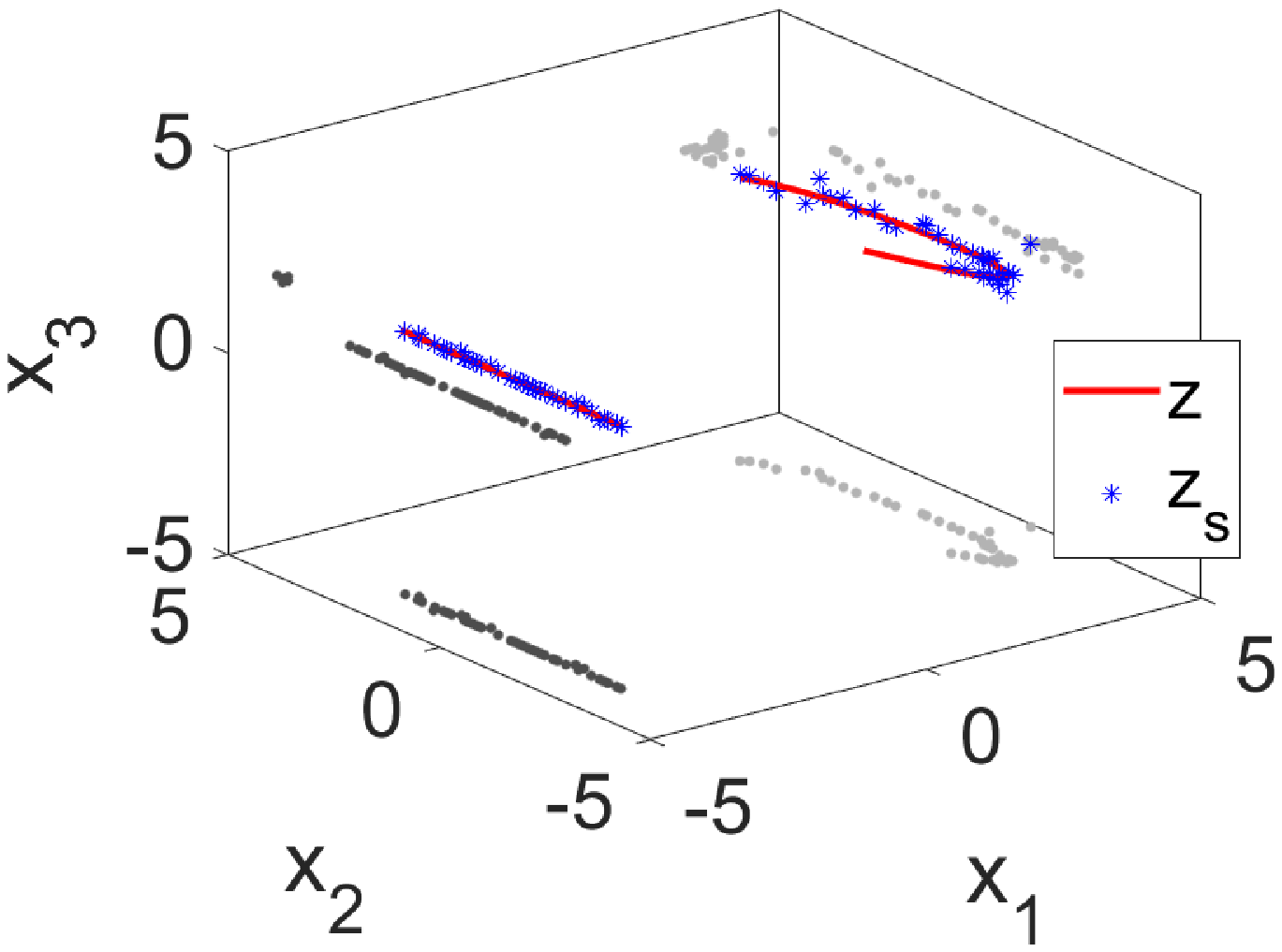}}
\subfloat[]{\includegraphics[width = 1.55in]{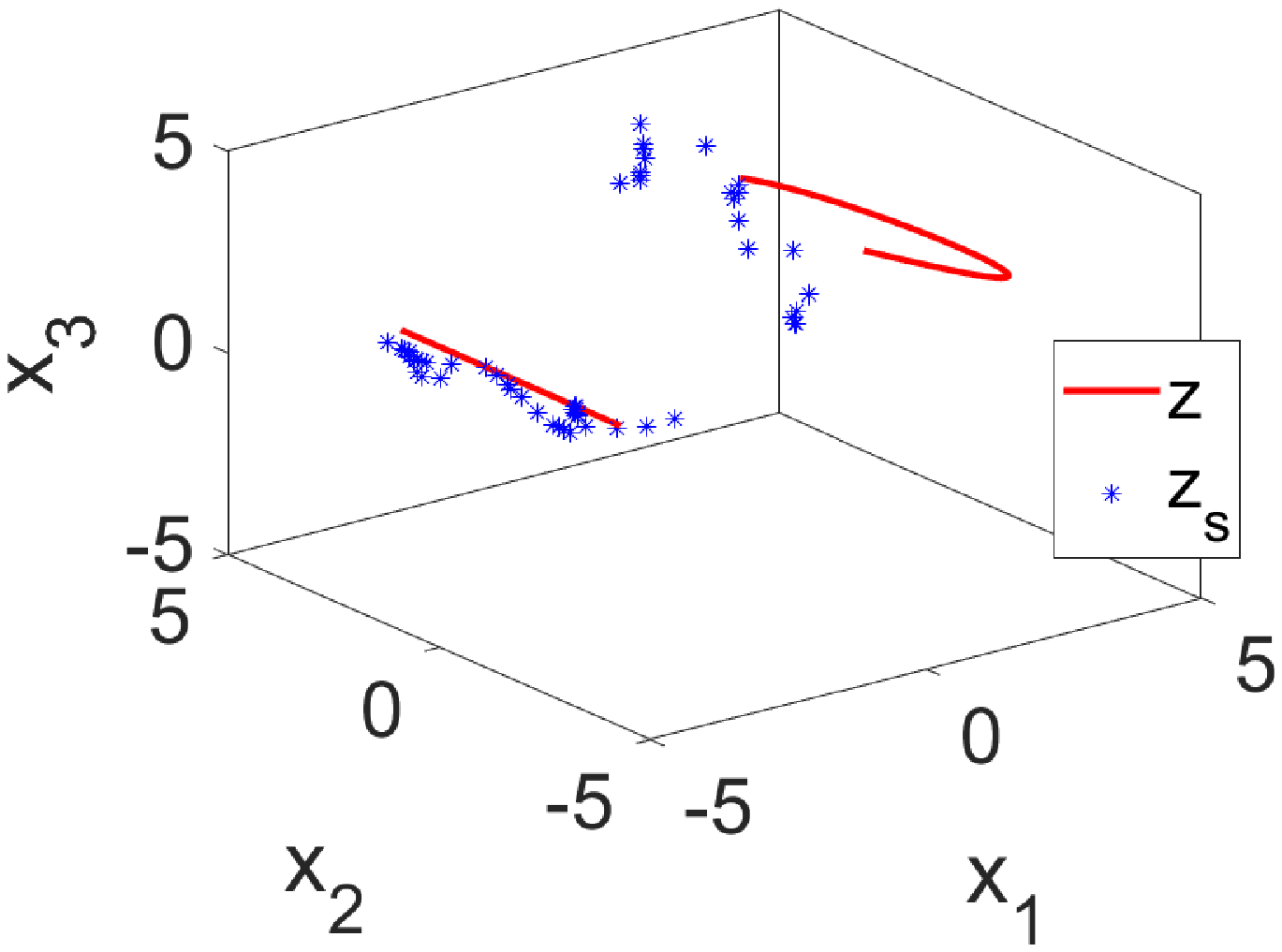}}
\subfloat[]{\includegraphics[width = 1.55in]{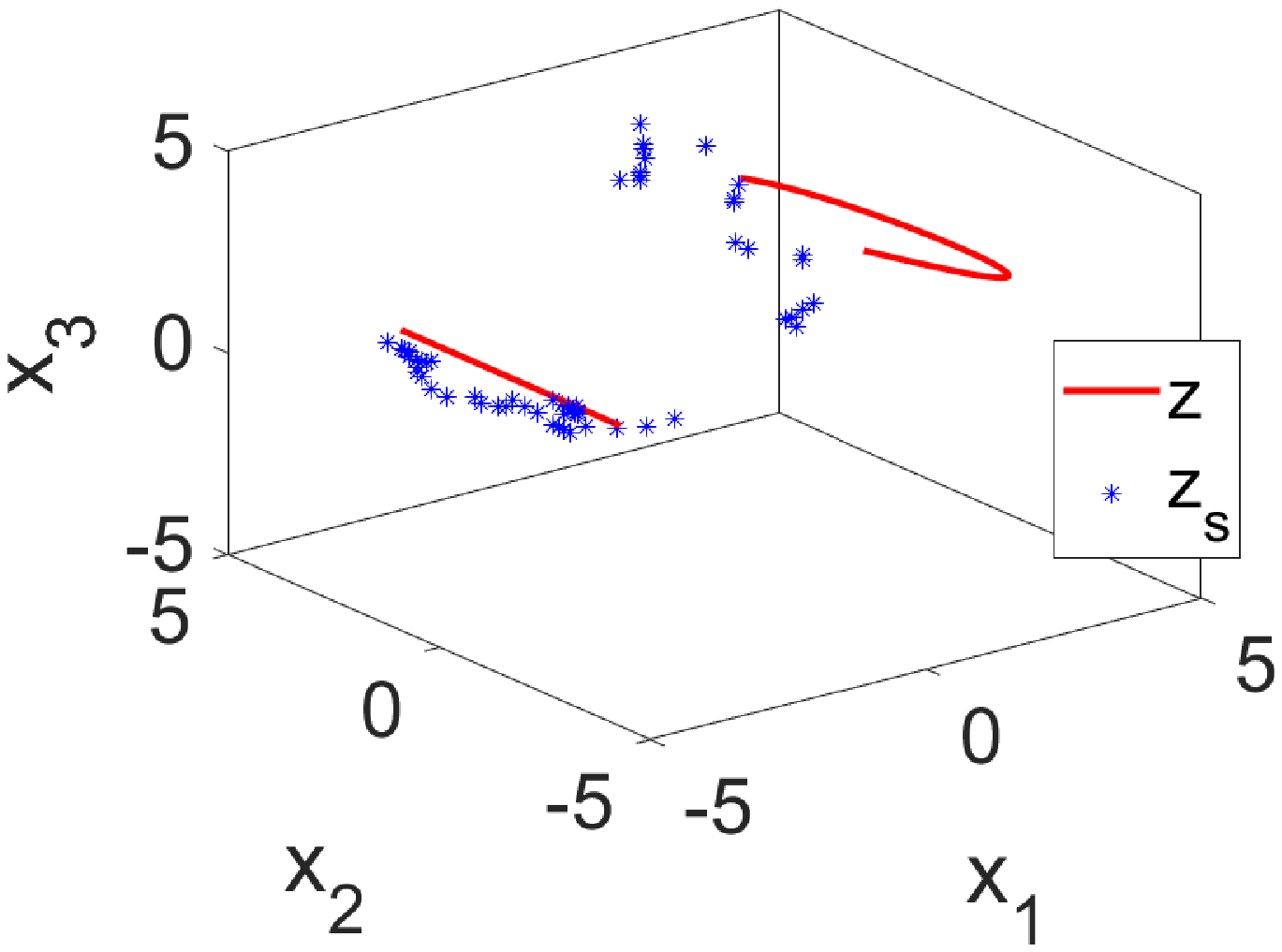}}
\\
\subfloat[]{\includegraphics[width = 1.55in]{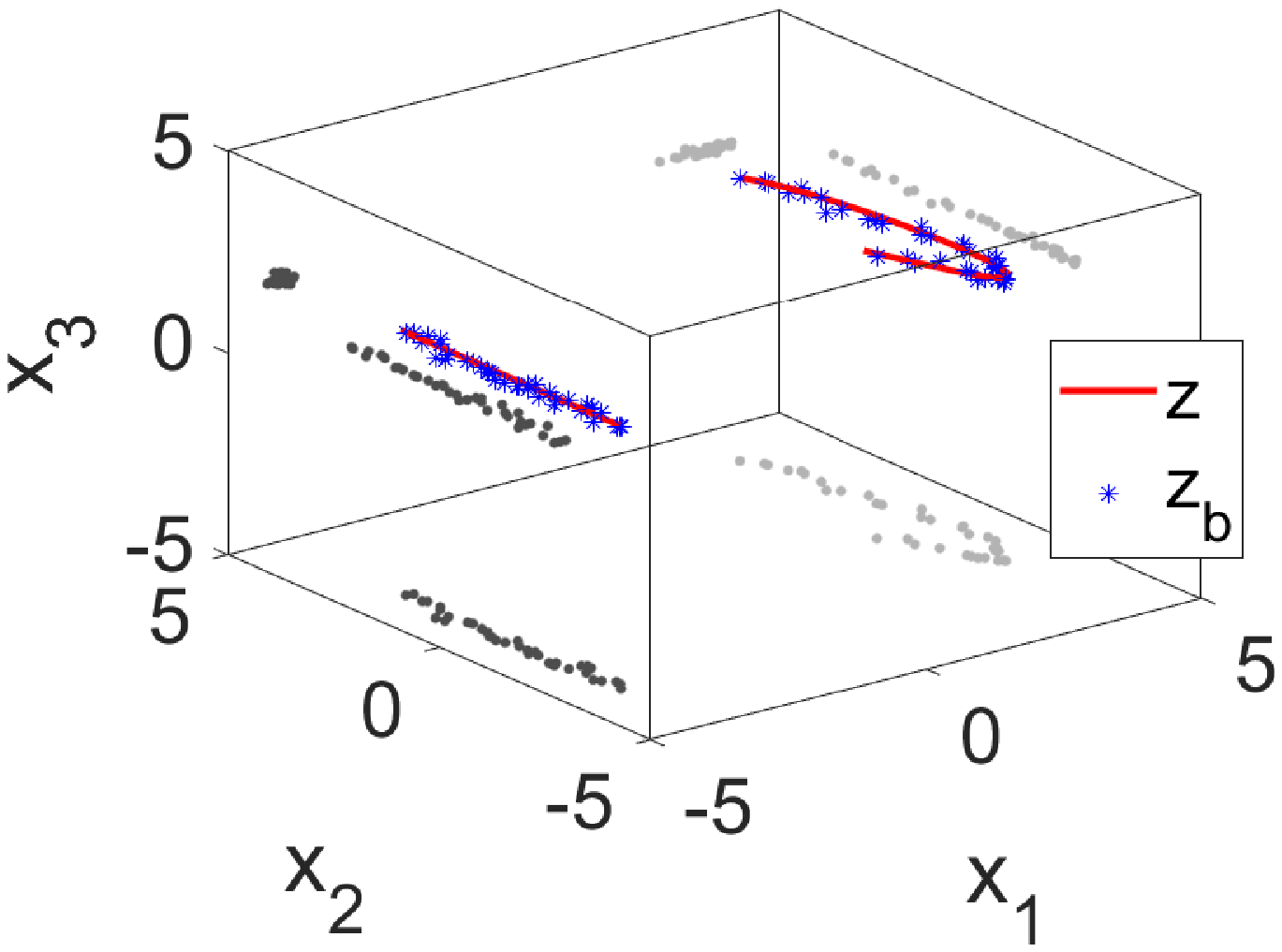}}
\subfloat[]{\includegraphics[width = 1.55in]{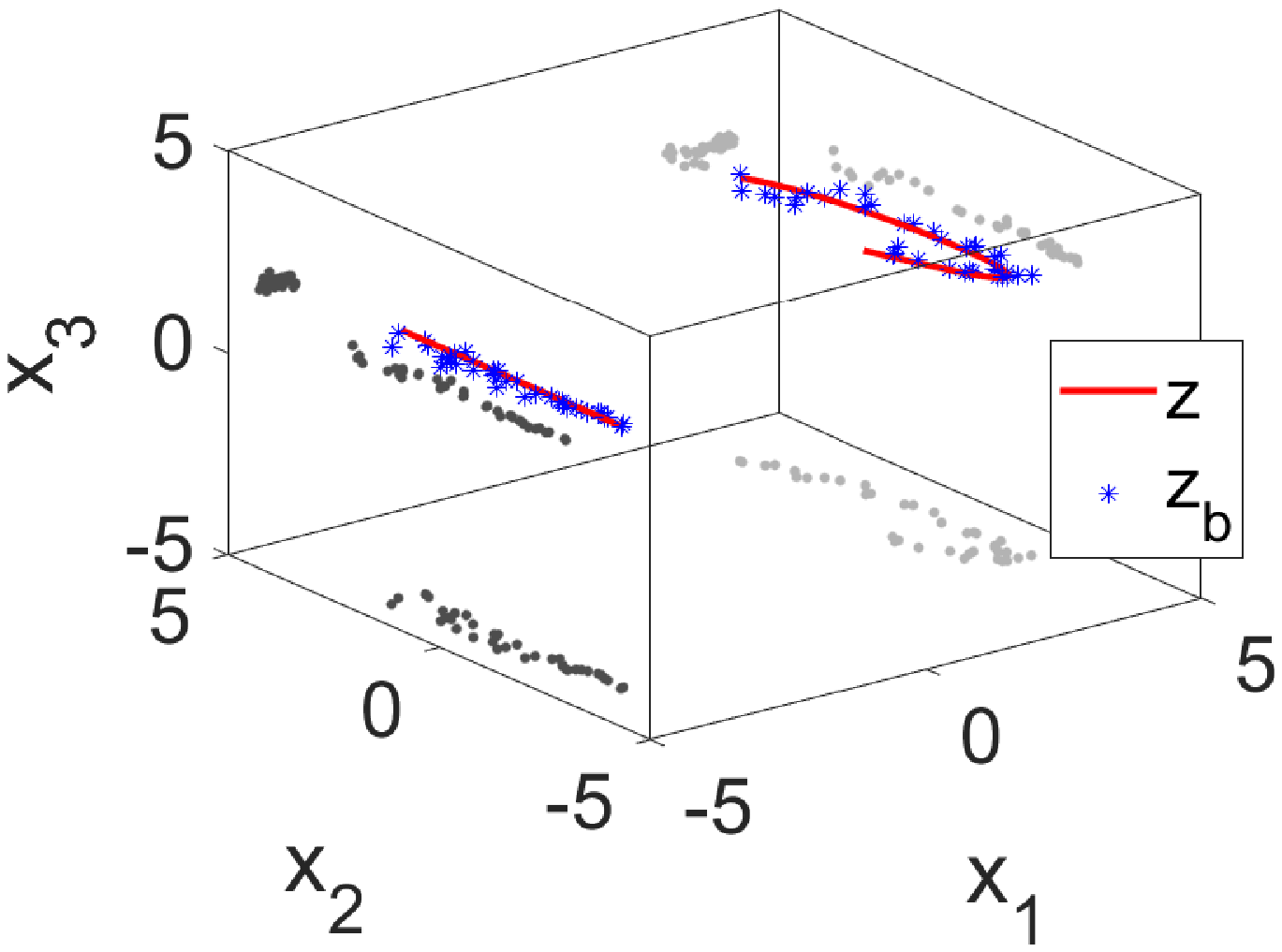}}
\subfloat[]{\includegraphics[width = 1.55in]{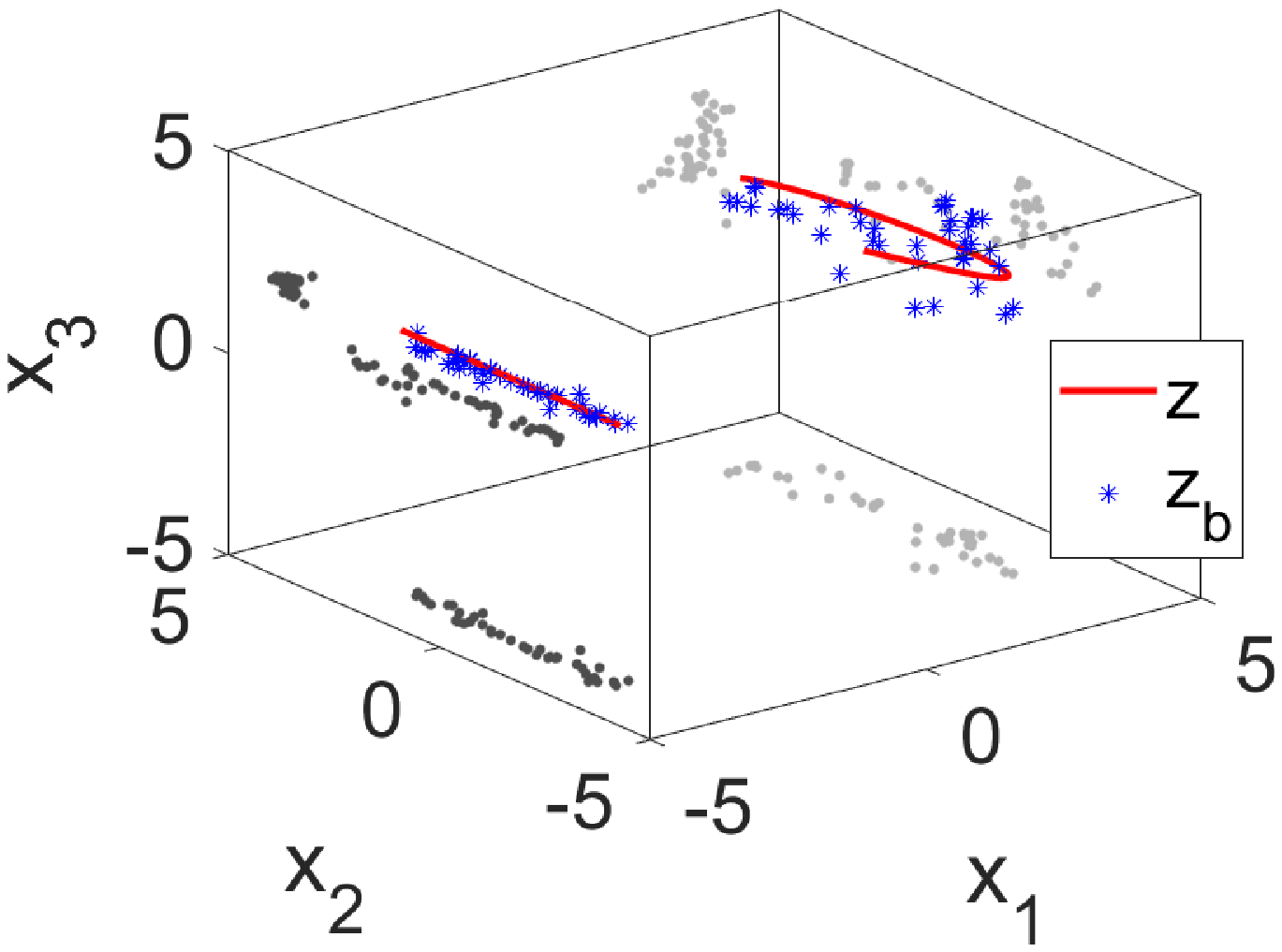}}
\subfloat[]{\includegraphics[width = 1.55in]{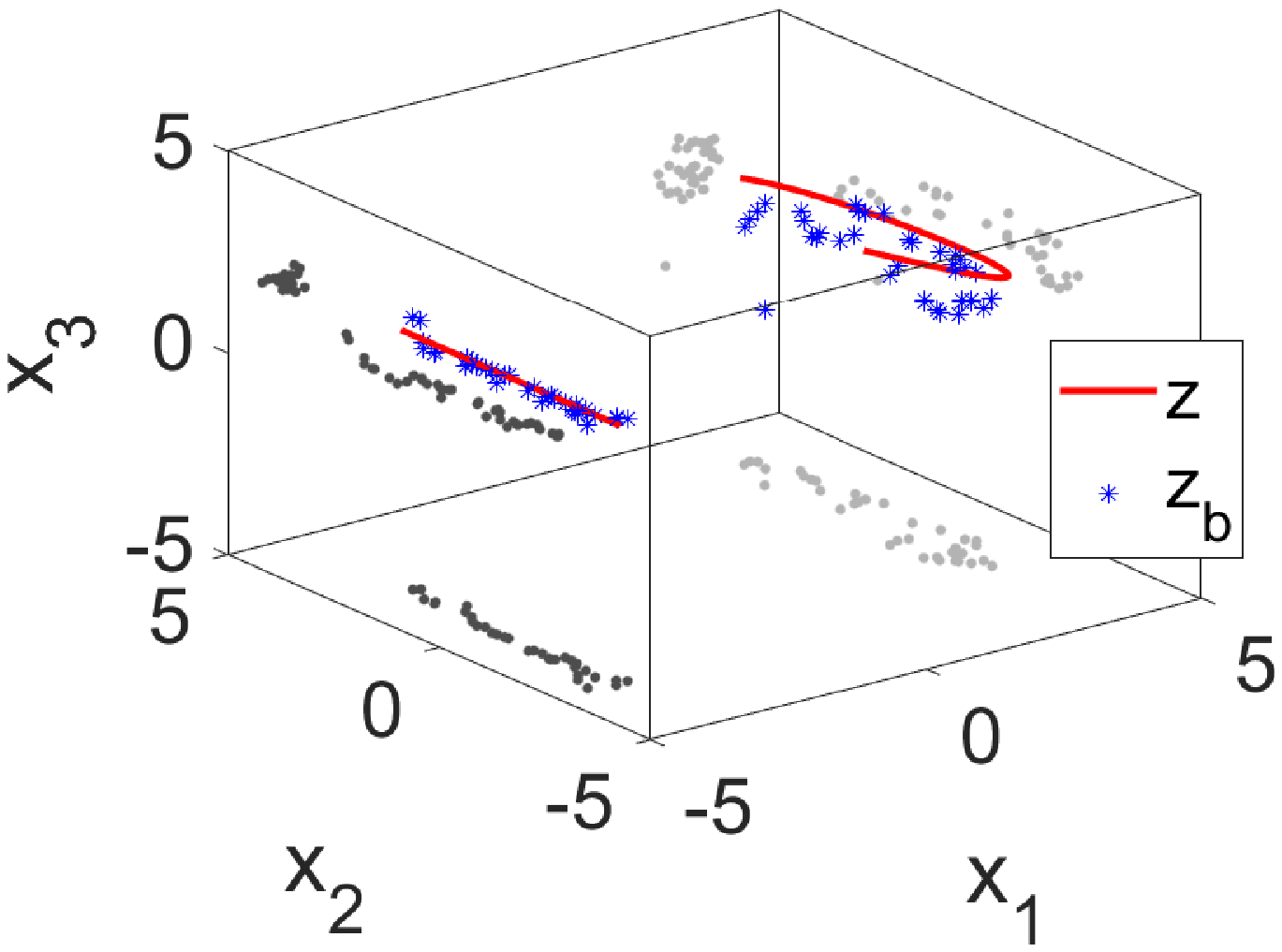}}
\\
\caption{Reconstruction of two paths. (a)-(c) are the cross-section plots of the indicator functions $\bm{I}(y,T_1)$. (d) is the exact trajectory $z(t)$. (e)-(g) are the reconstructions by the {ADSM} using data on $S_1, S_2, S_3$ with $\varepsilon =1\%$. (h) is the reconstruction by the ADSM using data on $S_3$ with $\varepsilon =10\%$. (i)-(k) are the reconstructions by the {ADSM-MCMC} using data on $S_1, S_2, S_3$ with $\varepsilon =1\%$. (l) is the reconstruction by the ADSM-MCMC using data on $S_3$ with $\varepsilon =10\%$. $\mathsf{z}$,$\mathsf{z_{s}}$, and $\mathsf{z_{b}}$ represent the exact paths, the reconstruction by the {ADSM}, and the reconstruction by the {ADSM-MCMC}, respectively.}
\label{fig: Two_curves}
\end{figure}

\section{Conclusions} \label{sec:conclusion}

A deterministic-statistical approach is proposed to reconstruct the moving sources using partial data. The approach contains two steps. The first step is a deterministic method to obtain some qualitative information of the unknowns. The second step is the Bayesian inversion with the prior containing the information obtained in the first step. Both steps are based on the same physical model and use the same set of measured data. Numerical results show that it is a promising technique for tracking the moving point sources using partial data.

The reconstruction by the ADSM deteriorates significantly as the data become fewer. However, it still contains useful information of the sources. Coded in the prior, such information is critical for the success of the MCMC. One can use the uniform prior and perform the Bayesian inverse directly.
In Fig.~\ref{fig:uniform_p}, we show the reconstructed path of the moving source for $S_3$ and $\varepsilon = 1\%$ with $5000$ samples which follow the uniform distribution $\mathcal{U}[-5, 5]$ . It can be seen that the reconstructions in in Fig.~\ref{fig:uniform_p} (c-d) are much worse than Fig.~\ref{fig: C} (i) and (k).

% In many practical applications, due to various restrictions, the measured data are partial and noisy. The related inverse problems are very challenging as pointed out in \cite{Baum1999}: ``We detect everything, but identify nothing''. The combined deterministic-statistical approach is promising to address such problem to some extend.

\begin{figure}[ht]
    \subfloat[]{\includegraphics[width = 1.55in]{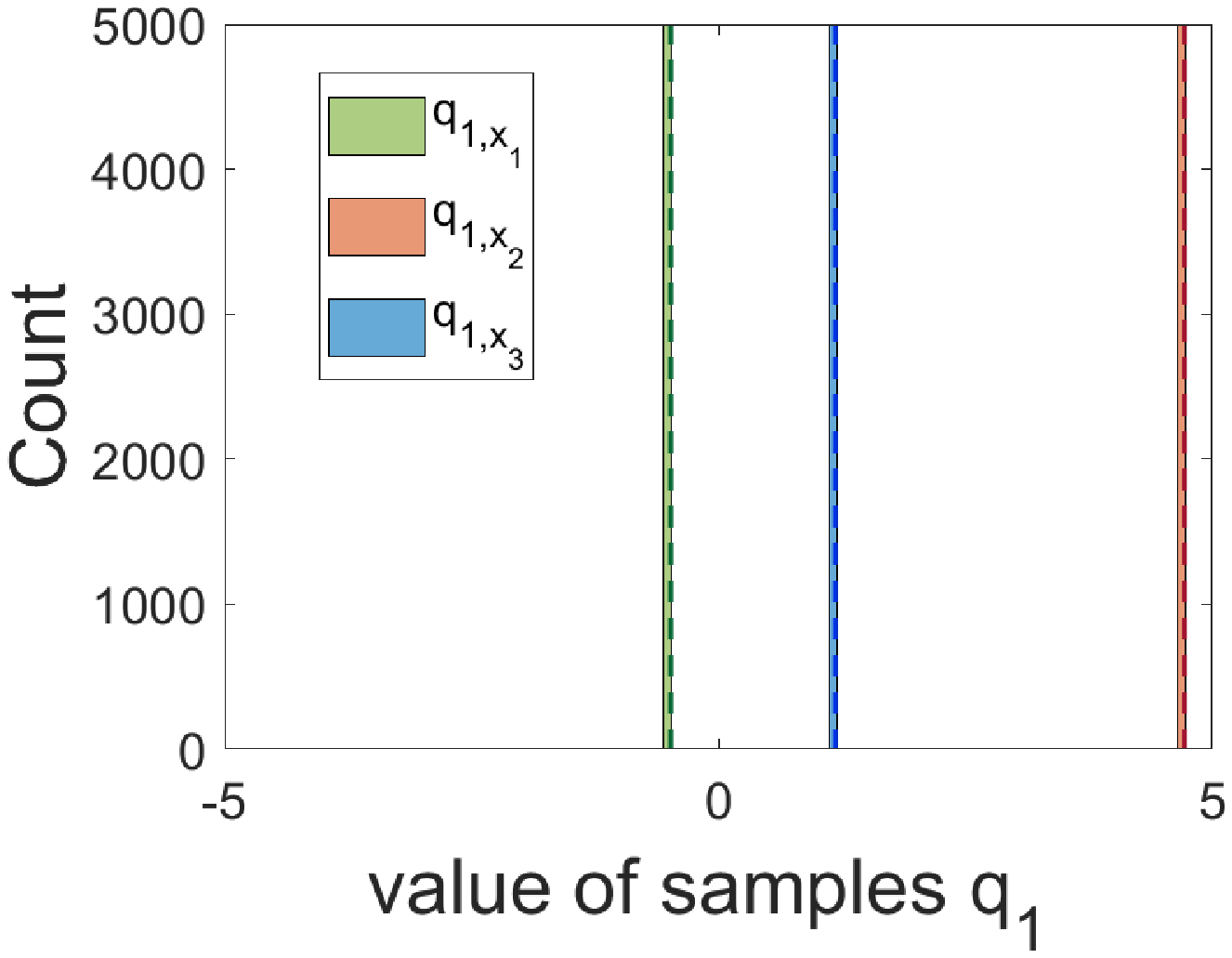}}
\subfloat[]{\includegraphics[width = 1.55in]{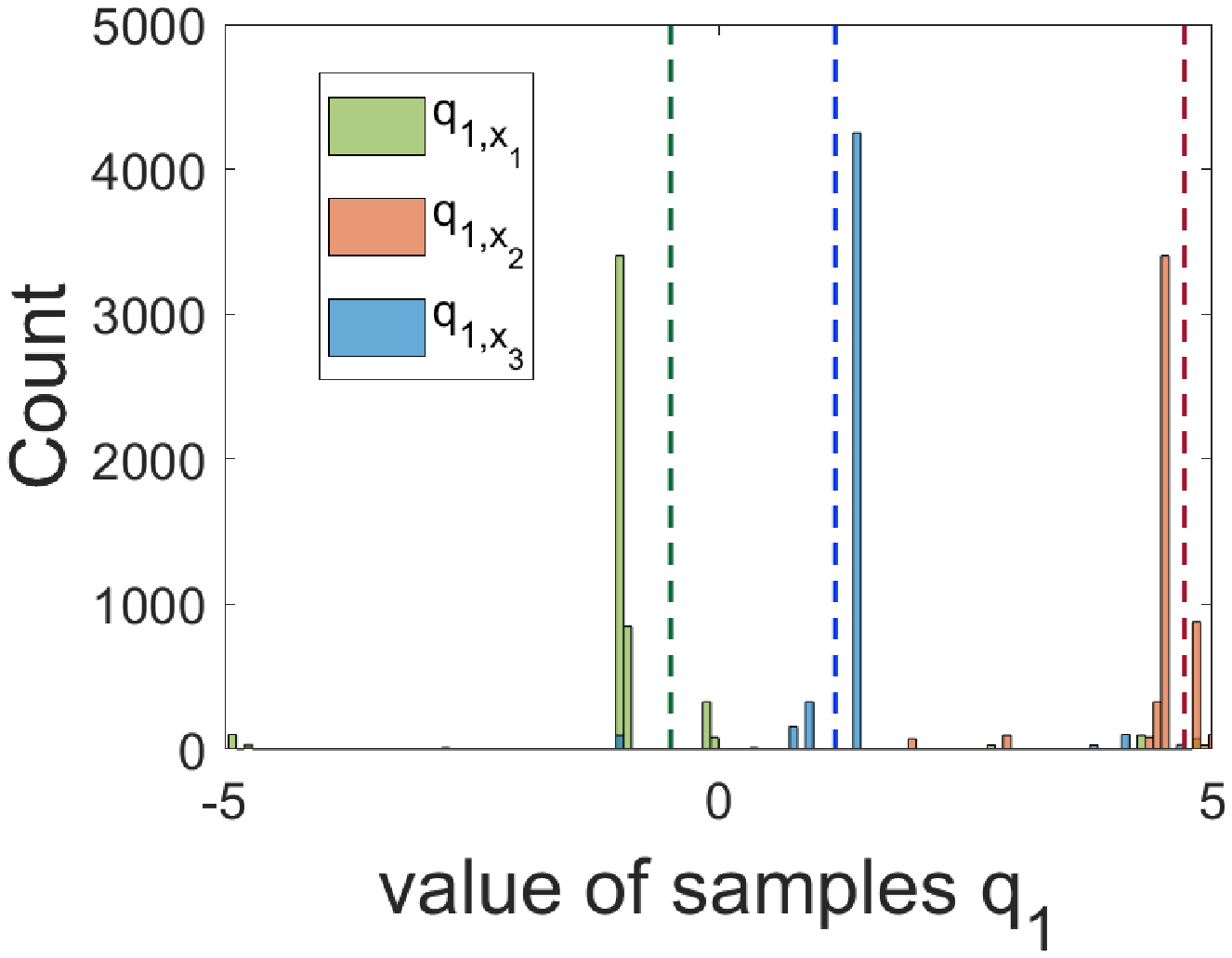}}
\subfloat[]{\includegraphics[width = 1.55in]{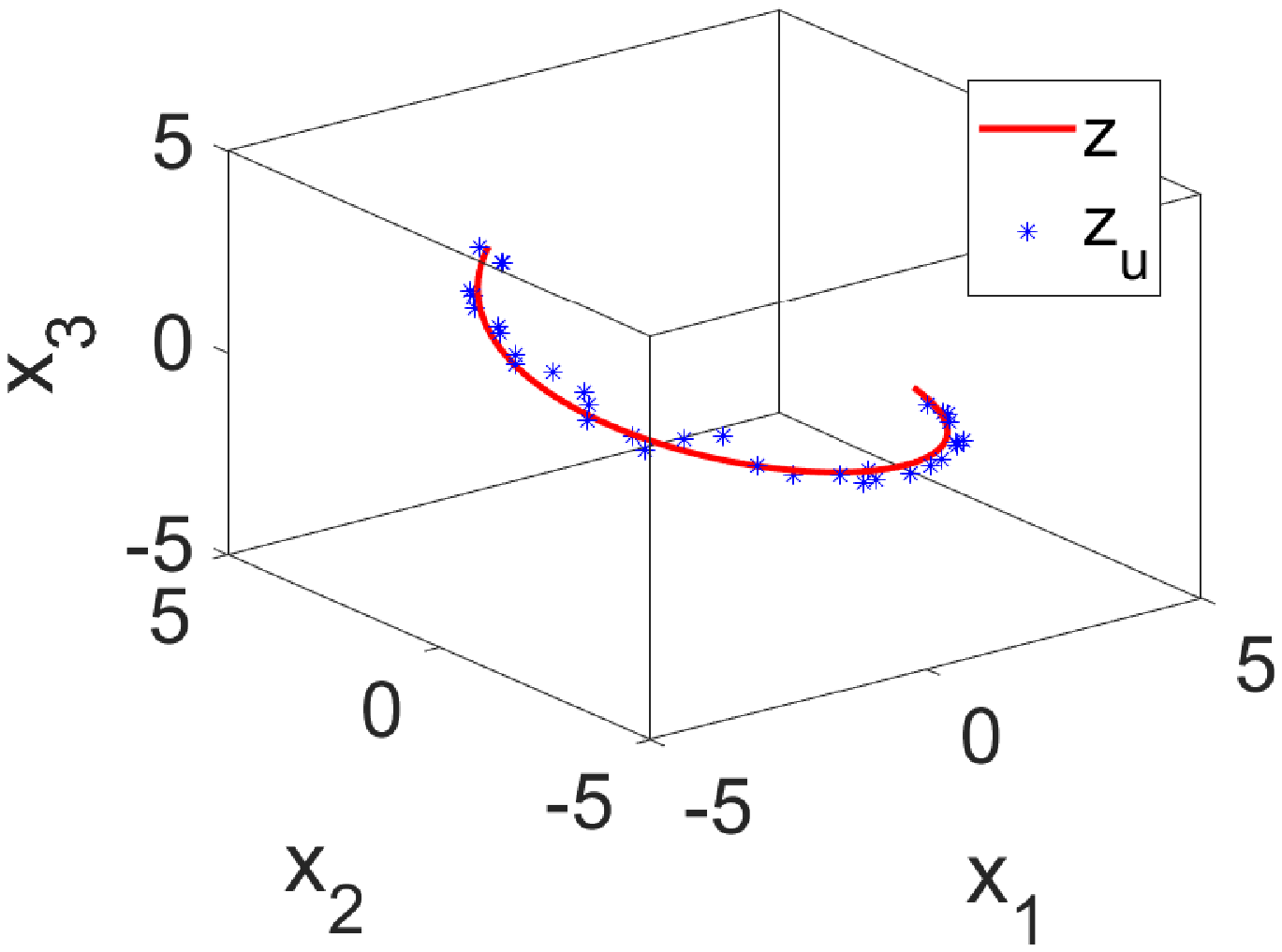}}
\subfloat[]{\includegraphics[width = 1.55in]{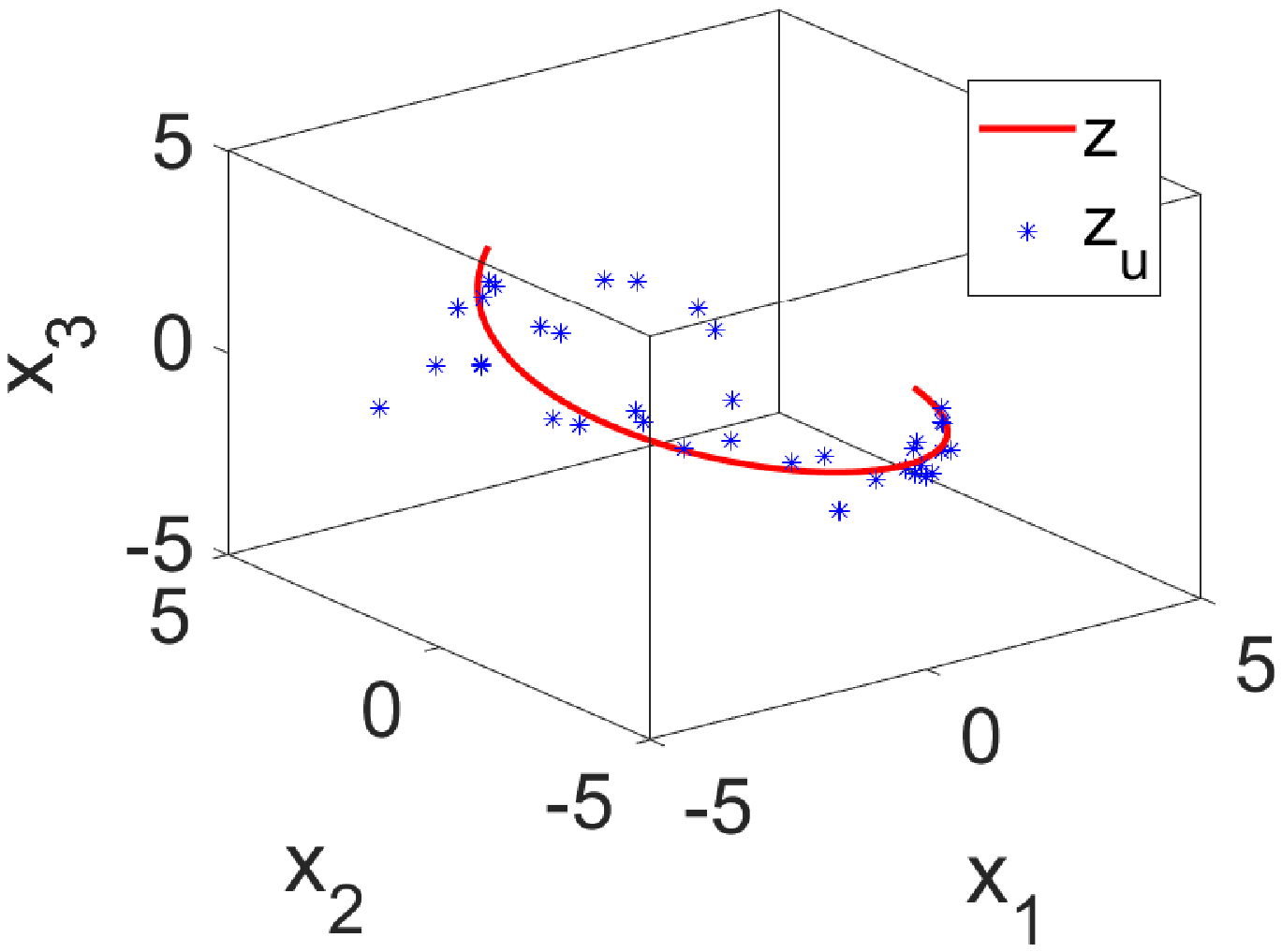}} 
    \caption{Reconstruction of the C-shape path: (a-b) histograms of the samples $\{q_{1}^{(k)}\}_{k=1}^{K}$ with normal and uniform prior for $S_{1}$, respectively; (c-d) reconstructions by MCMC with uniform prior for $S_{1}, S_{3}$, $\varepsilon =1\%$. The dashed line shows the exact value of $q_{1}$ and $\mathsf{z_{u}}$ is path reconstructed by the uniform prior.}
    \label{fig:uniform_p}
\end{figure}

\section*{Acknowledgement}
The second author is supported by NSFC grants 11971133 and 11601107. The collaboration started when the second and third authors attended the 2018 Workshop on Inverse Problems and Related Topics in Chengdu funded in part by NSFC grant 11771068.

\end{document}